\documentclass{amsart}
\usepackage{enumitem, subcaption, color, mathtools, mathrsfs, amsmath,amssymb,amsthm,color,mathrsfs, geometry, nicefrac, dutchcal, array, appendix, hyperref, cleveref}

\usepackage[ruled]{algorithm}
\usepackage{algpseudocode}
\usepackage{physics}
\usepackage{parskip}
\usepackage{lipsum}
\usepackage{tcolorbox}
\usepackage{bbm}
\usepackage[linewidth=1pt, framemethod=tikz]{mdframed}
\usepackage{ulem}
\usepackage{xcolor}

\newgeometry{vmargin={35mm}, hmargin={30mm,30mm}}

\usepackage{hyperref}

\newtheorem{theorem}{Theorem}
\newtheorem{definition}[theorem]{Definition}

\newtheorem{lemma}[theorem]{Lemma}

\newtheorem{remark}[theorem]{Remark}
\newtheorem{Notation}[theorem]{Notation}

\newcommand{\nc}{\newcommand}
\nc{\R}{\mathbb{R}}
\nc{\C}{\mathbb{C}}
\nc{\mrm}{\mathrm}
\nc{\mL}{\mrm{L}}
\nc{\mF}{\mrm{F}}
\nc{\mC}{\mrm{C}}
\nc{\mH}{\mrm{H}}
\nc{\mW}{\mrm{W}}
\nc{\mV}{\mrm{V}}
\nc{\mM}{\mrm{M}}
\nc{\mK}{\mrm{K}}
\nc{\mD}{\mrm{D}}
\nc{\mB}{\mrm{B}}
\nc{\mR}{\mrm{R}}
\nc{\mX}{\mrm{X}}
\nc{\mY}{\mrm{Y}}
\nc{\mS}{\mrm{S}}

\nc{\Ec}{\mrm{E_c}}
\nc{\calL}{\mathcal{L}}
\nc{\loc}{\mrm{loc}}
\nc{\comp}{c}
\nc{\supp}{\mrm{supp}}
\nc{\Hardy}{\mathfrak{H}}
\nc{\calH}{\mathcal{H}}
\nc{\ctru}{\mathfrak{u}}
\nc{\ctrv}{\mathfrak{v}}
\nc{\bc}{\boldsymbol{c}}
\nc{\be}{\boldsymbol{e}}
\nc{\br}{\boldsymbol{r}}
\nc{\bs}{\boldsymbol{s}}
\nc{\bt}{\boldsymbol{t}}
\nc{\bw}{\boldsymbol{w}}
\nc{\bx}{\boldsymbol{x}}
\nc{\by}{\boldsymbol{y}}
\nc{\bz}{\boldsymbol{z}}
\nc{\lbr}{\lbrack}
\nc{\rbr}{\rbrack}
\nc{\dsp}{\displaystyle}
\nc{\vphi}{\varphi}

\begin{document}
	\title[Additive two-level DMRG]{An additive two-level parallel variant of the DMRG algorithm with coarse-space correction}
    \author{Laura Grigori}\address{(L. Grigori) PSI Center for Scientific Computing, Theory, and Data,
5232 Villigen PSI, Switzerland and Institute of Mathematics, Ecole Polytechnique Fédérale de Lausanne (EPFL), 1015 Lausanne, Switzerland.}
	\author{Muhammad Hassan}\address{(M. Hassan) Chair of Analysis, School of Computation, Information and Technology, Technische Universität München. This work was done while the author was affiliated with the PSI Center for Scientific Computing, Theory, and Data, 5232 Villigen PSI, Switzerland, and prior to that, the Institute of Mathematics, Ecole Polytechnique Fédérale de Lausanne, Switzerland. }\email{muhammad.hassan@cit.tum.de}

    \begin{abstract}
  The density matrix renormalization group (DMRG) algorithm is a popular alternating minimization scheme for solving high-dimensional optimization problems in the tensor train format. Classical DMRG, however, is based on sequential minimization, which raises challenges in its implementation on  parallel computing architectures. To overcome this, we propose a novel additive two-level DMRG algorithm that combines independent, local minimization steps with a global update step using a subsequent coarse-space minimization. Our proposed algorithm, which is directly inspired by additive Schwarz methods from the domain decomposition literature, is particularly amenable to implementation on parallel, distributed architectures since both the local minimization steps and the construction of the coarse-space can be performed in parallel. Numerical experiments on strongly correlated molecular systems demonstrate that the method achieves competitive convergence rates while achieving significant parallel speedups.
\end{abstract}

	\subjclass{15A69, 65K10, 65N25, 90C06}
	\keywords{Density matrix renormalization group (DMRG), tensor trains decomposition, high-dimensional optimization, electronic structure calculations, parallel algorithms}
    
	\maketitle

\section{Introduction}\label{sec:1}
The past two decades have seen the emergence of the tensor train format as a powerful tool for solving high-dimensional optimization problems. Originally introduced by Steven White as a tool to solve energy minimization problems in quantum physics \cite{white1992density}, the tensor train (TT) format now has a wide variety of applications in computational science. This includes the compression of weight tensors in deep convolutional neural networks \cite{novikov2015tensorizing, tjandra2017compressing}, the numerical resolution of the so-called chemical master equation that describes chemical reaction dynamics \cite{kazeev2014direct, dolgov2015simultaneous}, low-rank approximation in image segmentation and related computer vision problems \cite{chen2020tensor, bengua2017efficient}, the time-dependent Schrödinger equation in computational quantum mechanics \cite{lubich2015time, haegeman2016unifying}, and distribution sampling in stochastic queuing models \cite{kressner2014low_markov, gelss2017nearest}. Very recently, the so-called quantized TT format has also been used to solve multi-scale PDEs using the variational principle (see, e.g., \cite{khoromskij2011d, dolgov2012fast, kazeev2022quantized}).

The essential idea of using a data-sparse tensor format to approximate high-dimensional optimization problems is certainly not new. A naive implementation of this idea, nevertheless, runs into computational issues since the most basic tensor formats, namely, the canonical polyadic (CP) decomposition and the Tucker decomposition each have their own drawbacks. Indeed, the set of CP-format tensors with rank bounded by $r \in \mathbb{N}$ is not closed, meaning that the best approximation problem in the CP format is ill-posed (see, e.g., \cite[Section 3]{kolda2009tensor}). While the Tucker format does not suffer from this ill-posedness problem (see, e.g., \cite[Section 4]{kolda2009tensor} and \cite{koch2010dynamical}), the storage cost of an order-$d$ tensor in Tucker format scales exponentially in $d$ so that the curse of dimensionality remains an issue in applications where $d$ is very large. As we explain in more detail in Section \ref{sec:2}, the tensor train format offers-- in some sense-- the best of both worlds by combining the favorable approximability properties enjoyed by the Tucker format with a linear scaling (in $d$) storage cost.

While the TT format has seen extensive use in the computational physics literature (under the name matrix product states) since the 2000's, the systematic study of the TT format from a numerical analysis perspective is a relatively recent development. Some of the earliest mathematical works on the tensor train format are the seminal contributions of Oseledets and Tyrtyshnikov \cite{oseledets2009compact, oseledets2009breaking, oseledets2009new, oseledets2011tensor}. About the same time, Hackbusch and Kühn \cite{hackbusch2009new} proposed a hierarchical generalization of the Tucker format. This, so-called  Hierarchical Tucker format (HT) also combined the excellent approximability properties of the Tucker format with a linear scaling (in $d$) storage cost \cite{ballani2013black, grasedyck2011introduction, uschmajew2013geometry}, and it was soon realized that the TT format was simply a powerful special case of the HT format. A fundamental feature of both the TT and the HT format is that they decompose a given higher-order tensor in terms of lower-order tensor components using generalizations of the classical matrix SVD (the TT-SVD algorithm for the former \cite{oseledets2011tensor} and the hierarchical SVD for the latter \cite{grasedyck2010hierarchical}). 

A geometrical study of the the TT format was first initiated by Holtz, Rohwedder and Schneider \cite{holtz2012manifolds} who demonstrated, in particular, that the set of order-$d$ tensors with fixed TT ranks (defined in Section \ref{sec:2} below) is an embedded sub-manifold. Subsequent work by Uschmajew and Vandereycken \cite{uschmajew2013geometry} examined in further detail the geometric structure of this set and obtained similar results for the HT format. An important consequence of this differential geometric structure is that it unlocks the use of so-called Riemannian optimization techniques to solve optimization problems in the HT and TT formats (see, e.g., \cite{steinlechner2016riemannian} as well as \cite{lubich2013dynamical} for time-dependent problems). 

An alternative strategy for solving optimization problems in the tensor train or Hierarchical Tucker format is the so-called alternating minimization algorithm. As we explain in Section \ref{sec:2.2}, these algorithms entirely ignore the manifold structure of the TT format and the HT format and rely instead on sequential minimization of the lower-order tensor components that form the TT or HT decomposition. The earliest examples of such an alternating scheme are the famous one-site and two-site density matrix renormalisation group (DMRG) algorithm \cite{white1992density} proposed in the quantum physics literature for solving Rayleigh-quotient minimization problems. In the mathematics literature, the so-called alternating least squares (ALS) and modified alternating least squares (MALS) algorithms \cite{holtz2012alternating} have been proposed as generalizations of DMRG for solving optimization problems beyond quotient minimization. Additional variants of DMRG include the alternating minimal energy (AMEn) method of Dolgov and Savostyanov \cite{dolgov2014alternating} and the extension EVAMEn proposed by Kressner, Steinlechner and Uschmajew \cite{kressner2014low}.


An essential feature of the alternating minimization algorithms listed above is the sequential nature of the lower-order tensor component minimization. Indeed, it is precisely this sequential nature of the local minimization that allows the interpretation of the ALS algorithm as a non-linear Gauss-Seidel relaxation. Thus, in the case of the TT decomposition where an order-$d$ tensor is decomposed in terms of $d$, order-three tensors, we must perform $d$ (or $d-1$ in the case of MALS and two-site DMRG) local minimization and update steps, \emph{one-after-the-other} in order to complete a global update. Given the difficulty of implementing sequential algorithms of this nature on parallel computing architectures, it is natural to ask if one can develop a non-sequential variant of DMRG in which the local minimization steps are performed independently of each other in parallel, and a global update is obtained by combining all these local updates in a single step. Note that this strategy can be seen as a two-step domain decomposition method applied to the tensor train (or HT) format.

The aim of the present contribution is to develop this line of reasoning by proposing an additive two-level DMRG algorithm that involves independent, local minimization steps, enriched by a second level coarse space correction. Our proposed algorithm, which is directly inspired by additive Schwarz methods from the domain decomposition literature (see {Remark \ref{rem:dd} below}), is well-suited for implementation on distributed architectures since both the local minimization steps and the coarse space construction can be performed in parallel. To our knowledge, the present proposal is the first attempt at a non-sequential variant of DMRG with a coarse-space minimization step; the only prior proposal for non-sequential DMRG that we are familiar with is due to Stoudenmire and White \cite{stoudenmire2013real}, which does not contain coarse-space correction. 

The remainder of this article is organized as follows. In the upcoming Sections \ref{sec:2.1} and \ref{sec:2.2}, we introduce the tensor train format and the classical one-site and two-site DMRG algorithms respectively. In Section \ref{sec:3}, we state our proposed additive two-level DMRG algorithm with coarse space correction, and in Section \ref{sec:3.1}, we describe, in detail, the coarse-step minimization step. Finally, as a first test of the performance of our proposed algorithm, we compute the ground-state (lowest) energy levels of some \textit{strongly-correlated} molecular systems from quantum chemistry.  {Strongly correlated molecules possess ground state eigenfunctions that are poorly approximated using a single Slater determinants \cite{ganoe2024notion} (as in Hartree-Fock or DFT approaches), and are therefore prime candidates for the application of DMRG.} These numerical experiments are the subject of Section \ref{sec:4}. Further computational aspects of the algorithm including a cost comparison with classical DMRG are given in the supplementary material. 

\section{Problem Formulation and Setting}\label{sec:2}

\subsection{Basic Framework}\label{sec:2.1}

Let $d \in \mathbb{N}$ and $\{n_j\}_{j=1}^d \subset \mathbb{N}^{d}$ be given. We denote by $\mathbb{R}^{n_1 \times n_2 \times \ldots n_d}$, the space of order $d$ tensors, and by $U \subset  \mathbb{R}^{n_1 \times n_2 \times \ldots n_d}$ some open set. Following standard convention, elements of the tensor space $\mathbb{R}^{n_1 \times n_2 \times \ldots n_d}$ will be denoted as $\bold{X}, \bold{Y}$, or $\bold{Z}$, etc, and we will assume, throughout this article, that all Euclidean spaces including $\mathbb{R}^{n_1 \times n_2 \times \ldots n_d}$ are equipped with the Frobenius inner product $\langle \cdot, \cdot \rangle$ and corresponding Frobenius norm $\Vert \cdot \Vert$. 

{We assume that we are given a twice continuously differentiable energy functional $\mathcal{J} \colon U \subset \mathbb{R}^{n_1 \times n_2 \times \ldots n_d} \rightarrow \mathbb{R}$ that possesses a (global) minimum in $U$,} and we are interested in the solution to the minimization problem
\begin{align}\label{eq:1}
	{E}^*:= \underset{\bold{X} \in U}{\min}\; \mathcal{J}(\bold{X}).
\end{align}
In the present study, we will particularly focus on the case when $U=\mathbb{R}^{n_1 \times n_2 \times \ldots n_d}\setminus \{0\}$ and the energy functional $\mathcal{J}\colon U \rightarrow \mathbb{R}$ takes the form
\begin{align}\label{eq:energy_eig}
	\forall \bold{X}\in  U=\mathbb{R}^{n_1 \times n_2 \times \ldots n_d}\setminus \{0\}\colon \qquad        \mathcal{J}(\bold{X})= \frac{\langle \bold{X}, \bold{A}\bold{X}\rangle}{\Vert \bold{X}\Vert^2},
\end{align}
where $\bold{A}\colon \mathbb{R}^{n_1 \times n_2 \times \ldots n_d} \rightarrow \mathbb{R}^{n_1 \times n_2 \times \ldots n_d}$ is a symmetric operator with a simple lowest eigenvalue. It is well-known that minimizers of the energy functional \eqref{eq:energy_eig} correspond to the lowest eigenfunction of the operator $\bold{A}$. Our motivation for focusing on an energy functional of this specific form therefore stems from our interest in solving high-dimensional eigenvalue problems, particularly those arising from quantum chemistry. 

Let us nevertheless point out that our proposals can easily be adapted to the case when $U=\mathbb{R}^{n_1 \times n_2 \times \ldots n_d}$ and the energy functional $\mathcal{J}\colon U \rightarrow \mathbb{R}$ takes the form
\begin{align}\label{eq:energy_lin}
	\forall \bold{X}\in  U=\mathbb{R}^{n_1 \times n_2 \times \ldots n_d}\colon \qquad        \mathcal{J}(\bold{X})= \frac{1}{2} \langle \bold{X}, \bold{A}\bold{X}\rangle- \langle \bold{F}, \bold{X}\rangle,
\end{align}
where $\bold{A}\colon \mathbb{R}^{n_1 \times n_2 \times \ldots n_d} \rightarrow \mathbb{R}^{n_1 \times n_2 \times \ldots n_d}$ is a symmetric, positive definite operator and $\bold{F} \in \mathbb{R}^{n_1 \times n_2 \times \ldots n_d}$. In this case, the minimizer of the energy functional \eqref{eq:energy_lin} corresponds to the solution of the linear operator equation $\bold{A}\bold{X}=\bold{F}$, such equations arising naturally when discretizing high-dimensional, symmetric, elliptic source problems.

In many applications-- in particular those arising from quantum chemistry-- the order $d \in \mathbb{N}$ of the tensor space is extremely large which results in the minimization problem \eqref{eq:1} being computationally intractable. A natural strategy to deal with this curse of dimensionality is to introduce a low-rank subset of $\mathbb{R}^{n_1 \times n_2 \times \ldots n_d}$ and resolve the minimization problem \eqref{eq:1} on the intersection of this low-rank subset with $U$. The tensor train format (or more generally the Hierarchical Tucker format \cite{hackbusch2009new}) allows for the possibility of defining such a low-rank subset. To explain this construction, we next introduce the notion of the separation ranks of a tensor introduced in \cite{holtz2012manifolds}.

\begin{Notation}[Canonical Unfoldings of a Tensor]
	Let $\{n_j\}_{j=1}^k \subset \mathbb{N}$ and let $\bold{X} \in \mathbb{R}^{n_1 \times n_2 \times \ldots n_k}$ be a tensor of order $k$. For any $j \in \{1, \ldots, k-1\}$, we write $\bold{X}^{<j>} \in \mathbb{R}^{n_1\ldots n_j \times n_{j+1}\ldots n_{k}}$ to denote the $j^{\rm th}$ canonical unfolding of $\bold{U}$. Note that this is a matricization of $\bold{X}$ with $\prod_{\ell=1}^j n_j$ rows and $\prod_{\ell=j+1}^k n_j$ columns. By convention, $\bold{X}^{<d>} \in \mathbb{R}^{n_1\ldots n_d \times 1}$ denotes the vectorization of $\bold{U}$.
\end{Notation}

\begin{definition}[Separation rank of a tensor]\label{def:sep_ranks}~
	
	Let $\{n_j\}_{j=1}^d \subset \mathbb{N}$, let $\bold{X} \in \mathbb{R}^{n_1 \times n_2 \times \ldots n_d}$ be a tensor of order $d \in \mathbb{N}$ and for each $j \in \{1, \ldots, d\}$ let $\bold{X}^{<j>}\in \mathbb{R}^{n_1\ldots n_j \times n_{j+1}\ldots n_d}$ denote the $j^{\rm th}$ canonical unfolding of $\bold{X}$. Then we define the $j^{\rm th}$ separation rank of $\bold{X}$, denoted $\text{\rm rank}_{{\rm sep}}(\bold{X})\in \mathbb{N}$ as
	\begin{align}
		\text{\rm rank}_{{\rm sep}, j}(\bold{X}):= \text{\rm rank}\big(\bold{X}^{<j>}\big).
	\end{align}
	Moreover, we define the separation rank of $\bold{X}$, denoted $\text{\rm rank}_{{\rm sep}}(\bold{X}) \in \mathbb{N}^{d-1}$, as 
	\begin{align}
		\text{\rm rank}_{{\rm sep}}(\bold{X}):= \Big( \text{\rm rank}_{{\rm sep}, 1}(\bold{X}),  \text{\rm rank}_{{\rm sep}, 2}(\bold{X}), \ldots,  \text{\rm rank}_{{\rm sep}, d-1}(\bold{X}) \Big).
	\end{align}
\end{definition}

Next, we introduce the following low-rank subsets of $\mathbb{R}^{n_1 \times n_2 \times \ldots n_d}$.

\begin{definition}[Manifold of Tensors with Fixed Separation Rank]\label{def:TT_manifolds}~
	
	Let $\bold{r}= (r_1, \ldots r_{d-1}) \subset \mathbb{N}^d$ be a $(d-1)$-tuple of natural numbers and let $\{n_j\}_{j=1}^d \subset \mathbb{N}$. We define the sets $\mathcal{T}_{\leq \bold{r}} \subset \mathbb{R}^{n_1 \times n_2 \times \ldots n_d}$ and $\mathcal{T}_{\bold{r}} \subset \mathbb{R}^{n_1 \times n_2 \times \ldots n_d}$ as
	\begin{align*}
		\mathcal{T}_{\leq \bold{r}} :=& \Big \{ \bold{X} \in \mathbb{R}^{n_1 \times n_2 \times \ldots n_d} \colon \qquad \text{\rm rank}_{{\rm sep}, j}(\bold{X}) \leq r_j \quad \forall j \in \{1, \ldots, d\} \Big\}, \quad \text{and}\\
		\mathcal{T}_{\bold{r}} :=& \Big \{ \bold{X} \in \mathbb{R}^{n_1 \times n_2 \times \ldots n_d} \colon \qquad \text{\rm rank}_{{\rm sep}, j}(\bold{X}) = r_j \quad \forall j \in \{1, \ldots, d\} \Big\}.
	\end{align*}
	We refer to $\mathcal{T}_{\bold{r}}$ as the manifold of tensors with fixed separation rank $\bold{r} \in \mathbb{N}^{d-1}$.
\end{definition}

Consider Definition \ref{def:TT_manifolds} of the sets $\mathcal{T}_{\leq \bold{r}}$ and $\mathcal{T}_{\bold{r}}$. As a first remark, let us point out that, depending on the choice of separation ranks $\bold{r} \in \mathbb{N}^{d-1}$, the set $\mathcal{T}_{\bold{r}}$ may, in fact, be empty (consider, e.g., $\bold{r}$ with $r_1 > \min\{n_1, n_2\ldots n_d\}$). In the non-trivial case when $\mathcal{T}_{\leq \bold{r}}$ and $\mathcal{T}_{\bold{r}}$ are both non-empty, the structure of these sets has been studied in-depth by Holtz, Rohwedder and Schneider \cite{holtz2012manifolds} and further by Uschmajew and Vandereycken \cite{uschmajew2013geometry}. In particular, it has been shown that $\mathcal{T}_{\bold{r}}$ is an embedded sub-manifold of $\mathbb{R}^{n_1 \times n_2 \times \ldots n_d}$ (hence our chosen nomenclature).


Returning now to the minimization problem \eqref{eq:1}, we observe that we can approximate the sought minimum $\mathcal{E}^*$ by considering, for a judicious choice of ranks $\bold{r}\in \mathbb{N}^{d-1}$, the alternative minimization problem 
\begin{align}\label{eq:2}
	E^*_{\bold{r}}:=\underset{\bold{X} \in  \mathcal{T}_{\bold{r}} \cap U}{\min}\; \mathcal{J}(\bold{X}).
\end{align}

Note that elements of $\mathcal{T}_{\bold{r}}$ are still order $d$ tensors, and it is, therefore, not clear that the minimization problem \eqref{eq:2} is computationally more tractable than the original minimization problem \eqref{eq:1}. It turns out, however, that elements of the manifold $\mathcal{T}_{\bold{r}}$ have a data-sparse representation in terms of the so-called tensor train decomposition. To see this, we next introduce the so-called tensor train parameter space.

\begin{definition}[Tensor Train Parameter Space]\label{def:parameter_space}~
	
	Let $\bold{r}:= (r_1, \ldots, r_{d-1}) \subset \mathbb{N}^{d-1}$, let $r_0=r_d=1$, and let $\{n_j\}_{j=1}^d \subset \mathbb{N}$. We define the tensor train parameter space $\overline{\mathcal{U}}_{\bold{r}}$ as the Cartesian product space
	\begin{align*}
		\overline{\mathcal{U}}_{\bold{r}}:= \bigtimes_{j=1}^d \mathbb{R}^{r_{j-1} \times n_j \times r_{j}},
	\end{align*}
	and we define the rank-constrained parameter set $\mathcal{U}_{\bold{r}} \subset \overline{\mathcal{U}}_{\bold{r}}$ as 
	\begin{align}
		\mathcal{U}_{\bold{r}}:= \Big \{ \big(\bold{U}_1, \ldots, \bold{U}_d\big) \in \overline{\mathcal{U}}_{\bold{r}} \colon \quad \text{\rm rank}(\bold{U}_j^{<2>})=\text{\rm rank}(\bold{U}_{j+1}^{<1>})=r_j \quad \forall j =1, \ldots, d-1\Big\}.
	\end{align}
\end{definition}

Consider Definition \ref{def:parameter_space} of the tensor train parameter space $\overline{\mathcal{U}}_{\bold{r}}$ and subset $\mathcal{U}_{\bold{r}}$. Three remarks are in order. First, as hinted at in this definition, elements of $\overline{\mathcal{U}}$ will typically be denoted by $\bold{U}, \bold{V}$, or $\bold{W}$. Second, given $\bold{U}=(\bold{U}_1, \bold{U}_2, \ldots \bold{U}_d) \in \overline{\mathcal{U}}_{\bold{r}}$, we notice that each $\bold{U}_j$ is a third-order tensor if $j\neq 1, d$ and a second order tensor otherwise, and we will frequently use the notation
\begin{align*}
	\bold{U}_j(:, x_j, :) \in \mathbb{R}^{r_{j-1}\times r_j}, \qquad j \in \{1, \ldots, d\}, ~ x_j \in \{1, \ldots n_j\}
\end{align*}
to denote the $x_j^{\rm th}$ slice of the tensor $\bold{X}_j$, which is simply a matrix of dimension $r_{j-1}\times r_j$. Finally, we note that the set $\mathcal{U}_{\bold{r}}$ is either empty or a dense and open subset of $\overline{\mathcal{U}}_{\bold{r}}$. In the sequel, we will always assume that $\bold{r}$ is chosen such that $\mathcal{U}_{\bold{r}}$ is non-empty.

The next definition allows us to relate the parameters sets to order-$d$ tensors.
\begin{definition}[Tensor Train Contraction Mapping]\label{def:tensor_contraction}~
	
	Let $\bold{r} \in \mathbb{N}^{d-1}$, let $\{n_j\}_{j=1}^d \subset \mathbb{N}$,  and let the tensor train parameter space $\overline{\mathcal{U}}_{\bold{r}}$ be defined as in Definition \ref{def:parameter_space}. We introduce the tensor train contraction mapping $\tau \colon \overline{\mathcal{U}}_{\bold{r}} \rightarrow \mathbb{R}^{n_1 \times n_2 \times \ldots n_d}$ as the function with the property that for all $\bold{U}= (\bold{U}_1, \ldots, \bold{U}_{d})$ $ \in \overline{\mathcal{U}}_{\bold{r}}$ and all $x_j \in \{1, \ldots, n_j\}, ~ j=1, \ldots, d$ it holds that
	\begin{align*}
		\tau(\bold{U})(x_1, \ldots, x_d)&:= \tau(\bold{U}_1, \ldots, \bold{U}_d)(x_1, \ldots, x_d)\\[0.5em]
		&:=  \bold{U}_1(:, x_1, :)\bold{X}_2(:, x_2, :)\ldots \bold{U}_d(:, x_d, :)\\[0.5em]
		&=\sum_{k_1=1}^{r_1} \sum_{k_2=1}^{r_2} \ldots \sum_{k_{d-1}=1}^{r_{d-1}} \bold{U}_1(1, x_1, k_1)\bold{U}_2(k_1, x_2, k_2)\ldots \bold{U}_d(k_{d-1}, x_d, 1).     
	\end{align*} 
\end{definition}

The motivation for introducing the tensor train parameter space $\overline{\mathcal{U}}_{\bold{r}}$ and the rank-constrained parameter set $\mathcal{U}_{\bold{r}} \subset \overline{\mathcal{U}}_{\bold{r}}$ is now apparent through the use of the tensor train contraction $\tau \colon \overline{\mathcal{U}}_{\bold{r}} \rightarrow \mathbb{R}^{n_1 \times n_2 \times \ldots n_d}$: we have the simple relations (see, e.g., \cite{holtz2012manifolds})
\begin{align}\label{eq:TT_relations}
	\tau\big(\overline{\mathcal{U}}_{\bold{r}}\big) = \mathcal{T}_{\leq \bold{r}} \qquad \text{and} \qquad \tau\big({\mathcal{U}}_{\bold{r}}\big) = \mathcal{T}_{\bold{r}}.
\end{align}
Consequently, given any tensor $\tau(\bold{U}) \in \mathcal{T}_{\leq \bold{r}}$, we refer to $\bold{U}=(\bold{U}_1, \ldots \bold{U}_d)$ as a tensor train decomposition of $\tau(\bold{U})$, and we refer to $\bold{U}_j$ as the corresponding $j^{\rm th}$ TT core. Additionally, we will use the phrases \textit{separation ranks} and \textit{TT ranks} interchangeably. 

It has been shown in \cite{uschmajew2013geometry} that the tensor contraction mapping $\tau \colon \mathcal{U}_{\bold{r}} \rightarrow \mathcal{T}_{\bold{r}}$ is a submersion, i.e., $\tau \vert_{\mathcal{U}_{\bold{r}}}$ is of constant rank given by $\text{dim}\mathcal{T}_{\bold{r}}$. This implies, in particular, that $\tau$ is a differentiable mapping that maps open sets to open sets. On the other hand, it is well-known that $\tau$ is \emph{not} injective. Indeed, given $\bold{r}=(r_1, \ldots, r_d) \in \mathbb{N}^{d-1}$, let $ \bold{A}= (\bold{A}_1, \ldots, \bold{A}_d)$ be a tuple of invertible matrices with each $\bold{A}_i \in \mathbb{R}^{r_i \times r_i}$, and let $\bold{U} \in \overline{\mathcal{U}}_{\bold{r}}$. We can then define $\theta_{\bold{U}}(\bold{A}) = (\widehat{\bold{U}}_1, \ldots \widehat{\bold{U}}_d)\in \overline{\mathcal{U}}_{\bold{r}}$ as
\begin{align*}
	\widehat{\bold{U}}_1({x}_1) = {\bold{U}}_1({x}_1)\bold{A}_1, \quad \widehat{\bold{U}}_2({x}_2) = \bold{A}_1^{-1}{\bold{U}}_2({x}_2)\bold{A}_2, \quad \ldots, \quad \widehat{\bold{U}}_d({x}_d) = \bold{A}_{d-1}^{-1}{\bold{X}}_d({x}_d).
\end{align*}
where each $x_j \in \{1, \ldots, n_j\}$ for $j=1, \ldots, d$. Clearly, we then have $\tau(\bold{U})= \tau(\theta_{\bold{U}}(\bold{A}))$. 

As a consequence of the non-injectivity of the tensor contraction mapping $\tau$, it is useful to introduce the notion of so-called orthogonal tensor train decompositions.

\begin{definition}[Orthogonal Tensor Train Decompositions]\label{def:orthog}~
	
	Let $\bold{r} \in \mathbb{N}^{d-1}$, let $\overline{\mathcal{U}}_{\bold{r}}$ be defined according to Definition \ref{def:parameter_space}, let the tensor contraction mapping $\tau \colon \overline{\mathcal{U}}_{\bold{r}} \rightarrow \mathbb{R}^{n_1 \times n_2 \times \ldots \times n_d}$ be defined according to Definition \ref{def:tensor_contraction}, and let $\bold{U}=(\bold{U}_1, \ldots, \bold{U}_d) \in \overline{\mathcal{U}}_{\bold{r}}$. Then,  given $i \in \{1, \ldots, d\}$, we say that $\bold{U}$ is an \textbf{$i$-orthogonal} tensor decomposition of $\tau(\bold{U})\in \mathbb{R}^{n_1 \times n_2 \times \ldots \times n_d}$ if and only if for all $j \in \{1,\ldots, i-1\}$ and all $\ell \in \{i+1, \ldots, d\}$ it holds that
	\begin{align*}
		\big(\bold{U}_j^{<2>}\big)^{\intercal}\bold{U}_j^{<2>} = \bold{I} \in \mathbb{R}^{r_j \times r_j} \text{ and } \bold{U}_\ell^{<1>}\big(\bold{U}_\ell^{<1>}\big)^{\intercal} = \bold{I} \in \mathbb{R}^{r_{\ell-1} \times r_{\ell-1}}.
	\end{align*} 
	Additionally, if $\bold{U}$ is a 1-orthogonal (respectively, $d$-orthogonal) tensor decomposition of $\tau(\bold{U})$, then we say that $\bold{U}$ is a right-orthogonal (respectively left-orthogonal) tensor train decomposition of $\tau(\bold{U})$.
\end{definition}


Returning now to the alternative minimization problem \eqref{eq:2}, we see that one computationally feasible approach to approximating the solution to this problem is to utilize \textit{Riemannian} optimization algorithms on the tensor manifold $\mathcal{T}_{\bold{r}}$. These algorithms requires manipulating elements of the \textit{tangent space} to~$\mathcal{T}_{\bold{r}}$, and these elements can be represented in a data-sparse manner using the parameter space $\overline{\mathcal{U}}_{\bold{r}}$ and Relation \eqref{eq:TT_relations}. We refer interested readers to, e.g., \cite{Stein_thesis} for an overview of such techniques.

An alternative strategy-- and the focus of the present contribution-- is to utilize the tensor contraction mapping $\tau \colon \overline{\mathcal{U}}_{\bold{r}} \rightarrow \mathbb{R}^{n_1 \times n_2 \times \ldots n_d}$ to transport the energy functional $\mathcal{J} \colon U \subset \mathbb{R}^{n_1 \times n_2 \times \ldots n_d} \rightarrow \mathbb{R}$ to the parameter space $\overline{\mathcal{U}}_{\bold{r}}$, and then use optimization techniques directly on $\overline{\mathcal{U}}_{\bold{r}}$. Indeed, using the surjectivity of the tensor contraction mapping $\tau \colon \overline{\mathcal{U}}_{\bold{r}} \rightarrow \mathcal{T}_{\leq \bold{r}}$, we can define the set $\mathcal{V}_{\bold{r}}:= \tau^{-1}\big(\mathcal{T}_{\bold{r}}\cap U\big) \subset {\mathcal{U}}_{\bold{r}} $, and introduce the energy functional $\mathcal{j} \colon {\mathcal{V}}_{\bold{r}} \rightarrow \mathbb{R}$ defined as
\begin{align}\label{eq:j}
	\forall \bold{U}\in \mathcal{V}_{\bold{r}}\colon \qquad    \mathcal{j}(\bold{U}) := \mathcal{J} \circ \tau(\bold{U}).
\end{align}
We then consider the following minimization problem on the parameter set $\overline{\mathcal{U}}_{\bold{r}}$:
\begin{align}\label{eq:3}
	\mathcal{E}^*_{\bold{r}}:=\underset{\bold{U} \in  \mathcal{V}_{\bold{r}}}{\min}\; \mathcal{j}(\bold{U}).
\end{align}
Since the tensor contract mapping $\tau \colon \mathcal{U}_{\bold{r}} \rightarrow \mathcal{T}_{\bold{r}}$ is a submersion, it readily follows that two minimization problems \eqref{eq:2} and \eqref{eq:3} are equivalent in the sense that $\mathcal{E}^*_{\bold{r}}=E^*_{\bold{r}}$, and if $\bold{U}$ is a local minimizer of the energy functional $\mathcal{j}$ in $\mathcal{V}_{\bold{r}}$, then $\tau(\bold{U})$ is a local minimizer of $\mathcal{J}$ in $\mathcal{T}_{\bold{r}}$ and vice-versa. It is important to note, however, that since the tensor contraction mapping $\tau$ is not injective, even if the minimizer $\tau(\bold{U})$ of $\mathcal{J}$ is locally unique, the corresponding minimizer $\bold{U}$ of $\mathcal{j}$ is not.

\subsection{Alternating Minimization on the Tensor Train Parameter Space}\label{sec:2.2} The goal of this section is to briefly describe the celebrated alternating optimization algorithms used to numerically resolve the minimization problem \eqref{eq:3} on the tensor train parameter space $\overline{\mathcal{U}}_{\bold{r}}$. For energy functionals of the form \eqref{eq:energy_eig}, these algorithms are known as the one-site and two-site DMRG methods. Throughout this section, we will make use of the notions introduced in Section \ref{sec:2.1}. In particular, we assume that $\bold{r} \in \mathbb{N}^{d-1}$ with $r_0=r_{d}=1$, that $\{n_j\}_{j=1}^d \subset \mathbb{N}$, and the space $\overline{\mathcal{U}}_{\bold{r}}$ and its non-empty subset $\mathcal{U}_{\bold{r}}$ are defined according to Definition \ref{def:parameter_space}. Moreover, $j \in \{1, \ldots, d\}$ and $k \in \{1, \ldots, d-1\}$ will denote two arbitrary natural numbers.

We begin by introducing the so-called retraction operators that are used to formally set-up the one-site and two-site DMRG methods.

\begin{definition}[Retraction Operators]\label{def:retraction} Let $\bold{U}=(\bold{U}_1, \ldots, \bold{U}_d) \in \overline{\mathcal{U}}_{\bold{r}}$. We define the one-site retraction operator $\mathbb{P}_{\bold{U}, j, 1} \colon \mathbb{R}^{r_{j-1}\times n_j \times r_j} \rightarrow \mathbb{R}^{n_1 \times \ldots \times n_d}$ as the mapping with the property that for all $\bold{V}_j \in   \mathbb{R}^{r_{j-1}\times n_j \times r_j}$ it holds that
	\begin{align*}
		\mathbb{P}_{\bold{U}, j, 1} \bold{V}_j= \tau\big(\bold{U}_1, \ldots, \bold{U}_{j-1}, \bold{V}_j, \bold{U}_{j+1}, \ldots, \bold{U}_d\big).
	\end{align*}
	Additionally, we define the two-site retraction operator $\mathbb{P}_{\bold{U}, k, 2} \colon \mathbb{R}^{r_{k-1}\times n_k \times n_{k+1} \times r_{k+1}} \rightarrow \mathbb{R}^{n_1 \times \ldots \times n_d}$ as the mapping such that for all $\bold{W}_{k, k+1} \in   \mathbb{R}^{r_{k-1}\times n_k \times n_{k+1} \times r_{k+1}}$, all $i \in \{1, \ldots, d\}$ and all ${x}_i \in \{1,\ldots, n_i \} $ it holds that
	\begin{align*}
		&\big(\mathbb{P}_{\bold{U}, k, 2} \bold{W}_{k, k+1} \big)({x}_1, {x}_2, \ldots, {x}_d)\\[0.5em] 
		= &\bold{U}_1(:, x_1, :)\ldots\bold{U}_{k-1}(:, x_{k-1}, :)\bold{W}_{k, k+1}(:, x_{k}, x_{k+1}, :) \bold{U}_{k+2}(:, x_{k+2}, :)\ldots \bold{U}_d(:, x_d, :).
	\end{align*}
\end{definition}

A number of important properties of the one-site and two-site retraction operators have been deduced in \cite{holtz2012alternating}. Of particular importance is the following result, which helps explain the importance of employing orthogonal tensor train decompositions (in the sense of Definition \ref{def:orthog}) in the course of the DMRG algorithm.

\begin{lemma}[Orthogonality of the One-site and Two-site Retraction Operators]\label{lem:retraction}
	Let $\bold{U}=(\bold{U}_1, \ldots, \bold{U}_d) \in \overline{\mathcal{U}}_{\bold{r}}$ and $\bold{V}=(\bold{V}_1, \ldots, \bold{V}_d) \in \overline{\mathcal{U}}_{\bold{r}}$ be a $j$-orthogonal and $k$-orthogonal tensor train decomposition respectively. Then the associated one-site retraction operator $\mathbb{P}_{\bold{U}, j, 1} \colon \mathbb{R}^{r_{j-1}\times n_j \times r_j} \rightarrow \mathbb{R}^{n_1 \times \ldots \times n_d}$ and two-site retraction operator $\mathbb{P}_{\bold{V}, k, 2} \colon \mathbb{R}^{r_{k-1}\times n_k \times n_{k+1} \times r_{k+1}} \rightarrow \mathbb{R}^{n_1 \times \ldots \times n_d}$ are orthogonal projections in the sense that $\forall \bold{W}_j, \widetilde{\bold{W}}_j \in   \mathbb{R}^{r_{j-1}\times n_j\times r_j}$ and $\forall \bold{W}_{k, k+1}, \widetilde{\bold{W}}_{k, k+1} \in   \mathbb{R}^{r_{k-1}\times n_k\times n_{k+1}\times r_{k+1}}$ it holds that 
	\begin{align*}
		\big \langle \mathbb{P}_{\bold{U}, j, 1} \bold{W}_j,  \mathbb{P}_{\bold{U}, j, 1} \widetilde{\bold{W}}_j\big \rangle &= \big \langle  \bold{W}_j,  \widetilde{\bold{W}}_j\big \rangle, \quad  \text{and}\\
		\big \langle \mathbb{P}_{\bold{V}, k, 2} \bold{W}_{k, k+1},  \mathbb{P}_{\bold{V}, k, 2} \widetilde{\bold{W}}_{k, k+1}\big \rangle &= \big \langle  \bold{W}_{k, k+1},  \widetilde{\bold{W}}_{k, k+1}\big \rangle.
	\end{align*}
\end{lemma}
\begin{proof}
	See \cite[Lemma 3.1]{holtz2012alternating} and the preceding discussion.
\end{proof}

Lemma \ref{lem:retraction} has several useful numerical consequences in the practical implementation of DMRG methods and explains, in particular, why DMRG algorithms typically involve an orthogonalization step wherein the underlying tensor train decomposition is appropriately orthogonalized prior to the optimization step. To explain this further, we next introduce the notion of the one-site and two-site DMRG micro-iteration. 

\begin{definition}[One-site and Two-site DMRG micro-iterations]\label{def:micro}~
	
	Let $\mathcal{V}_{\bold{r}}= \tau^{-1}\big(\mathcal{T}_{\bold{r}}\cap U\big) \subset {\mathcal{U}}_{\bold{r}}$. We define the $j^{\rm th}$ one-site DMRG micro-iteration $\mathcal{S}_j \colon \mathcal{V}_{\bold{r}}\rightarrow \mathbb{R}^{r_{j-1}\times n_j \times r_j} $ as the mapping with the property that
	\begin{align}\label{eq:ALS_Micro}
		\forall \bold{U}=(\bold{U}_1, \ldots, \bold{U}_d) \in \mathcal{V}_{\bold{r}}\colon  \qquad \mathcal{S}_j(\bold{U}) =\underset{\bold{W}_j \in \mathbb{R}^{r_{j-1}\times n_j \times r_j}}{\text{\rm argmin}} \mathcal{J} \circ \mathbb{P}_{\bold{U}, j, 1}(\bold{W}_{j} ).
	\end{align}     
	Additionally, we define the $k^{\rm th}$ two-site DMRG micro-iteration, denoted $\mathcal{S}_{k, k+1} \colon \mathcal{V}_{\bold{r}}\rightarrow \mathbb{R}^{r_{k-1}\times n_k \times n_{k+1} \times r_{k+1}}$, as the mapping with the property that
	\begin{align} \label{eq:MALS_Micro}
		\forall \bold{V}=(\bold{V}_1, \ldots, \bold{V}_d) \in {\mathcal{V}}_{\bold{r}}\colon  \hspace{2mm} \mathcal{S}_{k, k+1}(\bold{V}) = \hspace{-3mm}\underset{\substack{\bold{W}_{k, k+1} \in \mathbb{R}^{r_{k-1}\times n_k \times n_{k+1} \times r_{k+1}}}}{\text{\rm argmin}} \mathcal{J} \circ \mathbb{P}_{\bold{V}, k, 2}(\bold{W}_{k, k+1}).
	\end{align}
\end{definition}

Equipped with the definition of the one-site and two-site DMRG micro-iterations, we next present the following lemma which indicates the importance of a judicious choice of orthogonalization for the underlying tensor train decomposition.

\begin{lemma}[DMRG Micro-iterations for Eigenvalue Problems]\label{lem:ALS_micro}~
	
	Let $ \mathcal{V}_{\bold{r}}= \tau^{-1}\big(\mathcal{T}_{\bold{r}}\cap U\big) \subset {\mathcal{U}}_{\bold{r}}$, let $\bold{U}$ and $\bold{V} \in \mathcal{V}_{\bold{r}}$ be a $j$-orthogonal and $k$-orthogonal tensor train decomposition respectively, and assume that the underlying energy functional $\mathcal{J} \colon U \subset \mathbb{R}^{n_1 \times n_2 \times \ldots n_d} \rightarrow \mathbb{R}$ is given by Equation \eqref{eq:energy_eig}, i.e., $\mathcal{J}(\bold{X})= \frac{\langle \bold{X}, \bold{A}\bold{X}\rangle}{\Vert \bold{X}\Vert^2}$ for all $\bold{X}\in  U=\mathbb{R}^{n_1 \times n_2 \times \ldots n_d}\setminus \{0\}$ and some symmetric operator $\bold{A}\colon \mathbb{R}^{n_1 \times n_2 \times \ldots n_d} \rightarrow \mathbb{R}^{n_1 \times n_2 \times \ldots n_d}$. Then the one-site and two-site DMRG micro-iterations can be obtained by solving the eigenvalue problems
	\begin{align*}
		\mathbb{P}_{\bold{U}, j,1}^* \bold{A}\mathbb{P}_{\bold{U}, j, 1}\bold{W}_j= \lambda_{j}^{\bold{A}}\bold{W}_j, \qquad \text{and} \qquad 
		\mathbb{P}_{\bold{V}, k, 2}^* \bold{A}\mathbb{P}_{\bold{V}, k, 2}\bold{W}_{k, k+1}={\lambda}_{k, k+1}^{\bold{A}} \bold{W}_{k, k+1},
	\end{align*}
	where $ \mathbb{P}_{\bold{U}, j,1}^*$ and $\mathbb{P}_{\bold{V}, k, 2}^*$ are the adjoints of the retraction operators $ \mathbb{P}_{\bold{U}, j,1}$ and $\mathbb{P}_{\bold{V}, k, 2}$ respectively, and  $\lambda_{j}^{ \bold{A}}, {\lambda}_{k, k+1}^{\bold{A}}$ are the smallest eigenvalues of the respective operators.
\end{lemma}
\begin{proof}
	This essentially follows from Lemma \ref{lem:retraction} but see \cite{holtz2012alternating} for further details.
\end{proof}

\begin{remark}\label{rem:ALS_micro}
	Consider the setting of Lemma \ref{lem:ALS_micro}. As discussed in \cite{holtz2012alternating}, if $\bold{U}, \bold{V}\in \mathcal{V}_{\bold{r}}$ are not $j$-orthogonal and $k$-orthogonal respectively, the DMRG micro-iterations will consist of solving badly-conditioned generalized eigenvalue problems-- hence the importance of the orthogonalization step in practical DMRG implementations.
\end{remark}

The classical one-site and two-site DMRG algorithms are now presented below in the form of Algorithm \ref{alg:ALS}. We end this section with the following important remark on the well-posedness of successive DMRG micro-iterations.

\begin{remark}[Well-Posedness of Successive DMRG Micro-iterations]\label{rem:well-posedness}~
	
	In order for the DMRG Algorithm \ref{alg:ALS} to be well-posed, we require that for a given initialization $\bold{U}_{\rm LR}^{0, (0)} \in \mathcal{V}_{\bold{r}}$, each subsequent iterate $\bold{U}_{\rm LR}^{j, (n)}$, $\bold{U}_{\rm RL}^{k, (n)}$~ $j \in \{1, \ldots, d-1\}, ~ k \in \{2, \ldots, d\}$, and $n \in \mathbb{N}$ should also be an element of $ \mathcal{V}_{\bold{r}}$. This is indeed true for quadratic and Rayleigh quotient-based energy functionals $\mathcal{J}$ (and in the two-site DMRG case, for a judicious choice of rank truncation parameter $\widetilde{\bold{r}}$) provided that the initialization is sufficiently close to the true minimizer \cite{rohwedder2013local, uschmajew2013geometry}. For a convergence analysis of the one-site DMRG algorithm in these cases, we refer interested readers to \cite{rohwedder2013local, uschmajew2013geometry}.
\end{remark}

\begin{algorithm}
	\caption{Classical One-site/ Two-site DMRG algorithm}\label{alg:ALS}
	\begin{algorithmic}
		\State\hspace{-0.4cm}\textbf{Objective:} Compute $\underset{\bold{U} \in  {\mathcal{V}}_{\bold{r}}}{\min}\; \mathcal{j}(\bold{U})$ for some choice of $\bold{r}\in \mathbb{N}^{d-1}$.
		
		\vspace{1mm}
		\Require \textit{Right-orthogonal} TT decomposition $\bold{U}_{\rm LR}^{0, (0)} \in \mathcal{V}_{\bold{r}}= \tau^{-1}\big(\mathcal{T}_{\bold{r}}\cap U\big) \subset {\mathcal{U}}_{\bold{r}}$.
		\For{$n=0, 1, 2, \ldots,$}
		
		\vspace{1mm}
		\For{$i=1, 2, \ldots, d-1$}

		\State 1. Perform the $i^{\rm th}$ one-site/two-site DMRG micro-step, i.e., \vspace{-3mm}
		\begin{align*}
			\hspace{0.2cm}\text{Compute either }  \bold{W}_i:=\mathcal{S}_i(\bold{U}_{\rm LR}^{i, (n+1)}) \text{ or }   \bold{W}_{i, i+1}=\mathcal{S}_{i, i+1}(\bold{U}_{\rm LR}^{i, (n+1)}). 
		\end{align*}

		\vspace{1mm}
		\State 2. (One-Site DMRG) Perform the orthogonalization step:\vspace{-3mm}
		\begin{align*}
			&\hspace{1.45cm}\text{Compute } \bold{W}_i^{<2>}=\bold{Q}^{<2>}_i \bold{R}_i ~ \text{with } \bold{Q}^{<2>}_i\in \mathbb{R}^{r_{i-1}n_i \times r_i}, ~  \bold{R}_i \in \mathbb{R}^{r_i \times r_i},\\
			& \hspace{1.45cm}\text{Update tensor train decomposition:}\\
			&\hspace{1.45cm}\bold{U}_{\rm LR}^{i, (n+1)}= (\bold{U}^{(n+1)}_1, \ldots,\bold{U}^{(n+1)}_{i-1}, \widetilde{\bold{U}}_i^{(n)}, \bold{U}^{(n)}_{i+1}, \ldots,\bold{U}^{(n)}_d )\\
			&\hspace{0.85cm}\hookrightarrow \bold{U}_{\rm LR}^{i+1, (n+1)}= (\bold{U}^{(n+1)}_1, \ldots,\bold{U}^{(n+1)}_{i-1},\bold{U}^{(n+1)}_{i}, \widetilde{\bold{U}}^{(n)}_{i+1}, \ldots,\bold{U}^{(n)}_d), \\
			&\hspace{1.45cm}\text{where } \bold{U}^{(n+1)}_{i}= \bold{Q}_i, ~ \widetilde{\bold{U}}^{(n)}_{i+1}=\bold{R}_i\bold{U}^{(n)}_{i+1}.
		\end{align*}
		
		\State 2. (Two-Site DMRG) Select $\widetilde{\bold{r}}_i$ and perform the truncated SVD step:\vspace{-3mm}
		\begin{align*}
			&\hspace{1cm}\text{Compute } \bold{W}_{i, i+1}^{<2>}=\bold{P}_i^{<2>} \bold{\Sigma}_{i} \bold{Q}_{i}^{<1>}~ \text{with } \bold{P}_i^{<2>}\in \mathbb{R}^{r_{i-1}n_i \times \widetilde{r}_i}, ~  \bold{Q}_{i}^{<1>} \in \mathbb{R}^{\widetilde{r}_i \times n_{i+1} r_{i+1}},\\
			& \hspace{1cm}\text{Update tensor train decomposition:}\\
			&\hspace{1cm}\bold{U}_{\rm LR}^{i, (n+1)}= (\bold{U}^{(n+1)}_1, \ldots,\bold{U}^{(n+1)}_{i-1}, \widetilde{\bold{U}}_i^{(n)}, \bold{U}^{(n)}_{i+1}, \ldots,\bold{U}^{(n)}_d )\\
			&\hspace{0.4cm}\hookrightarrow \bold{U}_{\rm LR}^{i+1, (n+1)}= (\bold{U}^{(n+1)}_1, \ldots,\bold{U}^{(n+1)}_{i-1},\bold{U}^{(n+1)}_{i}, \widetilde{\bold{U}}^{(n)}_{i+1}, \ldots,\bold{U}^{(n)}_d), \\
			&\hspace{1cm}\text{where } \bold{U}^{(n+1)}_{i}= \bold{P}_i, ~ \widetilde{\bold{U}}^{(n)}_{i+1}=\bold{S}_i\bold{Q}_i.
		\end{align*}
		
		\EndFor
		\vspace{2mm}
		\State Set $\bold{U}_{\rm RL}^{d, (n+1)}=\bold{U}_{\rm LR}^{d-1, (n+1)}$.
		\vspace{2mm}
		
		\For{$i=, d, d-1, \ldots, 2$}
		\State Repeat Steps 1. and 2. above in the opposite direction.
		\EndFor
		\vspace{2mm}
		\State Set $\bold{U}_{\rm LR}^{1, (n+1)}=\bold{U}_{\rm RL}^{2, (n+1)}$.
		
		\vspace{2mm}
		\EndFor
	\end{algorithmic}
\end{algorithm}

\section{Additive Two-Level DMRG with Coarse Space Correction}\label{sec:3}~

Equipped with the tools and notions presented in Section \ref{sec:2}, we are now ready to present our proposal for the additive two-level DMRG algorithm with coarse space correction, or A2DMRG for short. Prior to presenting the details, let us first present a high-level intuitive explanation of the algorithm. To avoid tedious notation, we describe the two-site case but the idea behind the one-site version is identical.

{Essentially, the A2DMRG algorithm consists of the following steps: Given an initial low-rank tensor  $\bold{Y}^{(0), 0}\in \mathcal{T}_{\bold{r}}$, we first perform $d-1$ \textit{local} optimization steps. For each $j\in \{1, \ldots, d-1\}$, the $j^{\rm th}$ local optimization step optimizes the combined $j^{\rm th}$ and $j+1^{\rm st}$ \textit{two-site} TT core of the $j$-orthogonal TT decomposition of $\bold{Y}^{(0), 0}$, this choice of orthogonalization being motivated by Lemma \ref{lem:ALS_micro} and Remark \ref{rem:ALS_micro}. Consequently, each of these local steps can be performed independently in parallel and we obtain $d-1$ updated tensors $\{{\bold{Y}}^{(0), j}\}_{j=1}^{d-1}$. Next, we compute an optimal linear combination $\widetilde{\bold{Y}}^{(1), 0}\in \mathbb{R}^{n_1 \times \ldots \times n_d}$ of the tensors $\{{\bold{Y}}^{(0), j}\}_{j=0}^{d-1}$ by optimizing the given energy functional $\mathcal{J}\colon U \subset \mathbb{R}^{n_1\times \ldots \times n_d}\rightarrow \mathbb{R}$ on the coarse space $\text{span}\{{\bold{Y}}^{(0), j}\}_{j=0}^{d-1}$. Finally, we apply a compression step by projecting $\widetilde{\bold{Y}}^{(1), 0}$ onto the low-rank manifold $\mathcal{T}_{\widetilde{\bold{r}}} \subset \mathbb{R}^{n_1 \times \ldots \times n_d}$ for a judicious choice of rank parameters $\widetilde{\bold{r}}\in \mathbb{N}^{d-1}$. This yields an updated low-rank tensor $\bold{U}^{(1), 0}\in \mathcal{T}_{\widetilde{\bold{r}}}$ that can be used as a new initialization.}

An in-depth description of the A2DMRG algorithm is presented below in the form of \Cref{alg:PALS}. We observe that the algorithm consists of four core steps. Let us now elaborate on each of these steps.

\begin{algorithm}[H]
	\caption{The Additive Two-Level DMRG Algorithm (A2DMRG)}\label{alg:PALS}
	\begin{algorithmic}
		\State\hspace{-0.4cm}\textbf{Objective:} Compute $\underset{\bold{X} \in  {\mathcal{V}}_{\bold{r}}}{\min}\; \mathcal{j}(\bold{X})$ for some choice of $\bold{r}\in \mathbb{N}^{d-1}$.
		
		\vspace{0mm}
		\Require \textit{Left-orthogonal} tensor train decomposition $\bold{U}^{(0), d} \in \mathcal{V}_{\bold{r}}= \tau^{-1}\big(\mathcal{T}_{\bold{r}}\cap U\big)$.
		\For{$n=0, 1, 2, \ldots,$}
		
		\vspace{1mm}
		\State Save a \emph{left-orthogonal} TT decomposition $\bold{U}^{(n), d}= (\bold{U}_1^{(n)}, \ldots , \bold{U}_{d-1}^{(n)}, {\bold{W}}_d^{(n)}) \in \mathcal{V}_{\bold{r}}$:
		
		\vspace{1mm}
		
		\begin{mdframed}[hidealllines=true,backgroundcolor=gray!20]
			\For{$i= d, d-1, \ldots, 2$}
			\State Given an $i$-orthogonal TT decomposition: \vspace{-1mm}
			\begin{align*}
				\bold{U}^{(n), i}= \big(\bold{U}_1^{(n)}, \ldots , \bold{U}_{i-1}^{(n)}, \bold{W}_i^{(n)}, \bold{V}_{i+1}^{(n)}, \ldots, \bold{V}_{d}^{(n)}\big) \in \mathcal{V}_{\bold{r}}
			\end{align*}
			\State Perform an orthogonalization step:\vspace{-1mm}
			\begin{align*}
				\hspace{0.5cm}&\text{Compute }~   {\bold{W}}_i^{<1>}=\bold{L}_i\bold{Q}^{<i>}_i  \text{ with } \bold{L}_i\in \mathbb{R}^{r_{i-1}\times r_{i-1}}, ~  \bold{Q}^{<i>}_i \in \mathbb{R}^{r_{i-1} \times n_ir_i},\\
				\hspace{1.05cm}&\text{Save the $(i-1)$-orthogonal tensor train decomposition:}\\
				&\hspace{1.05cm}\bold{U}^{(n), {i-1}}= \big(\bold{U}_1^{(n)}, \ldots , \bold{U}_{i-2}^{(n)}, \bold{W}_{i-1}^{(n)}, \bold{V}_{i}^{(n)}, \ldots, \bold{V}_{d}^{(n)}\big) \in \mathcal{V}_{\bold{r}}\\
				&\hspace{1.45cm}\text{where } \bold{W}_{i-1}^{(n)}= \bold{U}_{i-1}^{(n)}\bold{L}_i, ~ \bold{V}_i^{(n)}=\bold{Q}_i.
			\end{align*}    
			\vspace{-9mm}
			\EndFor
		\end{mdframed}

		\vspace{1mm}

		\begin{mdframed}[hidealllines=true,backgroundcolor=gray!60]
			\For{$i=1, 2, \ldots, d$ (respectively, $i=1, 2, \ldots, d-1$) \underline{in parallel}}
			
			\State  Perform a one-site (respectively, two-site) DMRG micro-step, i.e., \vspace{-1mm}
			\begin{align*}
				\hspace{1.15cm} \text{Compute }~ \bold{V}_i=\mathcal{S}_i(\bold{U}^{(n), i}), \text{ or compute }~ \bold{W_{i, i+1}} =\mathcal{S}_{i, i+1}(\bold{U}^{(n), i}).
			\end{align*}
			\vspace{-1mm}
			\State Define updated TT decomposition $\widetilde{\bold{U}}^{(n+1), i}$ using $\bold{V}_i$ or  $\bold{W_{i, i+1}}$.  \vspace{1mm}   
			\EndFor
			\vspace{2mm}
			\State Set $\widetilde{\bold{U}}^{(n+1), 0}= \bold{U}^{(n), d}$.  
			
		\end{mdframed}

		\vspace{1mm}

		\begin{mdframed}[hidealllines=true,backgroundcolor=gray!20]
			
			\State  Compute a minimizer $\{c^*_j\}_{j=0}^d  \in \mathbb{R}^{d+1} $  of
			\vspace{-4mm}
			\begin{align}\label{eq:d_optim}
				\underset{\{c_j\}_{j=0}^d \in \mathbb{R}^{d+1}}{\text{min}} \;\mathcal{J} \Big(\sum_{j=0}^d c_j\tau\left(\widetilde{\bold{U}}^{(n+1), j}\right) \Big) \qquad &\text{(One-Site Case)},\\
				\intertext{\hspace{1.6cm}or a minimizer $\{d^*_j\}_{j=0}^{d-1}  \in \mathbb{R}^{d-1} $ of}
				\underset{\{d_j\}_{j=0}^{d-1} \in \mathbb{R}^d}{\text{min}} \;\mathcal{J} \Big(\sum_{j=0}^{d-1} d_j\tau\left(\widetilde{\bold{U}}^{(n+1), j}\right) \Big) \qquad &\text{(Two-Site Case)}. \nonumber
			\end{align}
			
		\end{mdframed}
		
		\vspace{1mm}

		\begin{mdframed}[hidealllines=true,backgroundcolor=gray!60]
			
			\State Apply compression, i.e., compute left-orthogonal, quasi-optimal approximation 
			\vspace{-6mm}
			\begin{align*}
				\mathcal{V}_{\widetilde{\bold{r}}} \ni \bold{U}^{(n+1), d} \approx \sum_{j=0}^{d} c^*_j\widetilde{\bold{U}}^{(n+1), j} \qquad &\text{(One-Site Case)} \quad \text{or},\\
				\mathcal{V}_{\widetilde{\bold{r}}} \ni \bold{U}^{(n+1), d} \approx \sum_{j=0}^{d-1} d^*_j\widetilde{\bold{U}}^{(n+1), j} \qquad &\text{(Two-Site Case)},
			\end{align*}
			\vspace{-1mm}
			for a given choice of rank parameter $\widetilde{\bold{r}}\geq \bold{r}$. 
		\end{mdframed}
		
		\vspace{0mm}
		\EndFor
	\end{algorithmic}
\end{algorithm}

The first step of the A2DMRG \Cref{alg:PALS} consists a sequence of LQ orthogonalizations are performed on the components of the initial tensor train decomposition $\bold{U}^{(n), d}$. The goal of this step is to obtain a $j$-orthogonal copy of $\bold{U}^{(n), d}$ for each $j \in \{1, \ldots, d\}$. Note that this type of orthogonalization procedure is typically applied as a pre-processing step in the classical one-site and two-site DMRG algorithms \ref{alg:ALS} to obtain an appropriate initialization. In the A2DMRG \Cref{alg:PALS} therefore, we simply have to apply the orthogonalization twice. Current implementations of this step are sequential in nature.

In the second step, for each $j \in \{1, \ldots, d\}$  (respectively, $j \in \{1, \ldots, d-1\}$), we apply the $j^{\rm th}$ one-site (respectively, two-site) DMRG micro-iteration on the $j$-orthogonal copy of $\bold{U}^{(n), d}$. In view of Lemma \ref{lem:ALS_micro}, the resulting micro-iterations are numerically well-conditioned if the energy functional $\mathcal{J}$ is a quadratic form involving a well-conditioned symmetric positive definite tensor operator $\bold{A}$. In case the energy functional $\mathcal{J}$ is a Rayleigh quotient involving a symmetric operator $\bold{A}$, Remark \ref{rem:ALS_micro} shows that the micro-iterations involve the solution of a symmetric eigenvalue problem. We emphasize that each one-site/two-site DMRG micro-iteration in this step can be performed by an independent processor in parallel, and this is discussed in further detail in Section \ref{sec:3.3} in the supplementary material.

In the third step, we perform a second-level minimization by seeking a minimizer of the energy functional $\mathcal{J}$ on a vector space of dimension at most $d+1$ (one-site setting) or $d$ (two-site setting). This step, which is the subject of extensive discussion in Section \ref{sec:3.1} below, requires gathering the eigenfunctions (reshaped into updated TT cores) computed by each processor in parallel. Note that this is similar to two-level domain decomposition methods for solving linear systems of equations. 

Since any minimizer obtained in the previous step will, in general, possess a tensor train decomposition with higher ranks than the initial ranks $\bold{r}$ of $\bold{U}^{(n), d}$, the fourth step in A2DMRG consists of the application of some compression algorithm (such as the TT rounding algorithm of Oseledets \cite{oseledets2011tensor}) to this minimizer so as to obtain a quasi-optimal approximation with lower ranks. This step, which is-- at least partially-- also sequential in nature, is the subject of further discussion in Section \ref{sec:3.3} in the supplementary material.

We conclude this section with a number of important remarks.


\begin{remark}[Extensions of the A2DMRG \Cref{alg:PALS}]
	
	Consider once again the A2DMRG \Cref{alg:PALS} and assume that the underlying energy functional $\mathcal{J}\colon \mathbb{R}^{n_1 \times \ldots \times n_d}$ takes the form of a Rayleigh quotient as given by Equation \eqref{eq:energy_eig}. In this situation, at least two straightforward extensions of \Cref{alg:PALS} are conceivable, both based on modifying the coarse space. 
	
	\begin{description} 
		\item[Extension One] As noted in Lemma \ref{lem:ALS_micro}, for each $i \in \{1, \ldots, d\}$, the $i^{\rm th}$ one-site DMRG micro-iteration consists of using the Lanczos algorithm to compute the lowest eigenpair of a matrix $\bold{A}_i$ of dimension $r_{i-1}n_i r_i \times r_{i-1}n_i r_i$. An updated TT decomposition is then obtained by replacing the $i^{\rm th}$ TT core of the previous iterate with the newly computed lowest eigenvector (see Step 2 of \Cref{alg:PALS}). The coarse space is constructed using these $d$ updated TT decompositions along with the previous global iterate.
		
		Notice, however, that the symmetric matrix $\bold{A}_i$ possesses $r_{i-1}n_ir_i$ orthogonal eigenvectors. An obvious idea is therefore to store not only the lowest eigenvector of $\bold{A}_i$ but the $K \in \mathbb{N}$ lowest eigenvectors yielded by the Lanczos algorithm. We can then construct $K$ updated TT decompositions by replacing the $i^{\rm th}$ TT core of the previous iterate with the each of the newly obtained and reshaped $K$ eigenvectors. Proceeding in this manner for each of the $d$ one-site DMRG micro-iterations allows to construct an enriched coarse space using the linear span of $Kd+1$ tensors rather than $d+1$ tensors as before. Similar considerations, of course, apply to two-site A2DMRG.
		
		\item[Extension Two] Given a TT decomposition consisting of $d$ TT cores, the one-site A2DMRG \Cref{alg:PALS} consists of performing $d$ independent, parallel local optimizations, one for each TT core. This yields $d$ updated TT decompositions and a coarse space consisting of the linear span of $d$+1 tensors.
		
		Another natural idea is to instead partition the $d$ TT cores into $M < d$ groups, with each group consisting of $d/M$ adjoining TT cores. For any such group, the $d/M$ adjoining TT cores can then be \textit{sequentially optimized} using the classical DMRG algorithm, resulting in a single, new TT decomposition that contains $d/M$ updated TT cores. Proceeding in this manner for each of the $M$ groups therefore yields $M$ updated TT decompositions (each computed independently and in parallel) with a corresponding coarse space constructed using these $M$ updated TT decompositions along with the previous global iterate. Note that this extension can,  in particular,  be seen as combining the approach of Stoudenmire and White \cite{stoudenmire2013real} with a second-level coarse space correction.
	\end{description}
	
	An exploration of these ideas will be the subject of future work. 
\end{remark}

\begin{remark}[Compression Step in A2DMRG \Cref{alg:PALS}]~
	
	A detailed study of the computational cost of the A2DMRG \Cref{alg:PALS} in comparison to the classical, sequential DMRG algorithm is the subject of Section \ref{sec:3.3} in the supplementary material. Let us nonetheless briefly comment on the computational cost of the final compression step of \Cref{alg:PALS}. 
	
	For the one-site (respectively, two-site) algorithm, this step consists of compressing the TT ranks of a linear combination of $d+1$ (respectively, $d$) tensors in TT format each with maximal rank ${r}\in \mathbb{N}$. It is well-known that such a linear combination of $d$ tensors possesses a tensor train decomposition of maximal rank $(d+1){r}$ (respectively, $dr$). At first glance, using the TT-rounding algorithm of Oseledets \cite{oseledets2011tensor} or its various randomized improvements would then cost $\mathcal{O}(d^3nr^3)$ flops but this is misleading. 
	
	The essential point is that each of the TT decompositions in the linear combination differs from the rest in at most two (respectively, four) TT cores. Consequently, any tensor resulting from a linear combination of such TT format tensors is highly structured (for instance, in the one-site case, it has a TT decomposition of maximal rank $2r$). The compression step can therefore be performed with a much reduced cost. We refer to Section \ref{sec:3.3} in the supplementary material for details. {Table \ref{table_1} below also summarizes the dominant computational cost of one global iteration (one half-sweep, respectively,) of the A2DMRG algorithm (classical DMRG algorithm, respectively).}
\end{remark}

\begin{table}[h!] 
	\centering
	\begin{tabular}{|c|c|c|}
		\hline
		~ & Half-sweep DMRG cost & A2MDRG global iteration cost per processor \\
		\hline
		1-site &  $\mathcal{O}\big(d n r^3 R(K+d)\big)$ & $\mathcal{O}\big(n r^3 R(K+d)\big)+ \mathcal{O}\big(d nr^3RK'\big)$  \\
		\hline
		2-site &  $\mathcal{O}\big(d n r^3 R(nK+d)\big)$ & $\mathcal{O}\big(n r^3 R(Kn+d)\big) + \mathcal{O}\big(d n^3r^3RK'\big)$ \\
		\hline
	\end{tabular}
	\caption{{Dominant computational cost of the classical DMRG and A2DMRG methods for solving eigenvalue problems arising from quantum chemistry. Here, $n=\max \{n_j\}_{j \in \mathbb{N}}$, $r=\max\bold{r}$, $R$ is the maximal rank in the MPO representation of the underlying tensor operator $\bold{A}$, $K$ is the maximal number of Krylov iterations required for all DMRG micro-iterations and $K'$ is the number of Krylov iterations required for the solution of the second-level minimization problem. We also assume that $nR < r$, and that the SVD steps in two-site DMRG micro-iterations produce maximal ranks $\widetilde{r}=\mathcal{O}(r)$. For a derivation of these costs, we refer to the supplementary material.}}
	\label{table_1}
\end{table}

{
	\begin{remark}[Connection to Additive Schwarz Methods]\label{rem:dd} 
		As claimed in the introduction, the A2DMRG Algorithm \ref{alg:PALS} is inspired by the well-known additive Schwarz methods from the domain decomposition literature. To see this connection, let us briefly recall the abstract framework of additive Schwarz methods \cite{griebel1995abstract}: Let $\big(V, (\cdot, \cdot)_V\big)$ be a finite-dimensional Hilbert space, let $A \colon V \rightarrow V$ be a symmetric, positive definite operator, and suppose we are interested in solving, for some given $b \in V$
		\begin{align*}
			Au &= b  \qquad \text{or equivalently, in variational form,}\\
			\underset{v \in V}{\text{argmin}} \mathcal{R}(v) &\qquad \text{with }~ \mathcal{R}(v)=\frac{1}{2} (Av, v)_V - (b, v)_V.
		\end{align*}
		We assume that $V$ possesses a decomposition of the form $V = \sum_{j=1}^d R_j^{\dagger}V_j$ where each $V_j$ is a Hilbert space and $R_j^{\dagger} \colon V_j \rightarrow V$ is a prolongation operator with corresponding restriction operator $R_j \colon V \rightarrow V_j$. In this framework, assuming the initialization $u^{(0)}\in V$, an additive Schwarz method with damping factors $\{\omega_j\}_{j=0}^d \in \R, ~\omega_0=1,$ can be defined through the iterates
		\begin{align*}
			u^{(\ell+1)}= \omega_0u^{(\ell)} - \sum_{j=1}^d \omega _jR_j^{\dagger}{z}^{(\ell)}_j \hspace{2mm} \text{with } ~  {z}^{(\ell)}_j=(R_j A R_j^{\dagger})^{-1} R_j A u^{(\ell)} - (R_j A R_j^{\dagger})^{-1} R_jb .
		\end{align*}
		The key observation is that each update $z^{(\ell)}_j\in V_j$ defined above can be viewed as the minimizer of an energy functional, i.e.,
		\begin{align*}
			z^{(\ell)}_j =& \underset{y_j \in V_j}{\text{argmin}} \frac{1}{2}\big(R_j A R_j^{\dagger} y_j, y_j\big)_V - \big(R_jA u^{(\ell)}- R_jb, y_j\big)_V=\underset{y_j \in V_j}{\text{argmin}}\; \mathcal{R}(u^{(\ell)}- R_j^{\dagger}y_j),\\
			\text{so that }& ~u^{(\ell+1)}= \omega_0u^{(\ell)} - \sum_{j=1}^d \omega _jR_j^{\dagger}\;\underset{y_j \in V_j}{\text{argmin}}\; \mathcal{R}(u^{(\ell)}- R_j^{\dagger}y_j).
		\end{align*}
		
		Using now a bit of algebra together with the multi-linearity of the tensor contraction mapping and known properties of the retraction operators \cite[Lemma 3.2]{holtz2012alternating}, we can deduce that one global iteration of the one-site A2DMRG algorithm \ref{alg:PALS} with initialization $\tau(\bold{U}^{(\ell)}) \in \mathcal{T}_{\bold{r}}$ corresponds to a generalization of the above additive Schwarz iteration where the quadratic energy functional $\mathcal{R}$ is replaced with a general functional $\mathcal{J}\colon U \subset \R^{n_1 \times \ldots n_d}\rightarrow \R$ and 
		\begin{align*}
			V&=\bold{T}_{\tau(\bold{U}^{(\ell)})}\mathcal{T}_{\bold{r}}, \quad \text{i.e., the tangent space at $\tau(\bold{U}^{(\ell)})$ to $\mathcal{T}_{\bold{r}}$},\\
			R_j^{\dagger}&= \mathbb{P}_{\bold{U}^{(\ell)}, j, 1}, \qquad \text{and} \qquad V_j= \mathbb{R}^{r_{j-1}\times n_j \times r_j},\\
			R_j^{\dagger}V_j&=\left\{\tau({\bold{U}}^{(\ell), j})\colon {\bold{U}}^{(\ell), j}= (\bold{U}^{(\ell)}_1, \ldots, \bold{U}_{j-1}^{(\ell)}, \bold{V}_j, \bold{U}_{j+1}^{(\ell)}, \ldots \bold{U}_d^{(\ell)}), ~ \bold{V}_j \in \R^{r_{j-1}\times n_j \times r_j}\right\},\\[0.5em]
			&\{\omega_j\}_{j=0}^d  \hspace{2mm}\text{ is determined by the second-level minimization}.
		\end{align*}
		A closer study of this connection as well as a convergence analysis of A2DMRG is work in progress.
	\end{remark}
}

\subsection{The Second-Level Minimization in the A2DMRG \Cref{alg:PALS}}\label{sec:3.1}~

An essential feature of the A2DMRG \Cref{alg:PALS} is the solution of a minimization problem of dimension at most $d+1 \in \mathbb{N}$ for the one-site case and $d\in \mathbb{N}$ for the two-site case (recall that $d$ is the order of the underlying tensor space). This minimization problem yields an optimal linear combination of the $d$ (respectively, $d-1$) tensor train decompositions arising from the first-level one-site (respectively, two-site) DMRG micro-iterations together with the previous global iterate. We now discuss the resolution of this second-level minimization problem for Rayleigh quotient-based energy functionals $\mathcal{J}\colon U \subset \mathbb{R}^{n_1 \times \ldots \times n_d} \rightarrow \mathbb{R}$ of the form \eqref{eq:energy_eig}, i.e., 
\begin{align*}
	\forall \bold{X}\in  U=\mathbb{R}^{n_1 \times n_2 \times \ldots n_d}\setminus \{0\}\colon \qquad        \mathcal{J}(\bold{X})= \frac{\langle \bold{X}, \bold{A}\bold{X}\rangle}{\Vert \bold{X}\Vert^2},
\end{align*}
where $\bold{A}\colon \mathbb{R}^{n_1 \times n_2 \times \ldots n_d} \rightarrow \mathbb{R}^{n_1 \times n_2 \times \ldots n_d}$ is a symmetric operator with a simple lowest eigenvalue. For simplicity, we assume we are in the one-site setting (the two-site being identical up to a modification of indices). In this case, the second-level minimization problem \eqref{eq:d_optim} for the $n^{\rm th}$ global iteration takes the form
\begin{align}\label{eq:d_opt_lin_3}
	\underset{0\neq \bold{X} \in \mathcal{Y}_n}{\min}\; \frac{\langle \bold{X}, \bold{A}\bold{X}\rangle}{\Vert \bold{X}\Vert^2} \qquad \text{where }~ \mathcal{Y}_n := \text{span}\Big\{\tau\big(\widetilde{\bold{U}}^{(n+1), i}\big)\colon \quad i \in \{0, \ldots, d\}\Big\}.
\end{align}
Using the fact that the tensors $\{\tau(\widetilde{\bold{U}}^{(n+1), i})\}_{i=0}^d$ span the space $\mathcal{Y}_n$ by definition, we can express the minimization problem \eqref{eq:d_opt_lin_3} as a matrix problem of the form
\begin{align}\label{eq:d_opt_lin_4}
	\underset{\substack{\bold{c}=(c_0, \ldots, c_d) \in \mathbb{R}^{d+1}\\[0.5em] {\sum_{i=0}^d c_i\tau(\widetilde{\bold{U}}^{(n+1), i})}\neq 0}}{\min}\; \frac{\langle \bold{c}, \widehat{\bold{A}}\bold{c}\rangle}{\langle \bold{c}, \widehat{\bold{S}}\bold{c}\rangle}.
\end{align}
Here, the symmetric matrix $\widehat{\bold{A}} \in \mathbb{R}^{d+1 \times d+1}$ and overlap matrix $\widehat{\bold{S}}\in \mathbb{R}^{d+1\times d+1}$ are
\begin{align}\label{eq:A_reduced}
	\forall i, j \in \{1, \ldots, d+1\}\colon \qquad [\widehat{\bold{A}} ]_{i, j}&:= \big\langle \tau(\widetilde{\bold{U}}^{(n+1), {i-1}}), \bold{A}\;\tau(\widetilde{\bold{U}}^{(n+1), {j-1}})\big\rangle\\ \nonumber
	\forall i, j \in \{1, \ldots, d+1\}\colon \qquad [\widehat{\bold{S}} ]_{i, j}&:= \big\langle \tau(\widetilde{\bold{U}}^{(n+1), {i-1}}),\tau(\widetilde{\bold{U}}^{(n+1), {j-1}})\big\rangle.
\end{align}
Since $\widehat{\bold{S}}$ is symmetric, we can introduce the eigendecomposition $\widehat{\bold{S}}= \widehat{\bold{V}}_{\bold{S}}\widehat{\bold{\Sigma}}_{\bold{S}} \widehat{\bold{V}}_{\bold{S}}^*$ and deduce that
\begin{align*}
	\underset{\substack{\bold{c}=(c_0, \ldots, c_d) \in \mathbb{R}^{d+1}\\[0.5em] {\sum_{i=0}^d c_i\tau(\widetilde{\bold{U}}^{(n+1), i})}\neq 0}}{\min}\; \frac{\langle \bold{c}, \widehat{\bold{A}}\bold{c}\rangle}{\langle \bold{c}, \widehat{\bold{S}}\bold{c}\rangle} &=   \underset{\substack{\bold{c}=(c_0, \ldots, c_d) \in \mathbb{R}^{d+1}\\[0.5em] {\sum_{i=0}^d c_i\tau(\widetilde{\bold{U}}^{(n+1), i})}\neq 0}}{\min}\; \frac{\langle \widehat{\bold{V}}_{\bold{S}}^*\bold{c}, \widehat{\bold{V}}_{\bold{S}}^*\widehat{\bold{A}}\widehat{\bold{V}}_{\bold{S}}\widehat{\bold{V}}_{\bold{S}}^*\bold{c}\rangle}{\langle \widehat{\bold{V}}_{\bold{S}}^*\bold{c}, \widehat{\bold{\Sigma}}_{\bold{S}} \widehat{\bold{V}}_{\bold{S}}^*\bold{c}\rangle}.
\end{align*}
If the matrix $\widehat{\bold{\Sigma}}_{\bold{S}}$ is well-conditioned, then we further deduce that 
\begin{align*}
	\underset{\substack{\bold{c}=(c_0, \ldots, c_d) \in \mathbb{R}^{d+1}\\[0.5em] {\sum_{i=0}^d c_i\tau(\widetilde{\bold{U}}^{(n+1), i})}\neq 0}}{\min}\; \frac{\langle \widehat{\bold{V}}_{\bold{S}}^*\bold{c}, \widehat{\bold{U}}_{\bold{S}}^*\widehat{\bold{A}}\widehat{\bold{V}}_{\bold{S}}\widehat{\bold{V}}_{\bold{S}}^*\bold{c}\rangle}{\langle \widehat{\bold{V}}_{\bold{S}}^*\bold{c}, \widehat{\bold{\Sigma}}_{\bold{S}} \widehat{\bold{V}}_{\bold{S}}^*\bold{c}\rangle} &= \underset{0\neq \widehat{\bold{V}}_{\bold{S}}^*\bold{c}\in \mathbb{R}^{d+1}}{\min}\; \frac{\langle \widehat{\bold{V}}_{\bold{S}}^*\bold{c}, \widehat{\bold{V}}_{\bold{S}}^*\widehat{\bold{A}}\widehat{\bold{V}}_{\bold{S}}\widehat{\bold{V}}_{\bold{S}}^*\bold{c}\rangle}{\langle \widehat{\bold{V}}_{\bold{S}}^*\bold{c}, \widehat{\bold{\Sigma}}_{\bold{S}} \widehat{\bold{V}}_{\bold{S}}^*\bold{c}\rangle}\\
	&=\underset{0\neq \widehat{\bold{c}}\in \mathbb{R}^{d+1}}{\min}\; \frac{\langle \widehat{\bold{c}}, \widehat{\bold{V}}_{\bold{S}}^*\widehat{\bold{A}}\widehat{\bold{V}}_{\bold{S}}\widehat{\bold{c}}\rangle}{\langle \widehat{\bold{c}}, \widehat{\bold{\Sigma}}_{\bold{S}} \widehat{\bold{c}}\rangle},
\end{align*}
where the second equality is due to the fact that the eigenvectors of $\widehat{\bold{S}} $ form an orthonormal basis of $\mathbb{R}^d$. Consequently, any minimizer of Problem \eqref{eq:d_opt_lin_4} is an eigenvector associated with the \emph{lowest} eigenvalue of the eigenvalue problem
\begin{align*}
	\big(\widehat{\bold{\Sigma}}_{\bold{S}}^{-1/2}\widehat{\bold{V}}_{\bold{S}}^*\widehat{\bold{A}}\widehat{\bold{V}}_{\bold{S}}\widehat{\bold{\Sigma}}_{\bold{S}}^{-1/2}\big)\widehat{\bold{\Sigma}}_{\bold{S}}^{1/2}\widehat{\bold{c}}= \lambda \widehat{\bold{\Sigma}}_{\bold{S}}^{1/2}\widehat{\bold{c}}.
\end{align*}
Thus, the minimizer of the second-level minimization problem \eqref{eq:d_opt_lin_3} is given by $\sum_{i=0}^d c_i^*\tau\big(\widetilde{\bold{U}}^{(n+1), i}\big)$ with the vector $\bold{c}^*= (c_0^*, \ldots, c_d^*)$ being obtained by computing the lowest eigenpair of a generalized eigenvalue problem involving the matrix $\widehat{\bold{A}}\in \mathbb{R}^{d+1 \times d+1}$ and the eigendecomposition factors of the overlap matrix $\widehat{\bold{S}}$.

If, on the other hand, the diagonal matrix $\widehat{\bold{\Sigma}}_{\bold{S}}$ is rank-deficient, say, of rank $p < d+1$, then we can decompose $\widehat{\bold{\Sigma}}_{\bold{S}}$ and $\widehat{\bold{U}}_{\bold{S}}$ as
\begin{align*}
	\widehat{\bold{\Sigma}}_{\bold{S}} = \begin{bmatrix}
		\widehat{\bold{\Sigma}}_{\bold{S}}^+ & 0\\
		0 & 0
	\end{bmatrix}, \hspace{1.5mm} \widehat{\bold{U}}_{\bold{S}}= \begin{bmatrix} \widehat{\bold{V}}_{\bold{S}}^+ & \widehat{\bold{V}}_{\bold{S}}^0\end{bmatrix} ~ \text{ with } \widehat{\bold{\Sigma}}_{\bold{S}}^+\in \mathbb{R}^{p \times p} \text{ full-rank and } \widehat{\bold{V}}_{\bold{S}}^+ \in \mathbb{R}^{d+1 \times p}.
\end{align*}
Using now the matrix $\widehat{\bold{V}}_{\bold{S}}^+ \in \mathbb{R}^{d+1 \times p}$, we can construct the following $p$-dimensional orthogonal basis for the approximation space $\mathcal{Y}_n$:
\begin{align*}
	\left\{ \sum_{j=0}^{d} [\widehat{\bold{V}}_{\bold{S}}^+]_{j+1,k} \;\tau(\widetilde{\bold{U}}^{(n+1), j}) \colon \quad k\in \{1, \ldots, p\}  \right\}.
\end{align*}
Consequently, the second-level minimization problem \eqref{eq:d_opt_lin_3} can be expressed as
\begin{align}\label{eq:d_opt_lin_5}
	\underset{0\neq \bold{X} \in \mathcal{Y}_n}{\min}\; \frac{\langle \bold{X}, \bold{A}\bold{X}\rangle}{\Vert \bold{X}\Vert} =  \underset{\substack{\bold{b}=(b_1, \ldots, b_p) \in \mathbb{R}^p\\ \bold{b}\neq 0}}{\min}\; \frac{\langle \bold{b}, (\widehat{\bold{V}}_{\bold{S}}^+)^*\widehat{\bold{A}}\widehat{\bold{V}}_{\bold{S}}^+ \bold{b}\rangle}{\langle \bold{b}, \bold{\Sigma}_{\bold{S}}^+\bold{b}\rangle}.
\end{align}
Any minimizer of Problem \eqref{eq:d_opt_lin_5} is an eigenvector associated with the \emph{lowest} eigenvalue of the eigenvalue problem
\begin{align}\label{eq:aux_eig_final_sym}
	\big( (\widehat{\bold{\Sigma}}_{\bold{S}}^+)^{-1/2} (\widehat{\bold{V}}_{\bold{S}}^+)^*\widehat{\bold{A}}\widehat{\bold{V}}_{\bold{S}}(\widehat{\bold{\Sigma}}_{\bold{S}}^+)^{-1/2}\big)(\widehat{\bold{\Sigma}}_{\bold{S}}^+)^{1/2} \widehat{\bold{b}}= \lambda (\widehat{\bold{\Sigma}}_{\bold{S}}^+)^{1/2}\widehat{\bold{b}}.
\end{align}
Thus, the sought minimizer of the second-level minimization problem \eqref{eq:d_opt_lin_3} is given by
\begin{align*}
	\sum_{k=1}^p \widehat{b}^*_k \sum_{j=0}^d [\widehat{\bold{V}}_{\bold{S}}^+]_{j+1, k} \; \tau(\widetilde{\bold{Y}}^{(n+1), j}), 
\end{align*}
with the vector $\widehat{\bold{b}}^*= (b^*_1, \ldots, b^*_p)$ being obtained once again by computing the lowest eigenpair of an eigenvalue problem involving the matrix $\widehat{\bold{A}}\in \mathbb{R}^{d+1 \times d+1}$ and the eigendecomposition factors of the overlap matrix $\widehat{\bold{S}}$. We remark that in the case when the diagonal matrix $\widehat{\bold{\Sigma}}_{\bold{S}}$ is ill-conditioned but not rank-deficient, we can similarly decompose $\widehat{\bold{\Sigma}}_{\bold{S}}$ and $\widehat{\bold{U}}_{\bold{S}}$ using a tolerance parameter $\epsilon >0$.

We end this section by remarking that very similar arguments can be used to solve the second-level minimization problem \eqref{eq:d_optim} for quadratic energy functionals $\mathcal{J}\colon U \subset \mathbb{R}^{n_1 \times \ldots \times n_d} \rightarrow \mathbb{R}$ of the form \eqref{eq:energy_lin}, i.e., if $\mathcal{J}$ is defined as
\begin{align}
	\forall \bold{X}\in  U=\mathbb{R}^{n_1 \times n_2 \times \ldots n_d}\colon \qquad        \mathcal{J}(\bold{X})= \frac{1}{2} \langle \bold{X}, \bold{A}\bold{X}\rangle- \langle \bold{F}, \bold{X}\rangle,
\end{align}
with $\bold{A}\colon \mathbb{R}^{n_1 \times n_2 \times \ldots n_d} \rightarrow \mathbb{R}^{n_1 \times n_2 \times \ldots n_d}$ symmetric and positive and $\bold{F} \in \mathbb{R}^{n_1 \times n_2 \times \ldots n_d}$.
{In this case, the local updates of A2DMRG (Step 2 in Algorithm \ref{alg:PALS}) will consist of solving linear systems involving an orthogonal projection of $\bold{A}$ (see, e.g., \cite{holtz2012alternating} for details), while the second level minimization (Step 3 in Algorithm \ref{alg:PALS}) will consist of the solution to a linear system involving the projection of $\bold{A}$ on the coarse $d+1$ or $d$-dimensional space spanned by the local updates.}

\subsection{Application to Ground State Calculations in Quantum Chemistry}\label{sec:3.2}

A fundamental problem in computational quantum chemistry is the numerical approximation of the lowest eigenvalue (or so-called ground state energy) of the electronic Hamiltonian-- a self-adjoint operator acting on an $L^2$-type Hilbert space of antisymmetric functions. In this section, we briefly recall how this energy minimization problem can be reformulated using the second quantization formalism to fit the framework of tensor train approximations described in Section \ref{sec:2}. Subsequently, in Section \ref{sec:4}, we will demonstrate the performance of the A2DMRG \Cref{alg:PALS} for the resolution of this energy minimization problem. For simplicity, our exposition ignores the spin variable but all numerical tests in Section \ref{sec:4} do, in fact, include spin.

We begin by considering a finite collection of $L^2(\R^3)$-orthonormal basis functions $\{\phi_\mu\}_{\mu=1}^d \subset H^1(\R^3)$ and we introduce, for each $\ell \in \{1, \ldots, d\}$, the set $\mathcal{B}_{\ell}$ of antisymmetric functions defined as
\begin{align*}
	\mathcal{B}_{\ell}:=& \left\{\Phi \in H^1(\R^{3\ell})\colon \Phi:= \phi_{\mu_1} \wedge \phi_{\mu_2} \ldots \wedge \ldots \phi_{\mu_{\ell}} \colon 1\leq \mu_{1} < \mu_{2} < \ldots  < \mu_{\ell}\leq d \right\}.
\end{align*}
Here, $\wedge$ denotes the antisymmetric tensor product (see, e.g., \cite[Chapter 1]{rohwedder2010analysis}) so that
\begin{align*}
	\Phi= \phi_{\mu_1} \wedge \ldots \wedge  \phi_{\mu_{\ell}} \iff \Phi(\bold{x}_1, \ldots, \bold{x}_\ell)=\frac{1}{\sqrt{\ell!}} \sum_{\pi \in {\mS}(\ell)} (-1)^{\rm sgn(\pi)} \overset{\ell}{\underset{i=1}{\otimes}} \phi_{\mu_i}(\bold{x}_{\pi(i)}), 
\end{align*}
with ${\mS}(\ell)$ the permutation group of order $\ell$ and ${\rm sgn(\pi)} $  the signature of $\pi~\in~{\mS}(\ell)$. 

Denoting now $\mathcal{V}_{\ell}:=\text{\rm span}\mathcal{B}_{\ell}$ for all $\ell \in \{1, \ldots, d\}$ and setting $\mathcal{V}_0 = \mathbb{R}$ by convention, we can introduce the so-called Fermionic \textit{Fock space} as the $2^d$-dimensional Hilbert space given by
\begin{align}
	\mathcal{F}_d:= \bigoplus_{\ell=0}^d \mathcal{V}_{\ell}, \qquad \text{equipped with the natural ~$\bigoplus_{\ell=0}^d L^2(\R^{3\ell})$ norm.}
\end{align}

Of particular importance are the so-called fermionic creation and annihilation operators acting on the Fock space $\mathcal{F}_d$ (see, e.g, \cite[Chapter 1]{helgaker2013molecular} for further details). For any $\ell \in \{1, \ldots, d\}$, the fermionic creation operator $a_{\ell}^{\dagger}\colon \mathcal{F}_{d}\rightarrow \mathcal{F}_d$ is defined as
\begin{align*}
	\forall \Phi \in \mathcal{F}_d \colon \qquad a_{\ell}^{\dagger} \Phi := \Phi \wedge \phi_\ell,
\end{align*}
The fermionic annihilation operator $a_{\ell}\colon \mathcal{F}_{d}\rightarrow \mathcal{F}_d$ is then the $\overset{d}{\underset{\ell=0}{\bigoplus}} L^2(\R^{3\ell})$ adjoint of $a_{\ell}^{\dagger}$. 

Given now a molecular system composed of $M$ nuclei with charges $\{Z_{\alpha}\}_{\alpha=1}^M$ and nuclear positions $\{R_{\alpha}\}_{\alpha=1}^M \subset R^{3}$, the corresponding \textit{second quantized} electronic Hamiltonian $\mathcal{H} \colon \mathcal{F}_d \rightarrow \mathcal{F}_d$ is given by
\begin{align*}
	\mathcal{H} &= \sum_{i,j =1}^d {h}_{ij}a_i^{\dagger} a_j +  {\frac{1}{2}}\sum_{i,j, k, \ell =1}^d {V}_{ijk \ell}a_i^{\dagger} a_j^\dagger a_\ell a_{k},
	\intertext{where, for all $i, j, k, \ell \in \{1, \ldots, d\}$, it holds that}
	{h}_{ij}:= &{\frac{1}{2}}\int_{\R^3} \nabla \phi_i(\bold{x}) \cdot \nabla \phi_j(\bold{x})\; d \bold{x} -    \int_{\R^3} \sum_{\alpha=1}^M \frac{Z_{\alpha}}{\vert R_{\alpha}-\bold{x}\vert}\phi_i(\bold{x}) \phi_j(\bold{x})\; d \bold{x},\\[1em]
	\quad {V}_{ijk\ell}:= & \int_{\R^3}\int_{\R^3} \frac{\phi_i(\bold{x}) \phi_j(\bold{y}) \phi_{k}(\bold{x}) \phi_{\ell}(\bold{y})}{\vert \bold{x}-\bold{y}\vert}\; d \bold{x} d\bold{y}.
\end{align*}

The ground-state problem for a molecular system composed precisely of $N$ electrons consists of finding the lowest eigenvalue of the second quantized Hamiltonian $\mathcal{H}$, subject to the constraint that the corresponding eigenvector describes an $N$-electron system. This constraint can be imposed by requiring that the sought energy minimizer be an eigenfunction of the so-called particle number operator $\mathcal{N}= \sum_{j=1}^d a_j^{\dagger}a_j\colon \mathcal{F}_d \rightarrow \mathcal{F}_d$ with eigenvalue precisely $N$. In other words, we seek
\begin{align}\label{eq:GS_problem}
	\mathcal{E}_{\rm GS} = \min_{0\neq\Phi \in \mathcal{F}_d} \frac{\langle \Phi, \mathcal{H}\Phi\rangle_{\mathcal{F}_d}}{\langle \Phi, \Phi\rangle_{\mathcal{F}_d} } \quad \text{subject to the constraint } \mathcal{N}\Phi:= \sum_{j=1}^d a_j^{\dagger}a_j \Phi= N\Phi.
\end{align}
{The key insight that now allows for the numerical resolution of the minimization problem \eqref{eq:GS_problem} in tensor format is that
	\begin{align*}
		\exists \text{ a unitary isomorphism } \iota \text{ from the Fock space } \mathcal{F}_d ~\text{ to the tensor space }  \mathbb{R}^{2^d}.  
	\end{align*}
	Consequently, the minimization problem \eqref{eq:GS_problem}, which is a priori posed on the Fock space $\mathcal{F}_d$ can be reformulated over the space $\mathbb{R}^{2^d}=\bigotimes_{j=1}^d \mathbb{R}^2$ of order-$d$ tensors:
	\begin{align}\label{eq:GS_problem_tensor}
		\mathcal{E}_{\rm GS} = \min_{0\neq \bold{X} \in \mathbb{R}^{2^d}} \frac{\langle \bold{X}, \iota \circ\mathcal{H} \circ\iota^{-1}\bold{X}\rangle}{\langle \bold{X}, \bold{X}\rangle} \qquad \text{subject to the constraint } \iota \circ\mathcal{N} \circ\iota^{-1}\bold{X}= N \bold{X}.
	\end{align}
	We refer to, e.g., \cite{faulstich2019analysis, holtz2012alternating} for more details.

	The minimization problem \eqref{eq:GS_problem_tensor} can now be approximated using the tensor train format and alternating minimization algorithms described in detail in Sections \ref{sec:2} and~\ref{sec:3}. Note, however, that these sections did not include any discussion on the imposition of the particle number constraint. Let us, therefore, briefly describe how this additional constraint can be dealt with in the tensor train format.
	
	It has been shown by Bachmayr, Götte, and Pfeffer \cite{bachmayr2022particle} that eigenfunctions of the particle number operator $\iota \circ \mathcal{N}\circ \iota^{-1} \rightarrow  \mathbb{R}^{2^d} \times \mathbb{R}^{2^d}$ are characterized by a block-sparse structure in their tensor-train decomposition. Therefore, to approximate the minimization problem \eqref{eq:GS_problem_tensor} on a tensor manifold $\mathcal{T}_{\bold{r}}, \bold{r}\in \mathbb{N}^{d-1}$ using, e.g., classical one-site DMRG, we first initialize the algorithm with some $\bold{U}^{(0)} \in \mathcal{U}_\bold{r}$ that possess the requisite block-sparse structure. This guarantees that $\tau(\bold{U}^{(0)})$ is an eigenfunction of $\iota \circ \mathcal{N}\circ \iota^{-1}$ with eigenvalue $N$. Next, assume that we are performing micro-iteration $k +1$ of the one-site DMRG algorithm, and that the previous iterate $\bold{U}^{(k)} \in \mathcal{U}_{\bold{r}}$ is $k+1$-orthogonal and possesses the requisite block-sparse structure. As an alternative to Lemma \ref{lem:ALS_micro}, we can characterize the $k+1$ DMRG micro-iterate $\bold{U}^{(k+1)} \in \mathcal{U}_{\bold{r}}$ through projection operators on the tangent space $\mathbb{T}_{\tau(\bold{U}^{(k)})} \mathcal{T}_{\bold{r}}$. 
	
	To this end, for each $j \in \{1, \ldots, d \}$, let us denote by $\bold{U}^{(k)}_{\leq j}  \in \mathbb{R}^{n_1\ldots n_{i} \times r_i}$ and $\bold{U}^{(k)}_{\geq j}  \in \mathbb{R}^{r_{j-1}\times n_{j}\ldots n_{d}}$ the matrices obtained by contracting and reshaping  the first $j$ TT cores and the last $d-j+1$ TT cores, respectively, of $\bold{U}^{(k)}$. We then introduce, for each $j \in \{1, \ldots, d\}$ the projection operators $\mathbb{P}_{\leq j}, \mathbb{P}_{\geq j}, \mathbb{P}_{j} \colon \mathbb{R}^{n_1 \times \ldots \times n_d} \rightarrow \mathbb{T}_{\tau(\bold{U}^{(k)})} \mathcal{T}_{\bold{r}}$ as
	\begin{align*}
		\mathbb{P}_{\leq j}(\bold{Z})= \text{\rm Ten}_j \big(\bold{U}^{(k)}_{\leq j} \bold{Z}^{\langle j\rangle}\big), \quad \mathbb{P}_{\geq j}(\bold{Z})= \text{\rm Ten}_j \big(\bold{Z}^{\langle j-1\rangle}\bold{U}^{(k)}_{\geq j}\big), \quad \text{and} \quad \mathbb{P}_{j}=\mathbb{P}_{\leq j-1} \mathbb{P}_{\geq j+1}.
	\end{align*}
	Here, $\text{\rm Ten}_j$ is a reshape operation that takes a matrix in $\mathbb{R}^{n_1\ldots n_j \times n_{j+1}\ldots n_d}$ and reshapes it into a tensorial element of $\mathbb{R}^{n_1\times \ldots \times n_d}$, and we adopt the convention that $\mathbb{P}_{\leq 0}= \mathbb{P}_{\geq d+1}=1$. Recalling Lemma \ref{lem:ALS_micro}, it is now possible to show (see, e.g., \cite{bachmayr2022particle}) that
	\begin{align*}
		\tau(\bold{U}^{(k+1)}) \text{ is the lowest eigenfunction of }  \mathbb{P}_{j}\circ \iota \circ\mathcal{H} \circ\iota^{-1} \circ \mathbb{P}_{j}^*.
	\end{align*}
	
	On the other hand, it has been shown in \cite{bachmayr2022particle} that the projector $ \mathbb{P}_{j}$ commutes with the particle number operator $\iota \circ \mathcal{N}\circ \iota^{-1}$, which in turn commutes with the Hamiltonian $\iota \circ \mathcal{H}\circ \iota^{-1}$. It therefore follows that $\tau(\bold{U}^{(k+1)})$ is also an eigenfunction of $\iota \circ \mathcal{N}\circ \iota^{-1}$, and hence, all one-site DMRG micro-iterates satisfy the particle-number constraint. 
	
	A similar argument applies in the case of the two-site DMRG algorithm without rank-truncation. If rank-truncation \textit{is} performed (as in Algorithm \ref{alg:ALS}), then care must be taken in the application of the SVD. Essentially, as explained in \cite[Algorithm 3]{bachmayr2022particle}, the truncated SVD is not applied directly on the relevant TT cores but rather on the individuals non-zero blocks within the block structure of each TT core.
	
	Finally, we turn to the question of particle-number preservation in the A2DMRG \Cref{alg:PALS}. As in the classical DMRG algorithms, we begin with an initialization $\bold{U}^{(0)} \in \mathcal{U}_\bold{r}$ that possesses the requisite block-sparse structure which guarantees that $\tau(\bold{U}^{(0)})$ is an eigenfunction of $\iota \circ \mathcal{N}\circ \iota^{-1}$ with eigenvalue $N$. We then perform $d$ (resp. $d-1$ in the two-site case) parallel DMRG micro-iterations resulting in updated tensors $\{\tau(\bold{V}^{(j)})\}_{j=1}^d$ (resp. $\{\tau(\bold{W}^{(j, j+1)})\}_{j=1}^{d-1}$). As argued above, each DMRG micro-iteration with an appropriately orthogonalized and particle-preserving initialization results in a particle-preserving tensor. Therefore, all tensors $\{\tau(\bold{V}^{(j)})\}_{j=1}^d$ (resp. $\{\tau(\bold{W}^{(j, j+1)})\}_{j=1}^{d-1}$) are eigenfunctions of $\iota \circ \mathcal{N}\circ \iota^{-1}$ with eigenvalue $N$ and so are linear combinations of such tensors. It therefore suffices to consider the effect of the compression step (Step 4 in \Cref{alg:PALS}) which projects a linear combination $c^*_0\tau(\bold{U}^{(0)})+\sum_{j=1}^d c_j^*\tau(\bold{V}^{(j)}), ~ \{c_j^*\}_{j=0}^d \subset \mathbb{R}$ (resp. $d_0^*\tau(\bold{U}^{(0)})+\sum_{j=1}^{d-1}d_j^*\tau(\bold{W}^{(j, j+1)}), ~ \{d_j^*\}_{j=0}^{d-1} \subset \mathbb{R}$) onto some low-rank tensor manifold $\mathcal{T}_{\widetilde{\bold{r}}}$. Such a compression step can be performed using either the particle-preserving version of the TT-rounding algorithm \cite[Algorithm 3]{bachmayr2022particle}, or more-generally, by solving a best-approximation problem using standard tensor train algorithms (such as alternating least squares), which are known to be particle-preserving \cite[Section 6]{bachmayr2022particle}.}

\section{Numerical Experiments}\label{sec:4}

For the purpose of the subsequent numerical study, we consider the tensor formulation of the ground state energy minimization problem \eqref{eq:GS_problem_tensor} in the case of several so-called \textit{strongly correlated} molecular systems. This includes linear chains of hydrogen atoms, i.e., molecular systems of the form H$_{2p}$ for $p \in \{3, 4, 5, 6\}$ as well as the C$_2$ and N$_2$ dimers. The bond lengths in each Hydrogen chain and the C$_2$ dimer are set to one Angstrom while the bond length for the N$_2$ dimer is set to $1.0976$ Angstrom. For each molecular system, the $L^2(\R^3)$-orthonormal single-particle basis $\{\phi_j\}_{j=1}^d$ is obtained by diagonalizing the Hartree-Fock (or Mean-Field) operator in the minimal STO-3G atomic basis. This implies, in particular, the relation $d=4p$ for each $p \in \{3, \ldots, 6\}$ and $d=20$ for C$_2$ and N$_2$ so that our numerical experiments involve tensors of minimal order $d=12$ in the case of H$_{6}$ and maximal order $d=24$ in the case of H$_{12}$. The input data, i.e., the second quantized electronic Hamiltonian, and the reference minimal energies are obtained from the PySCF Python module for quantum chemistry \cite{sun2020recent} while the DMRG simulations themselves are run using an in-house Julia module developed and described in \cite{badreddine2024leveraging}. 

\begin{figure}[h!]
	\centering
	\begin{subfigure}[t]{0.49\textwidth}
		\centering
		\includegraphics[width=\textwidth]{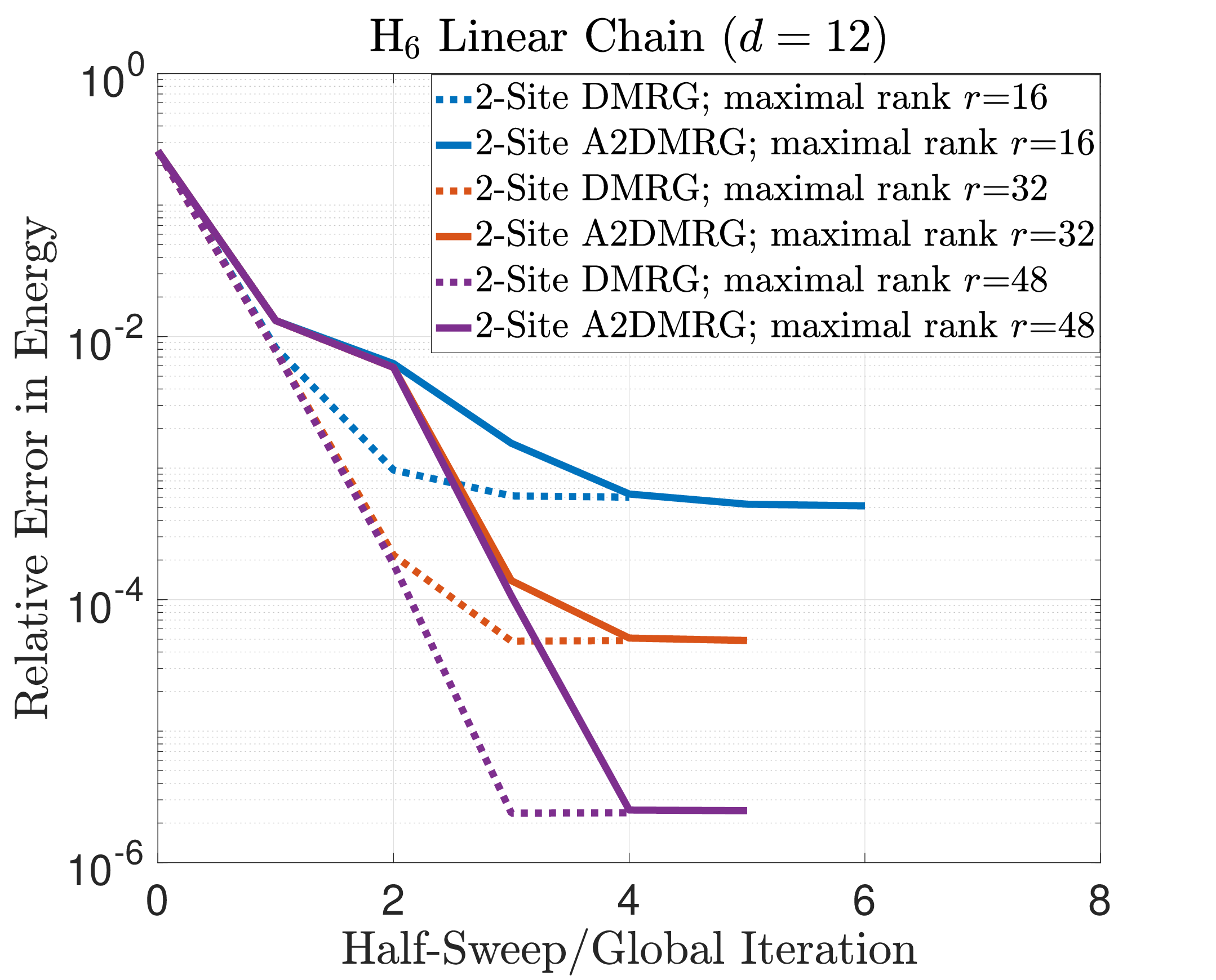} 
	\end{subfigure}\hfill
	\begin{subfigure}[t]{0.49\textwidth}
		\centering
		\includegraphics[width=\textwidth]{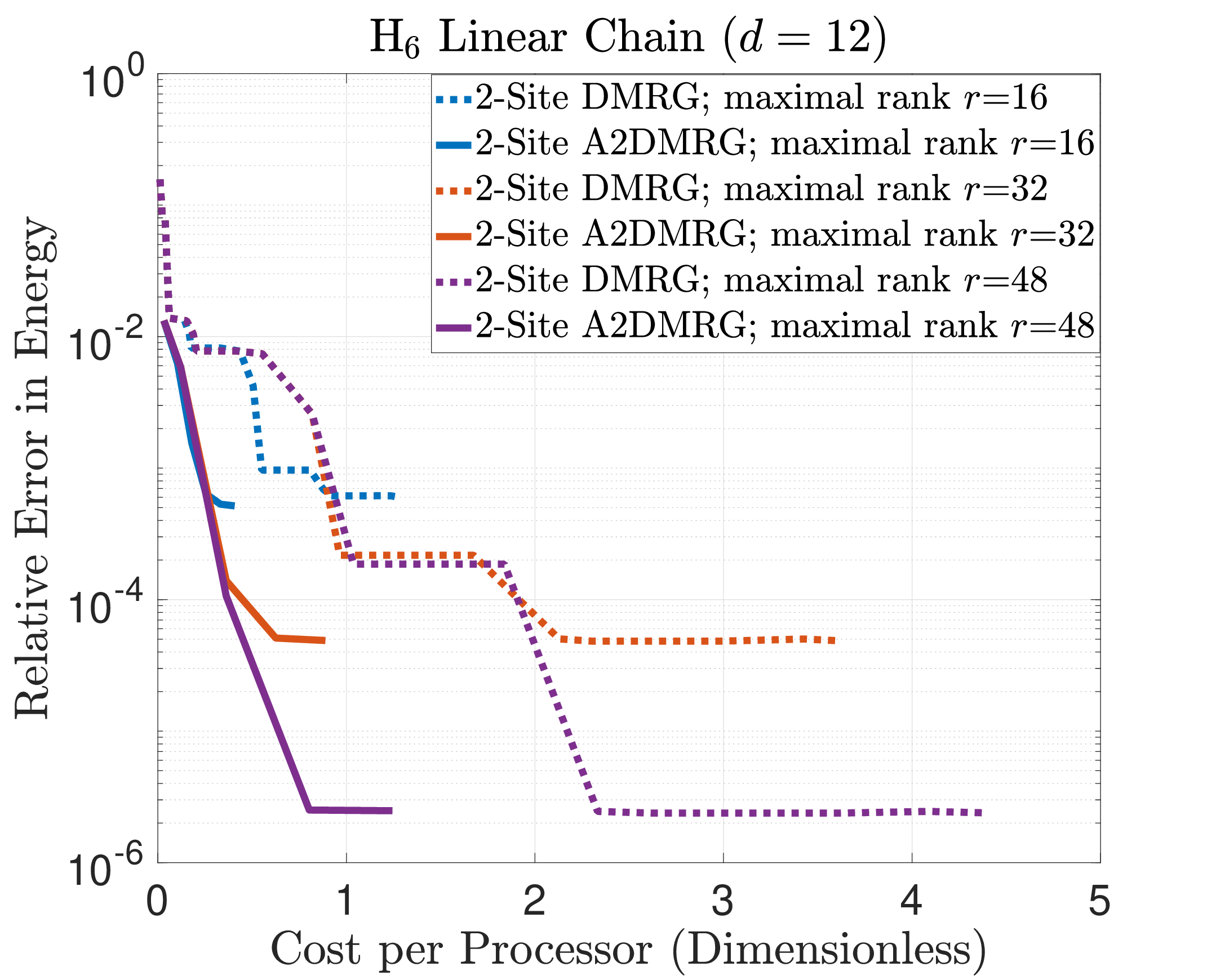} 
	\end{subfigure}\hfill
	\begin{subfigure}[t]{0.49\textwidth}
		\centering
		\includegraphics[width=\textwidth]{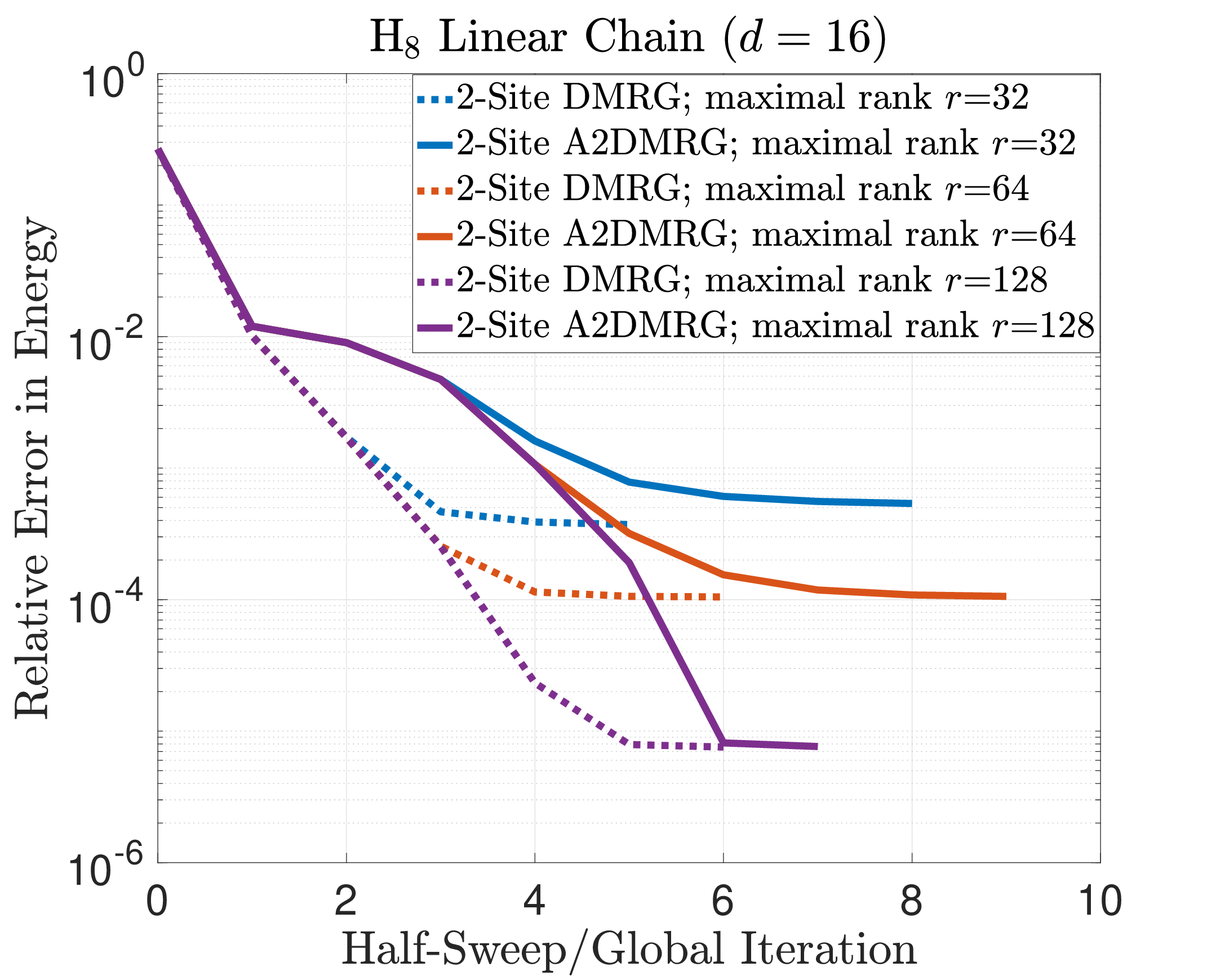} 
	\end{subfigure}\hfill
	\begin{subfigure}[t]{0.49\textwidth}
		\centering
		\includegraphics[width=\textwidth]{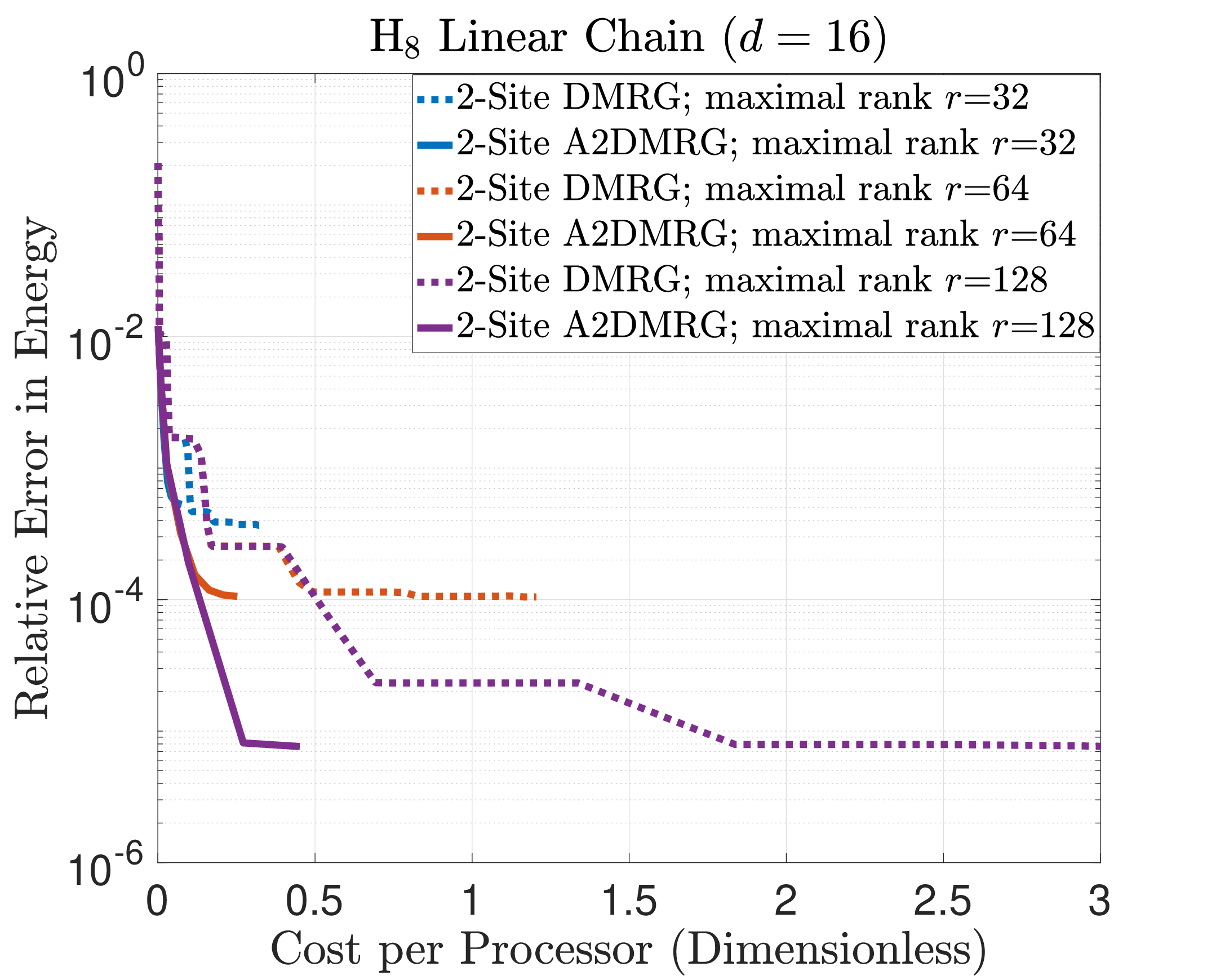} 
	\end{subfigure}
	\caption{Convergence plot of the classical two-site DMRG and two-site A2DMRG algorithms for a linear, stretched H$_{6}$ and H$_8$ chain. The tolerances for the eigensolvers and the truncated SVD were set to $10^{-6}$ and the algorithms were run until the relative energy difference between successive half-sweeps/global iterations was less than $10^{-6}$.}
	\label{fig:1}
\end{figure}

To begin with, we simply implement the two-site DMRG Algorithm \ref{alg:ALS} and the two-site A2DMRG \Cref{alg:PALS} to compute the ground states of these molecular systems, this choice being motivated by the fact that two-site DMRG is, by far, the popular choice in the quantum chemistry community. In order to maintain computational tractability, we impose rank truncation after each DMRG micro-iteration with maximal ranks varying from $\bold{r}= 16$ to $\bold{r}=512$. In all cases, the Lanczos algorithm is used to compute solutions to the DMRG micro-iterations and the second-level minimization problem in Algorithm \eqref{alg:PALS} with initialization vectors provided by the converged Lanczos results from the previous global iterations. Both the classical and additive DMRG algorithms are initiated with identical random initializations possessing very small rank parameters.

\begin{figure}[t]
	\centering
	\begin{subfigure}[t]{0.49\textwidth}
		\centering
		\includegraphics[width=\textwidth]{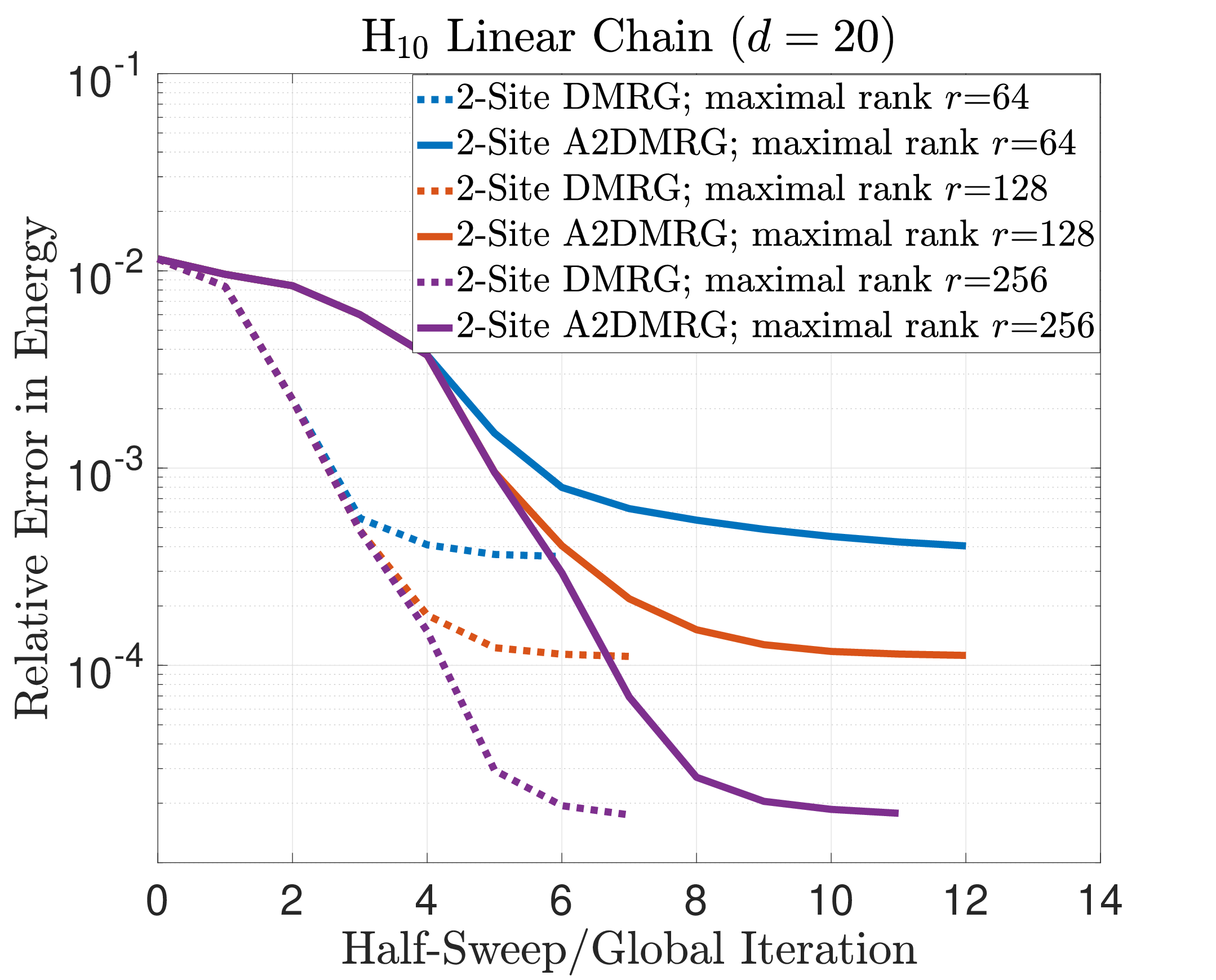} 
	\end{subfigure}\hfill
	\begin{subfigure}[t]{0.49\textwidth}
		\centering
		\includegraphics[width=\textwidth]{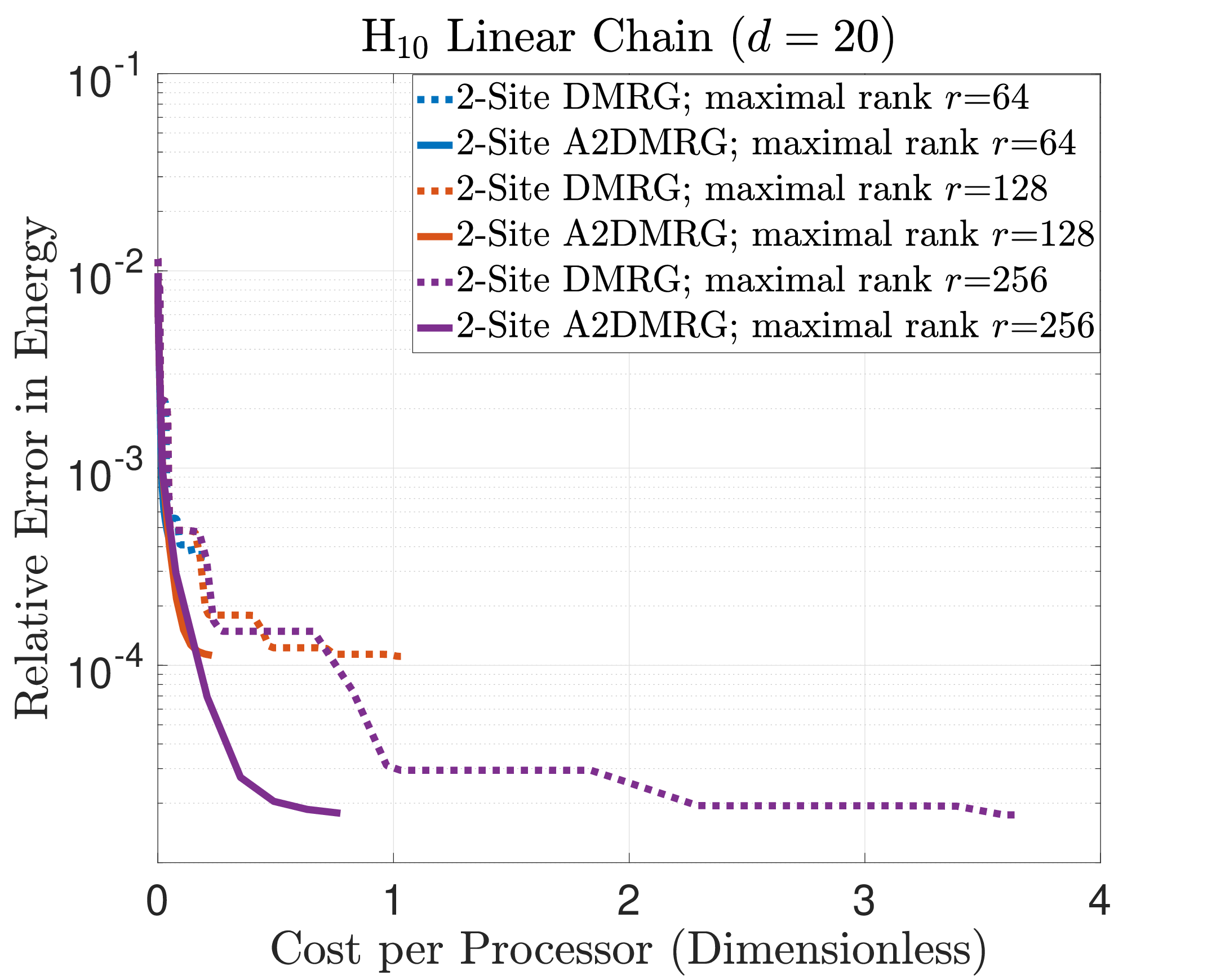} 
	\end{subfigure}\hfill
	\centering
	\begin{subfigure}[t]{0.49\textwidth}
		\centering
		\includegraphics[width=\textwidth]{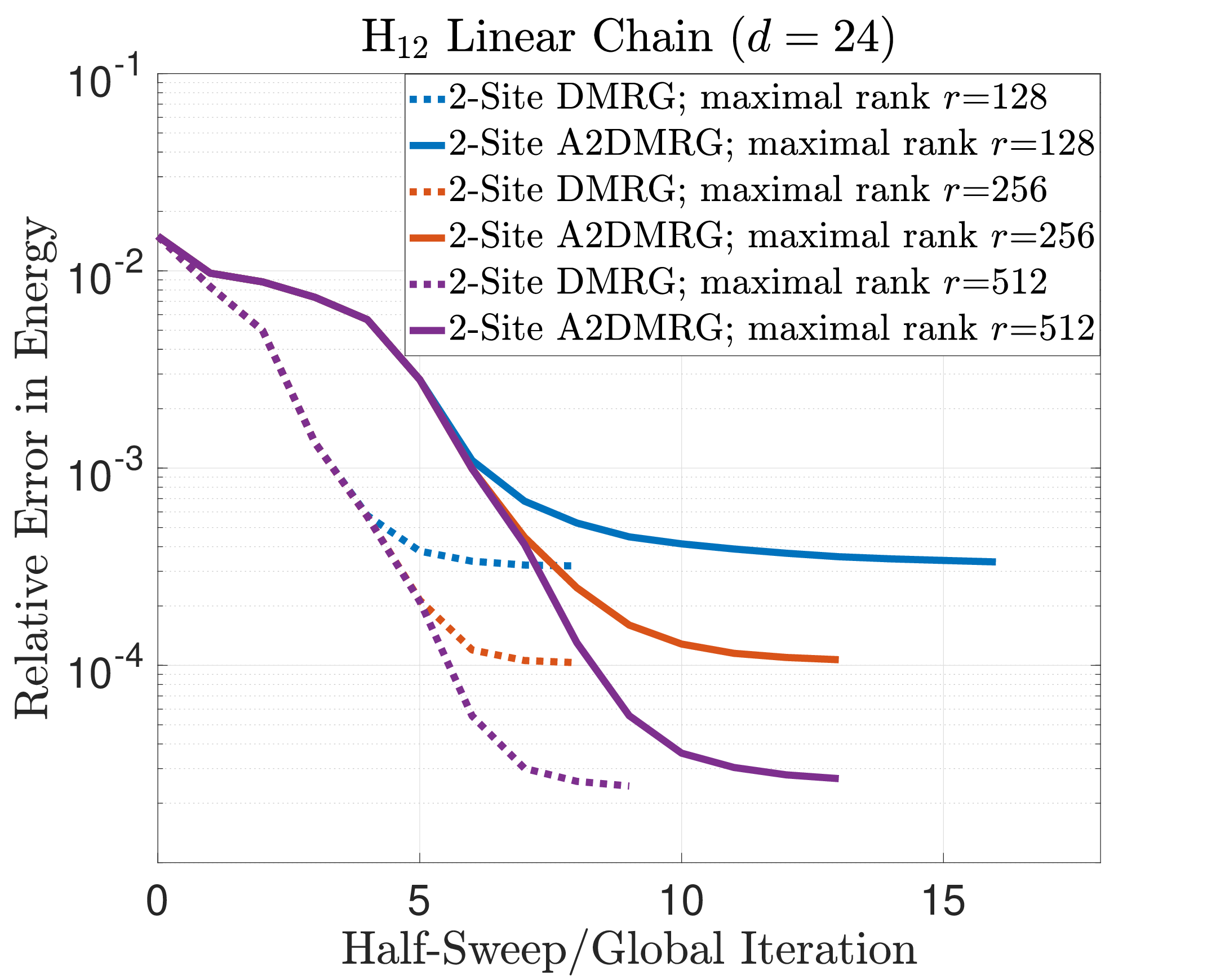} 
	\end{subfigure}\hfill
	\begin{subfigure}[t]{0.49\textwidth}
		\centering
		\includegraphics[width=\textwidth]{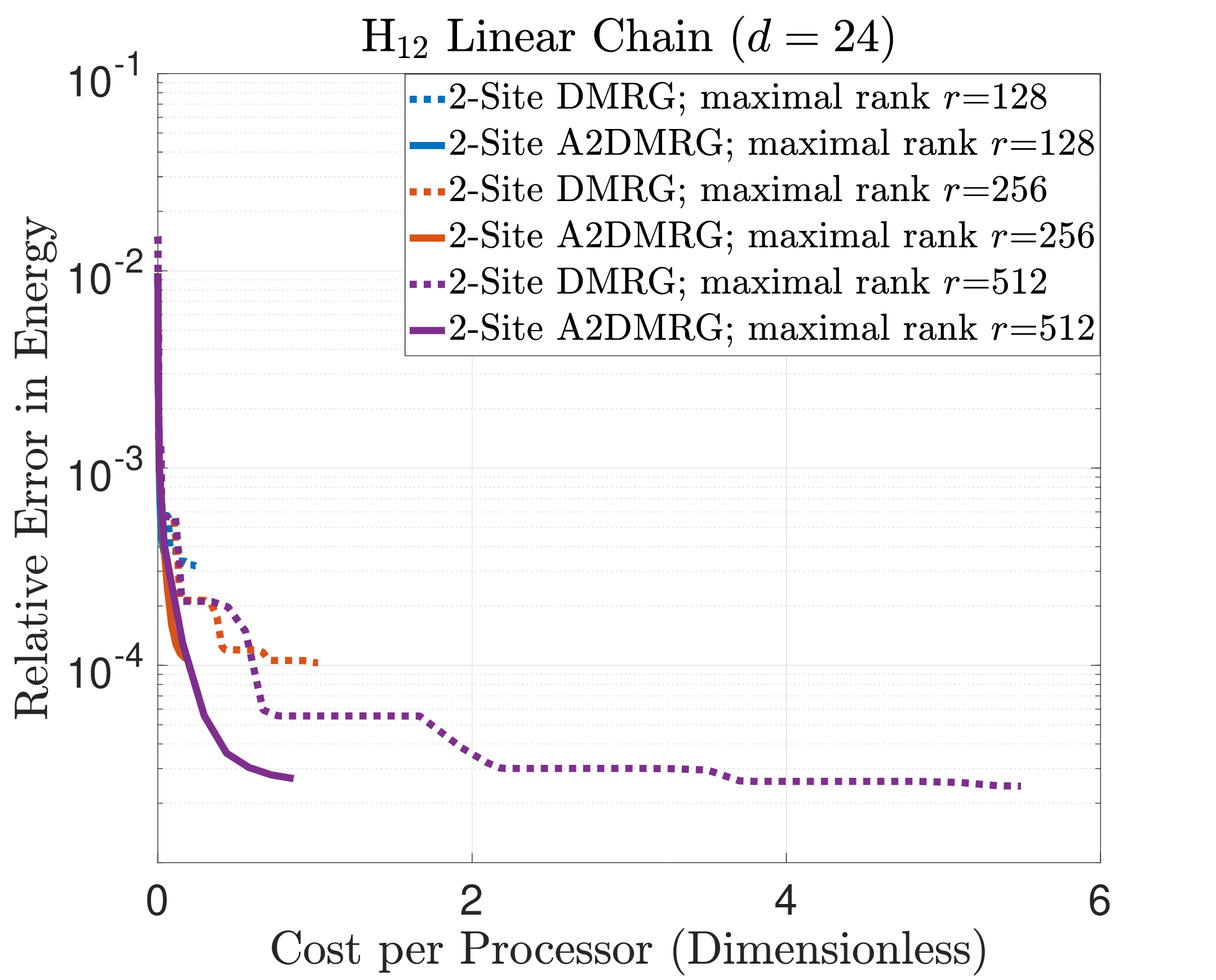} 
	\end{subfigure}
	\caption{Convergence plot of the classical two-site DMRG and two-site A2DMRG algorithms for a linear, stretched H$_{10}$ and H$_{12}$ chain. The tolerances for the eigenvalue solvers and the truncated SVD were set to $10^{-6}$ and the algorithms were run until the relative energy difference between successive half-sweeps/global iterations was smaller than $10^{-6}$.}
	\label{fig:2}
\end{figure}

Figures \ref{fig:1}-\ref{fig:3} display the resulting error plots as a function of either the number of global iterations or the (dimensionless) \emph{cost per processor}. The term global iteration refers to one half-sweep of the classical DMRG algorithm or one iteration of the A2DMRG \Cref{alg:PALS}, both of which consists of solving exactly $d-1$ local minimization problems. The cost per processor is computed by evaluating, through the course of each global iteration, the number and flop-count of tensor contractions, matrix-vector products, and linear algebra operations such as QR and singular value decompositions. In the case of the classical DMRG algorithm, which is sequential in nature, we assume that all operations are performed on a single processor. In the case of the A2DMRG \Cref{alg:PALS}, we assume that Steps 2 and 3, i.e., the DMRG micro-iterations and the solution of the second-level minimization problems are computed as much as possible in parallel. More precisely, we assume that each of the $(d-1)$ DMRG micro-iterations in Step 2 of \Cref{alg:PALS} is performed on an independent processor, and each entry of the $d\times d$ symmetric coarse matrices appearing in the second-level minimization problem (Recall Equation \ref{eq:aux_eig_final_sym}) is computed on an independent processor. The cost per processor is then computed by taking the processor with the highest cost for Steps 2 and 3, and adding to it the cost of Steps 1 and 4 of \Cref{alg:PALS}, the latter two steps being sequential in nature.

\begin{figure}[t]
	\centering
	\begin{subfigure}[t]{0.49\textwidth}
		\centering
		\includegraphics[width=\textwidth]{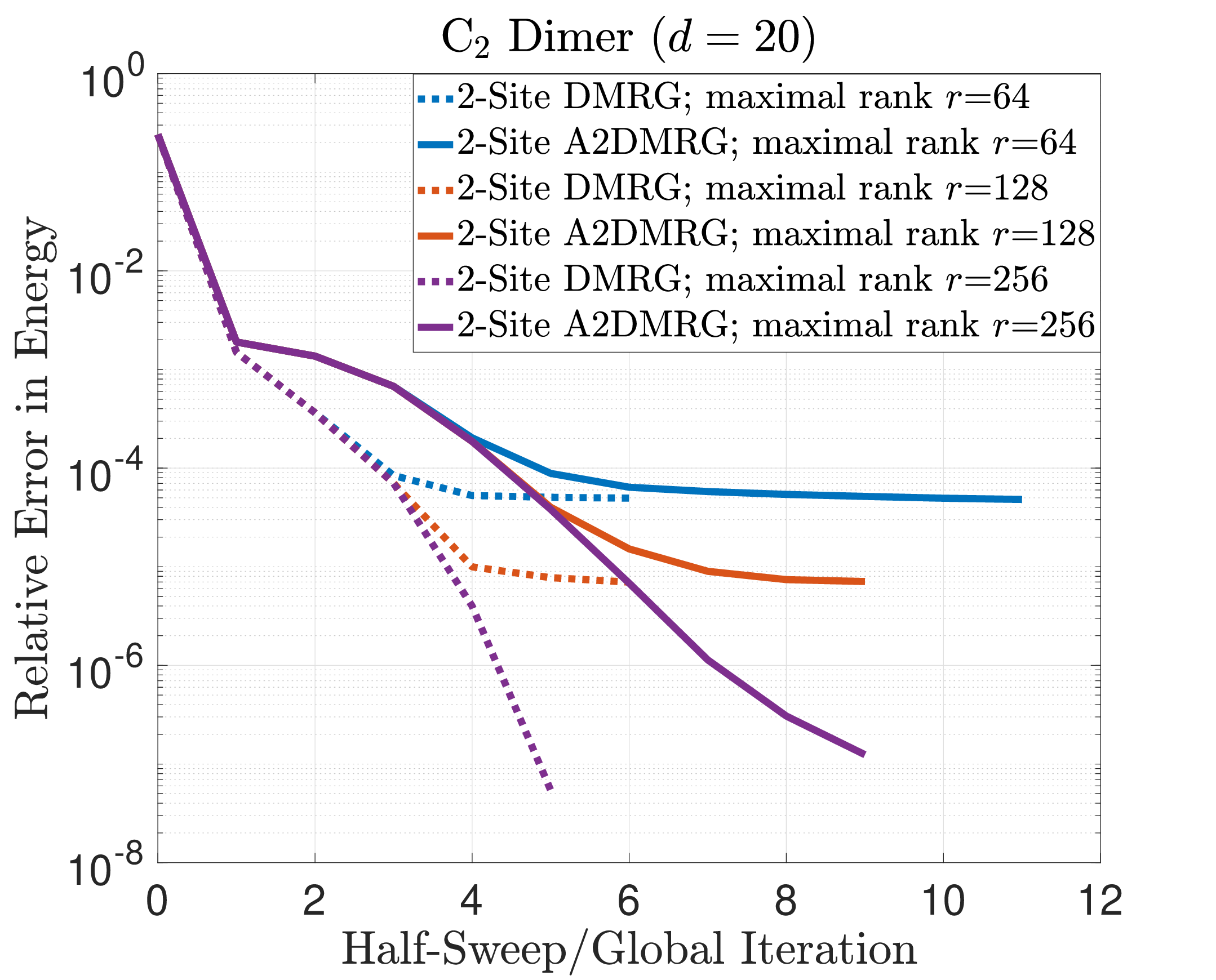} 
	\end{subfigure}\hfill
	\begin{subfigure}[t]{0.49\textwidth}
		\centering
		\includegraphics[width=\textwidth]{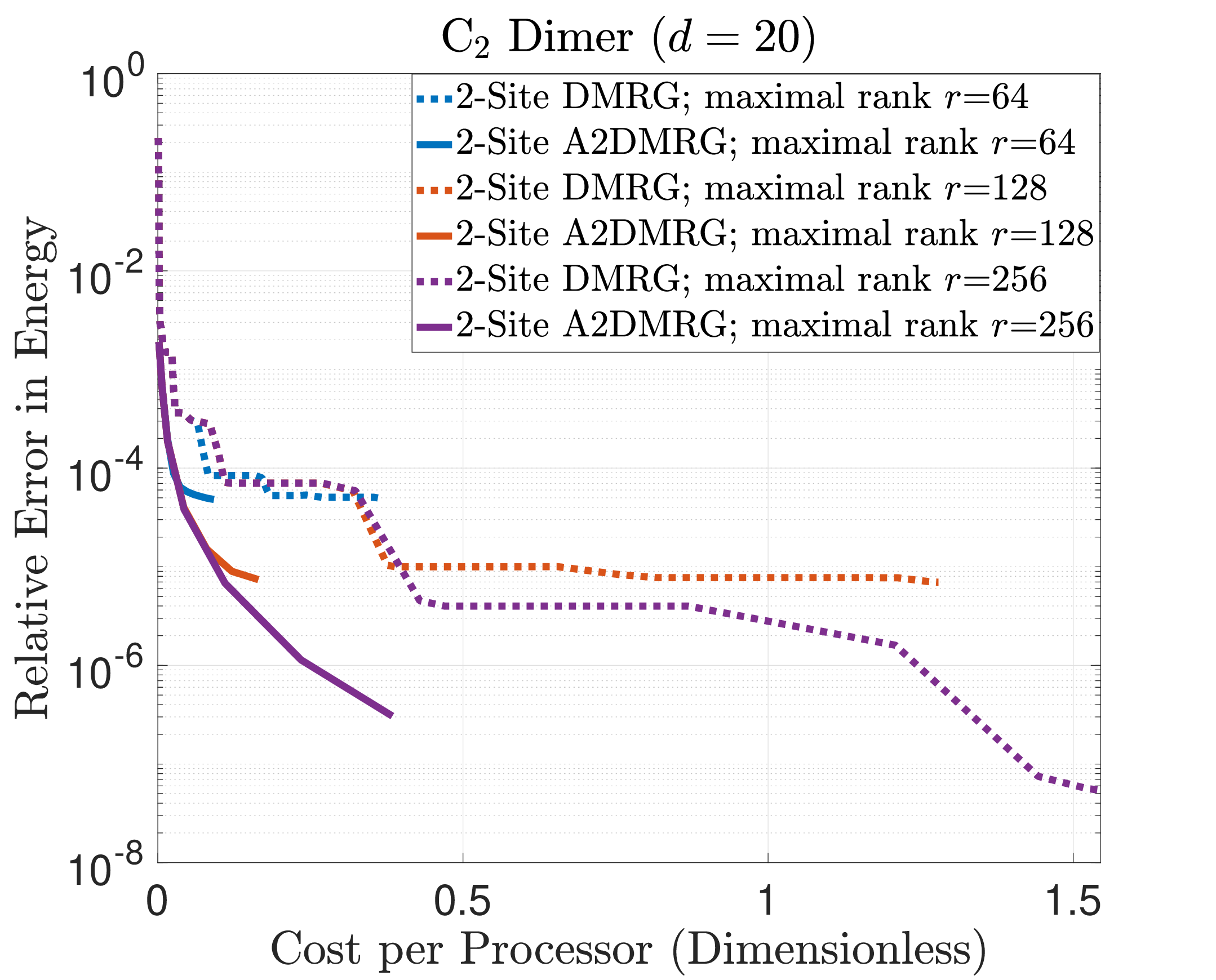} 
	\end{subfigure}\hfill
	\centering
	\begin{subfigure}[t]{0.49\textwidth}
		\centering
		\includegraphics[width=\textwidth]{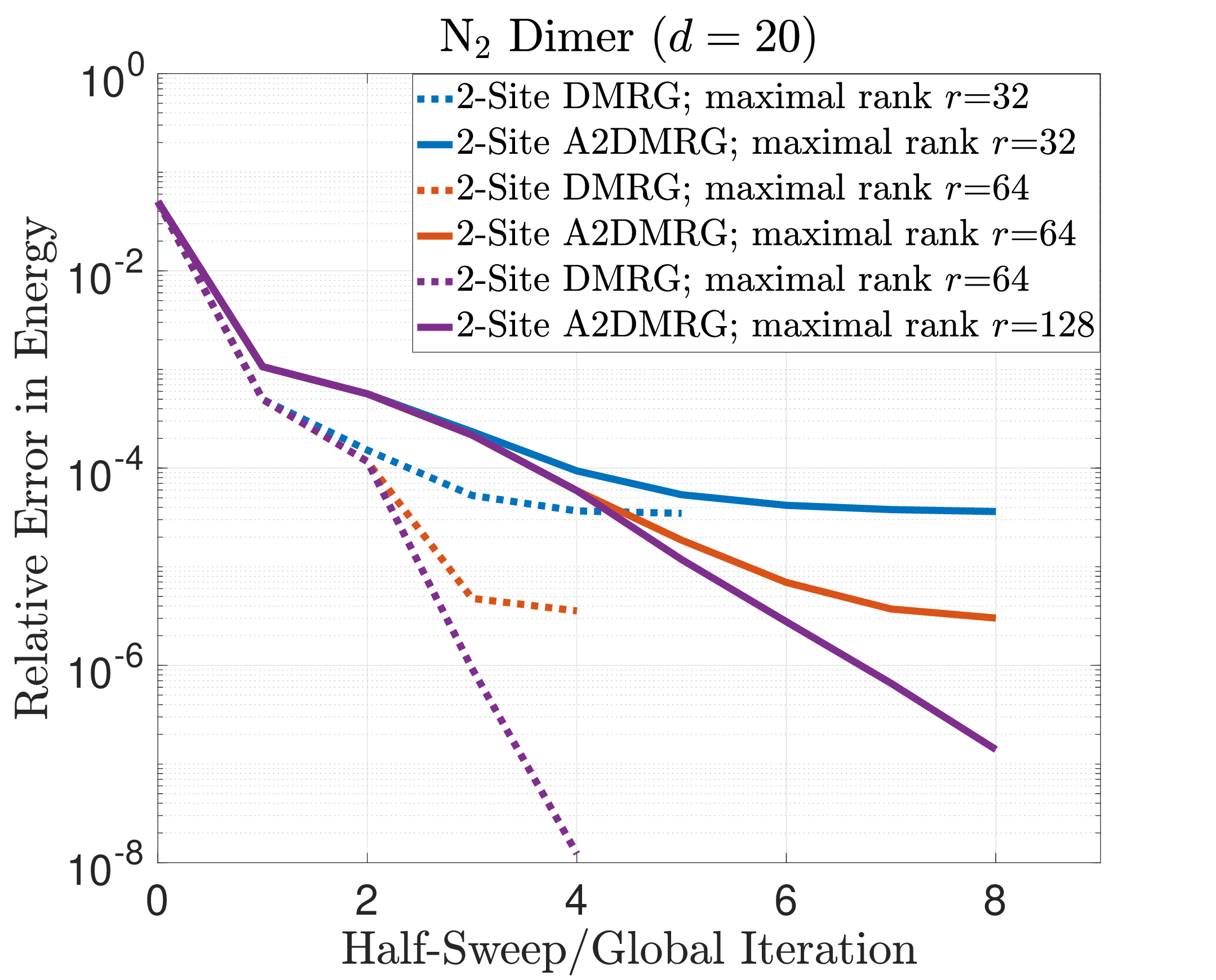} 
	\end{subfigure}\hfill
	\begin{subfigure}[t]{0.49\textwidth}
		\centering
		\includegraphics[width=\textwidth]{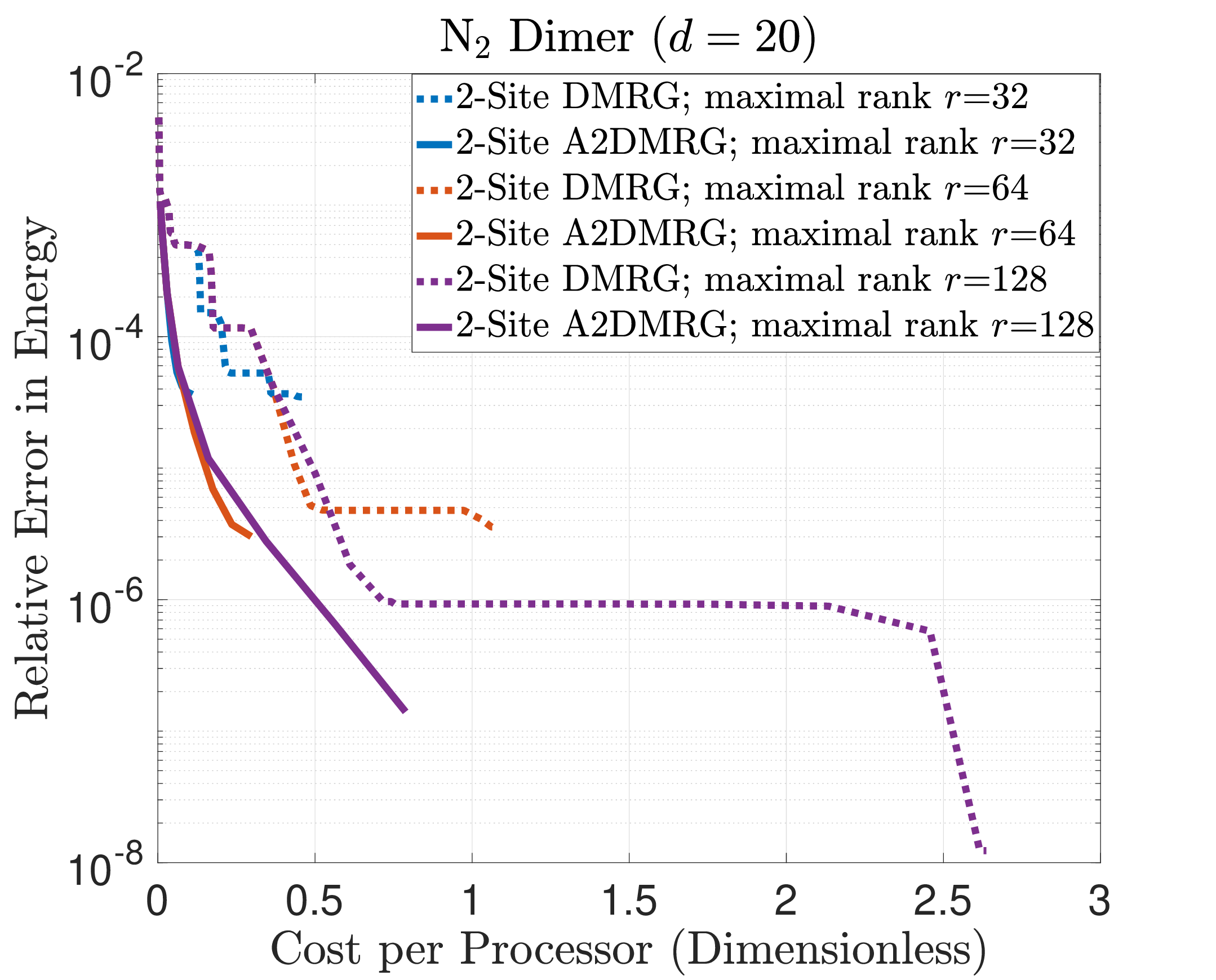} 
	\end{subfigure}
	\caption{Convergence plot of the classical DMRG and A2DMRG algorithms for the C$_2$ and N$_2$ dimers. The tolerances for the eigenvalue solvers and the truncated SVD were set to $10^{-6}$ and the algorithms were run until the relative energy difference between successive half-sweeps/global iterations was smaller than $10^{-6}$.}
	\label{fig:3}
\end{figure}

\begin{figure}[H]
	\centering
	\begin{subfigure}[t]{0.49\textwidth}
		\centering
		\includegraphics[width=\textwidth]{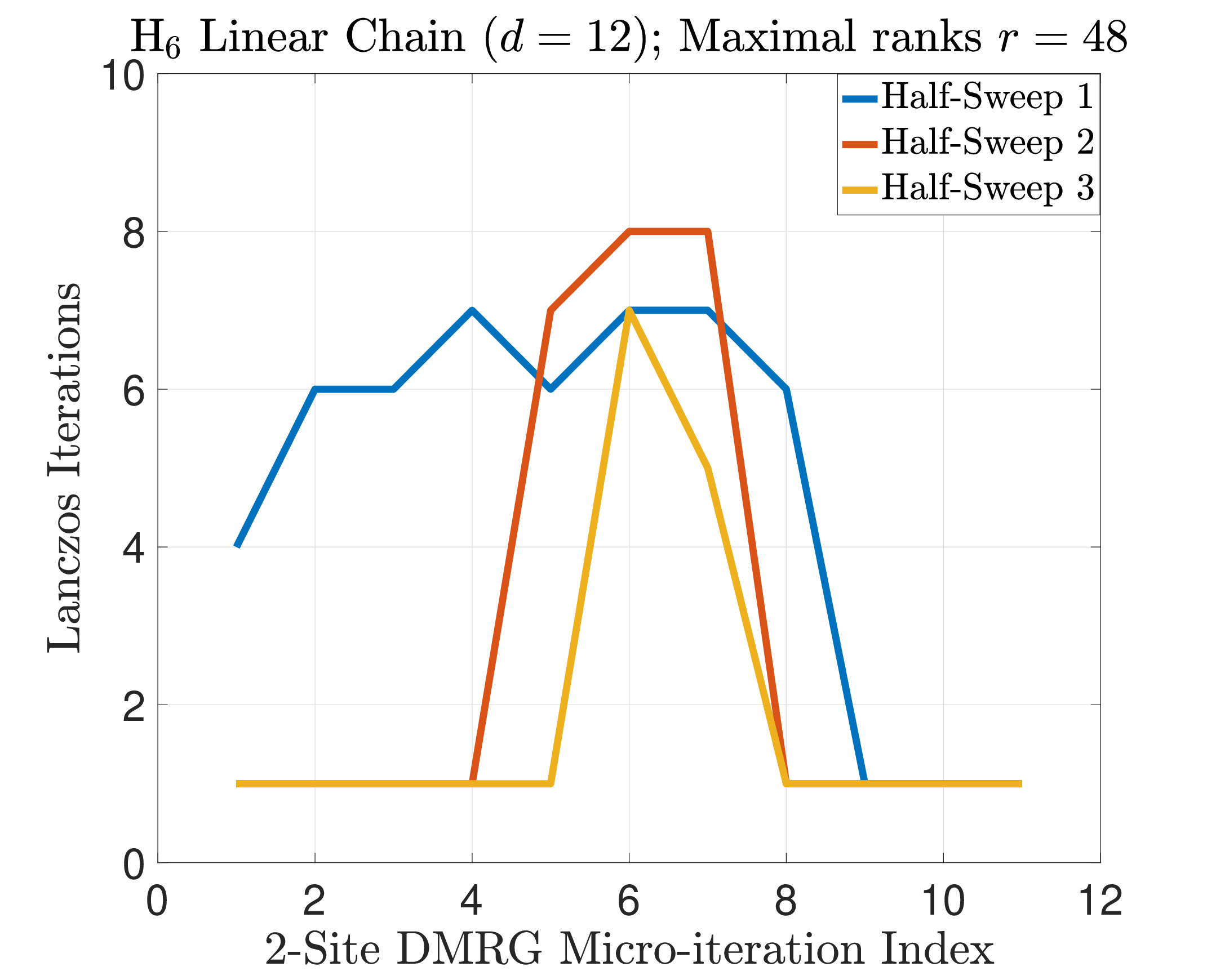} 
		\caption{The classical two-site DMRG algorithm.}
	\end{subfigure}\hfill
	\begin{subfigure}[t]{0.49\textwidth}
		\centering
		\includegraphics[width=\textwidth]{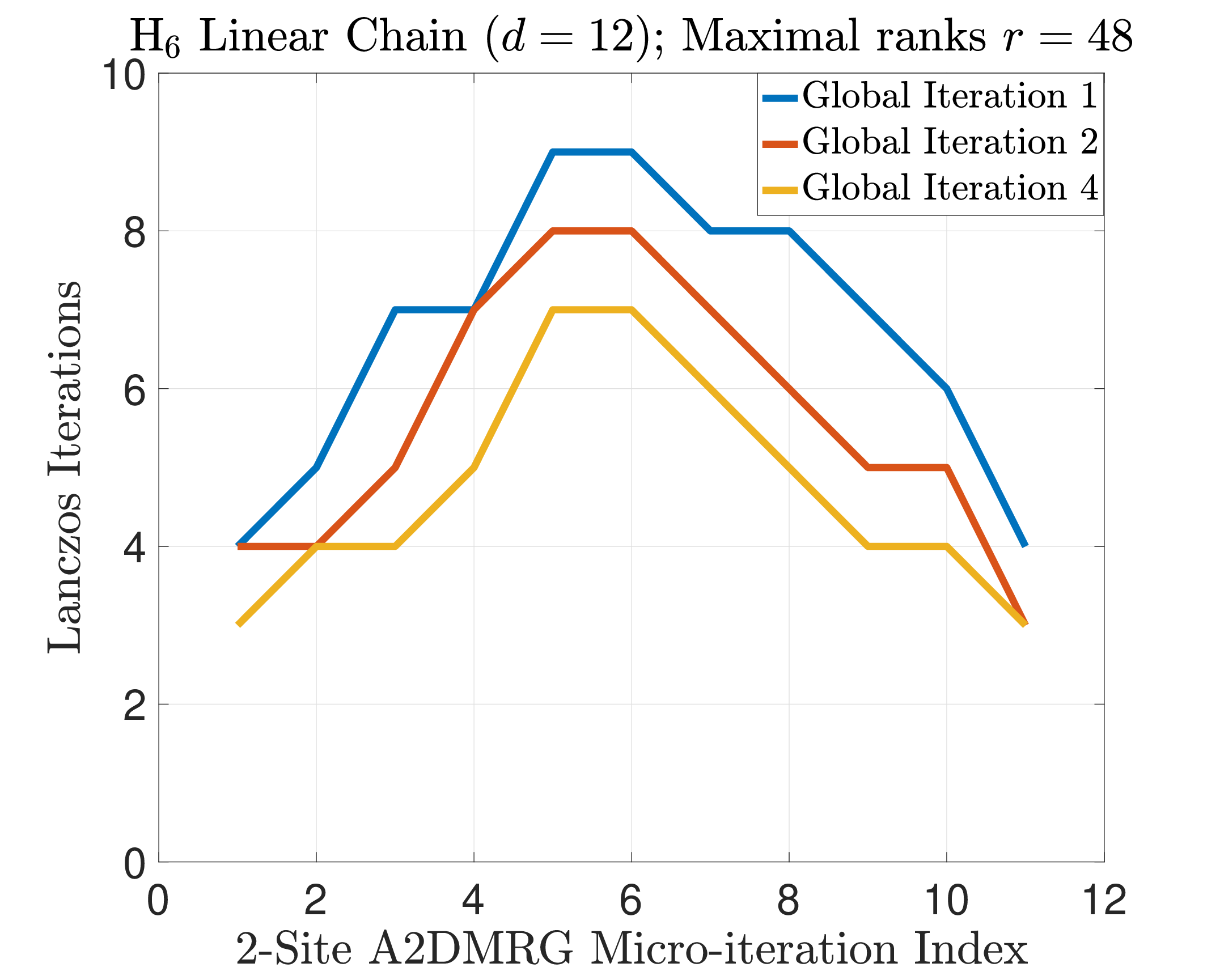} 
		\caption{The two-site A2DMRG algorithm.}
	\end{subfigure}\hfill
	\centering
	\begin{subfigure}[t]{0.49\textwidth}
		\centering
		\includegraphics[width=\textwidth]{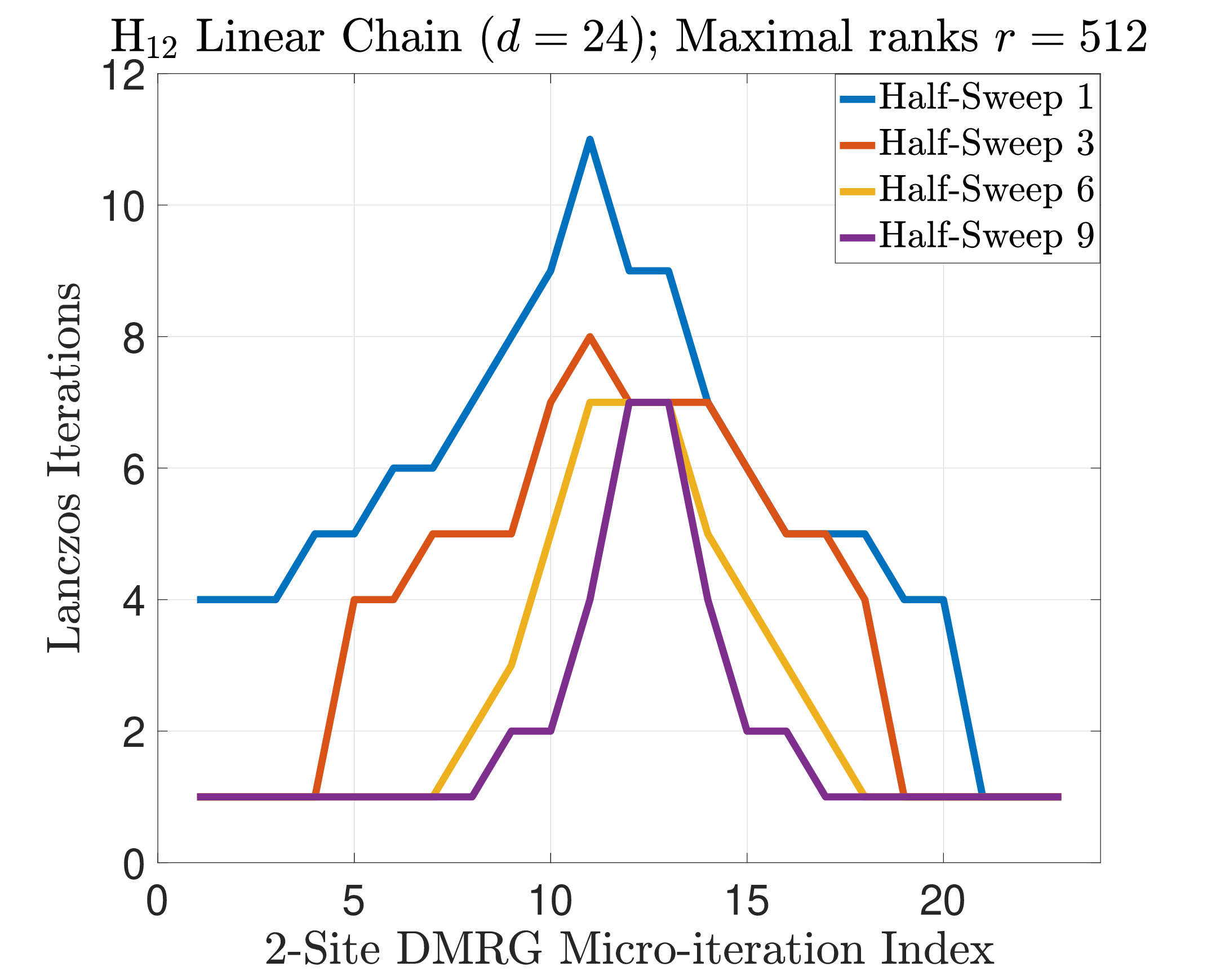} 
		\caption{The classical two-site DMRG algorithm.}
	\end{subfigure}\hfill
	\begin{subfigure}[t]{0.49\textwidth}
		\centering
		\includegraphics[width=\textwidth]{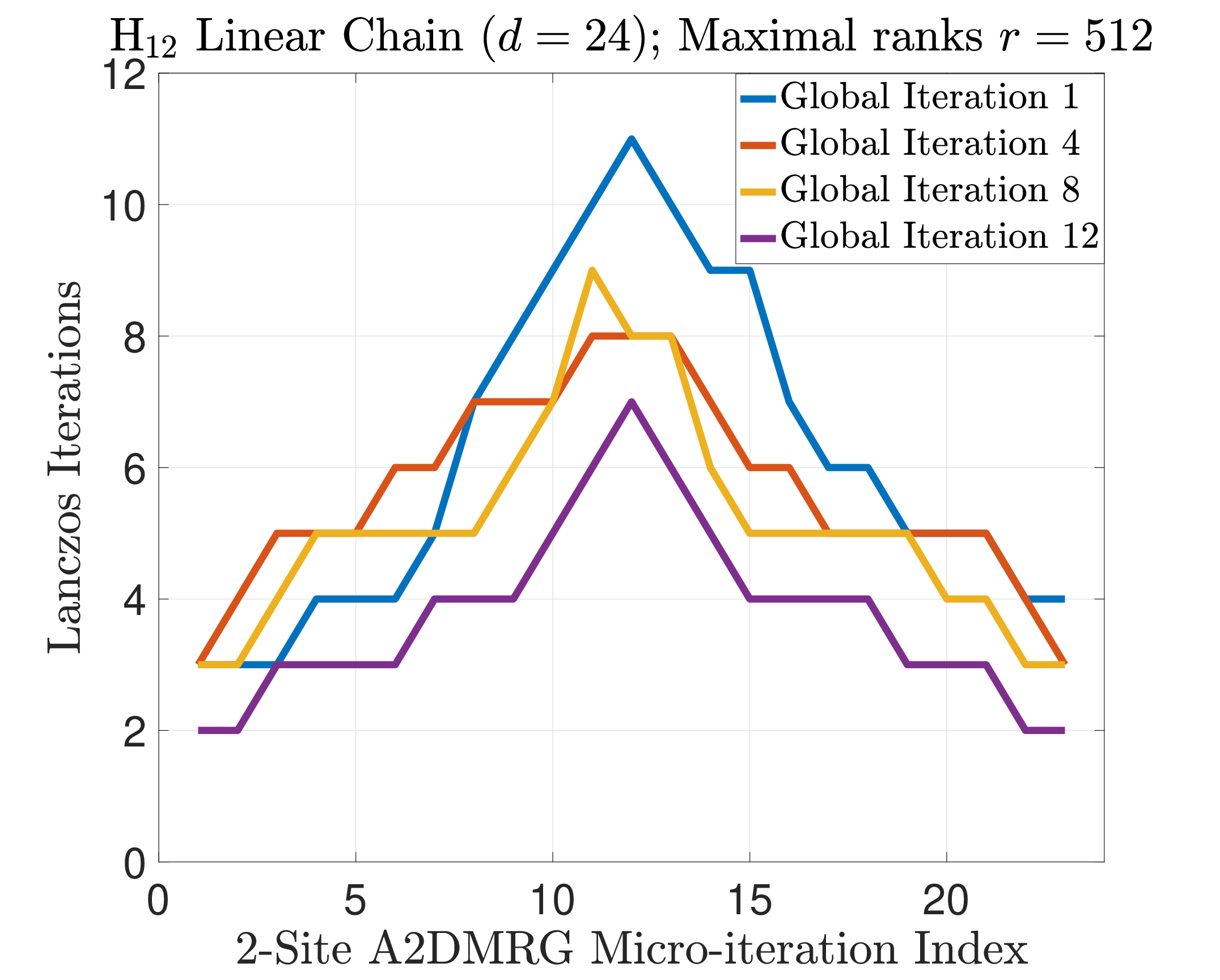} 
		\caption{The two-site A2DMRG algorithm.}
	\end{subfigure}
	\centering
	\begin{subfigure}[t]{0.49\textwidth}
		\centering
		\includegraphics[width=\textwidth, height=4.8cm]{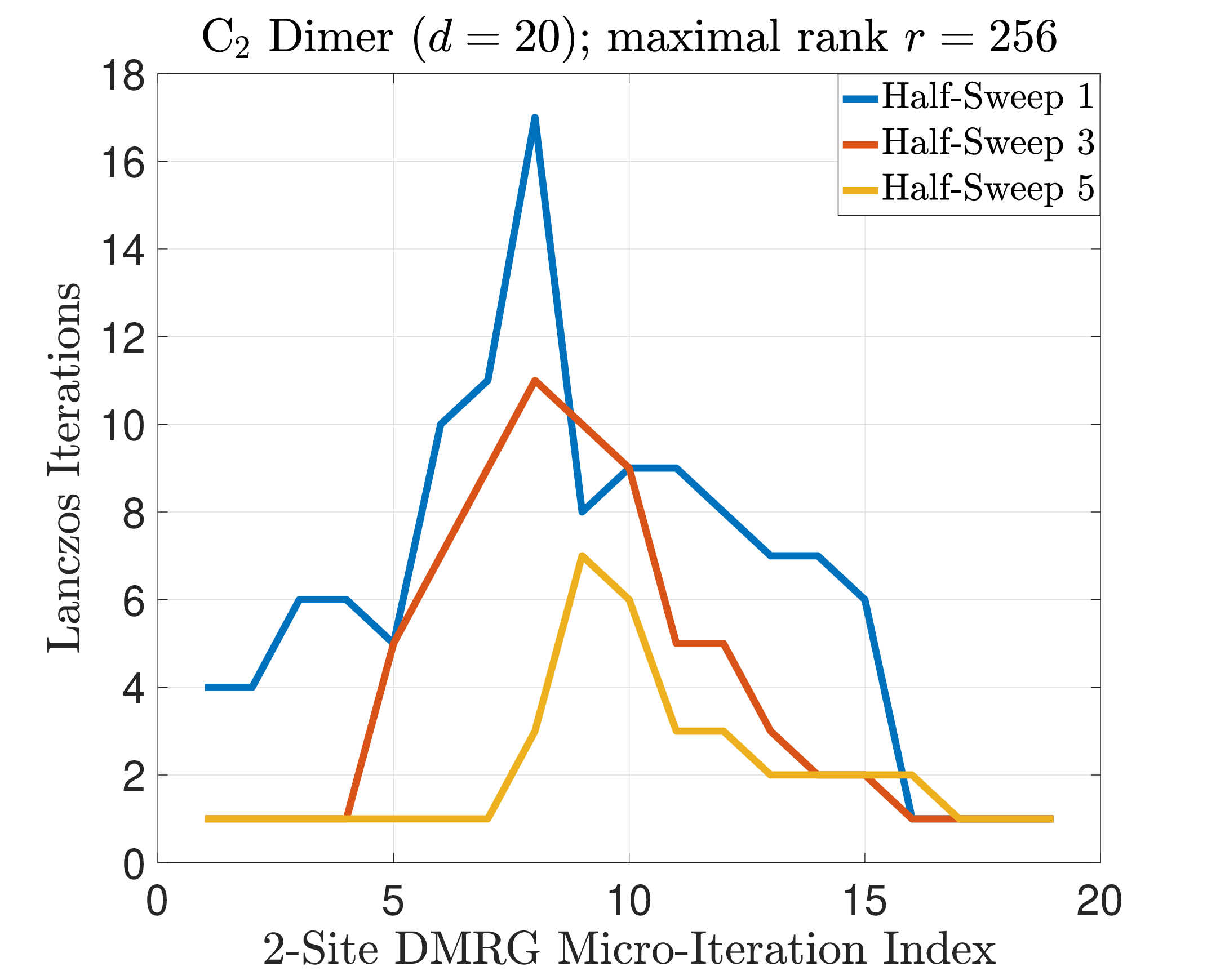} 
		\caption{The classical two-site DMRG algorithm.}
	\end{subfigure}\hfill
	\begin{subfigure}[t]{0.49\textwidth}
		\centering
		\includegraphics[width=\textwidth, height=4.8cm]{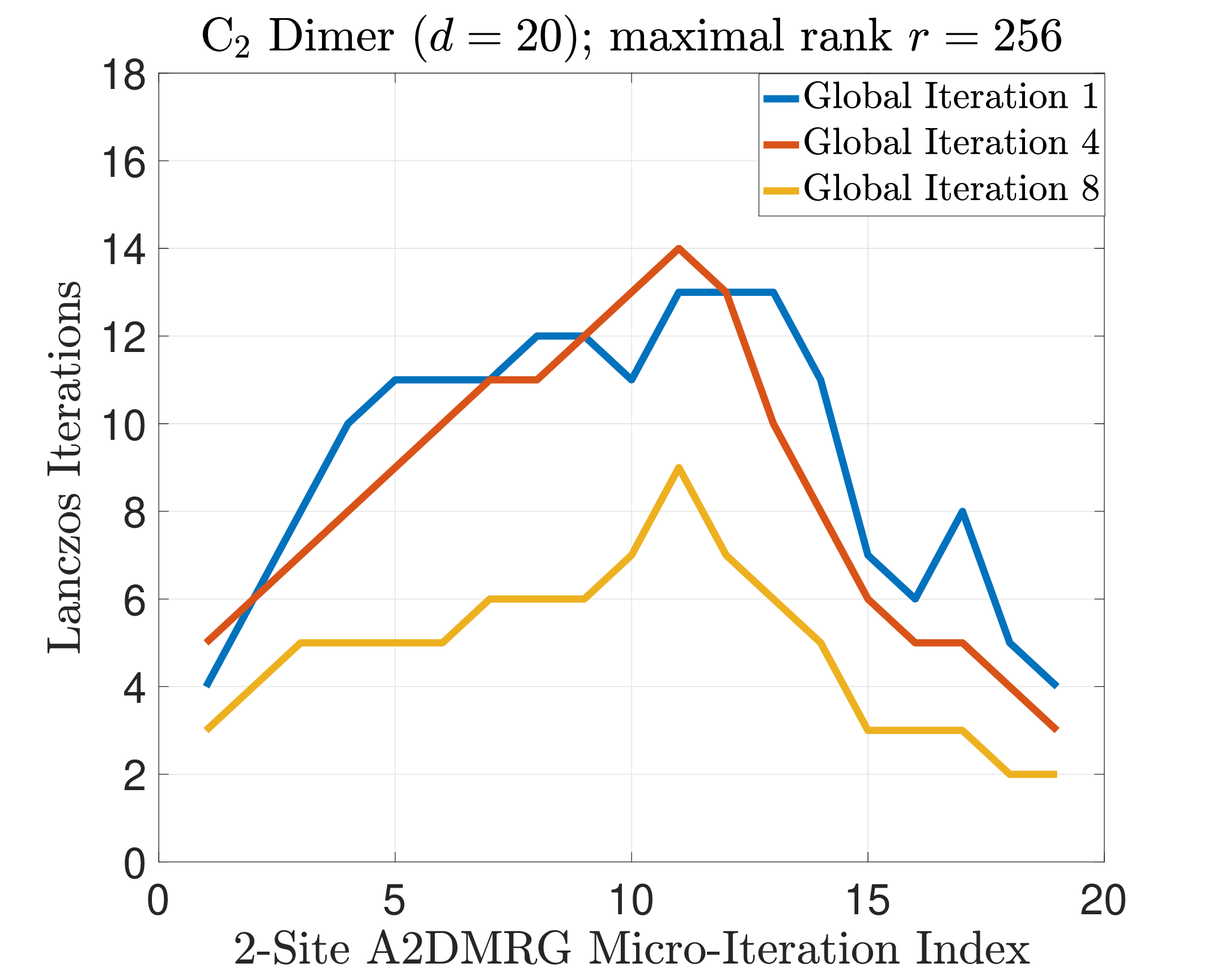} 
		\caption{The two-site A2DMRG algorithm.}
	\end{subfigure}
	\caption{The number of Lanczos iterations required for the micro-iterations in the classical DMRG and A2DMRG algorithms. The tolerances for the eigenvalue solvers and the truncated SVD were set to $10^{-6}$.}
	\label{fig:4}
\end{figure}

A number of observations can now be made. First we note that the A2DMRG \Cref{alg:PALS} demonstrates slower convergence than the classical DMRG algorithm~\ref{alg:ALS} as a function of the global iterations. Although, a rigorous quantification of the relative convergence rates of these two algorithms seems out of reach at the moment (indeed, even for the one-site DMRG algorithm, the local convergence analysis in~\cite{rohwedder2013local} does {not} yield convergence rates), we feel that there is an intuitive explanation for the observed behavior: First, due to the sequential nature of the classical DMRG micro-iterations, the TT ranks of the ansatz grow very quickly and reach the maximal rank parameter $\bold{r}$ within two half-weeps. On the other hand, the A2DMRG algorithm requires several global iterations before a maximal rank parameter is achieved thereby slowing the convergence rate. Second, it is well-known in the domain decomposition literature that multiplicative Schwarz methods (which can be seen as generalizations of the Gauss-Seidel method) converge, in general, faster than additive Schwarz methods (which can be seen as generalizations of the Jacobi method). Intuitively, therefore, the classical DMRG algorithm which can be viewed as a non-linear, overlapping multiplicative Schwarz method should also converge faster than the A2DMRG \Cref{alg:PALS} which can be interpreted as a non-linear, overlapping \emph{additive} Schwarz method. Interestingly, such a domain-decomposition-based interpretation of DMRG has-- to our knowledge-- not been studied in the literature.

Second, in spite of the previous observation on its relative convergence rate, we nevertheless observe that the A2DMRG \Cref{alg:PALS} exhibits a significant parallel speedup in terms of the cost per processor. The exact speed-up varies depending on the choice of ranks and molecule (the latter determines the order $d$ of the tensor), but ranges from a factor two to a factor six. Part of the reason for this speed-up (which is lower than the theoretical optimal speed-up of factor $d/4$) is that the classical and additive DMRG algorithms seem to require a different number of Lanczos iterations for the solution of the micro-iterations. To explore this, we plot in Figure \ref{fig:4}, the number of Lanczos iterations required for the DMRG micro-iterations over the course of the DMRG half-sweeps/global iterations for several of these molecular systems. These results indicate that, in general, only a minority of local problems in the classical DMRG algorithm require more than a single Lanczos iteration. This is not the case for the A2DMRG algorithm, which typically requires several Lanczos iterations for each micro-iteration, thereby producing a non-optimal speed-up. However, we emphasize that the number of Lanczos iterations for the A2DMRG algorithm displays minimal growth as a function of the maximal TT ranks $\bold{r}$ and the order $d$ of the tensor. The corresponding results for H$_8$, H$_{10}$ and the N$_{2}$ dimer are similar and are therefore not displayed for brevity.

\begin{figure}[H]
	\centering
	\begin{subfigure}[t]{0.49\textwidth}
		\centering
		\includegraphics[width=\textwidth, height=4.8cm]{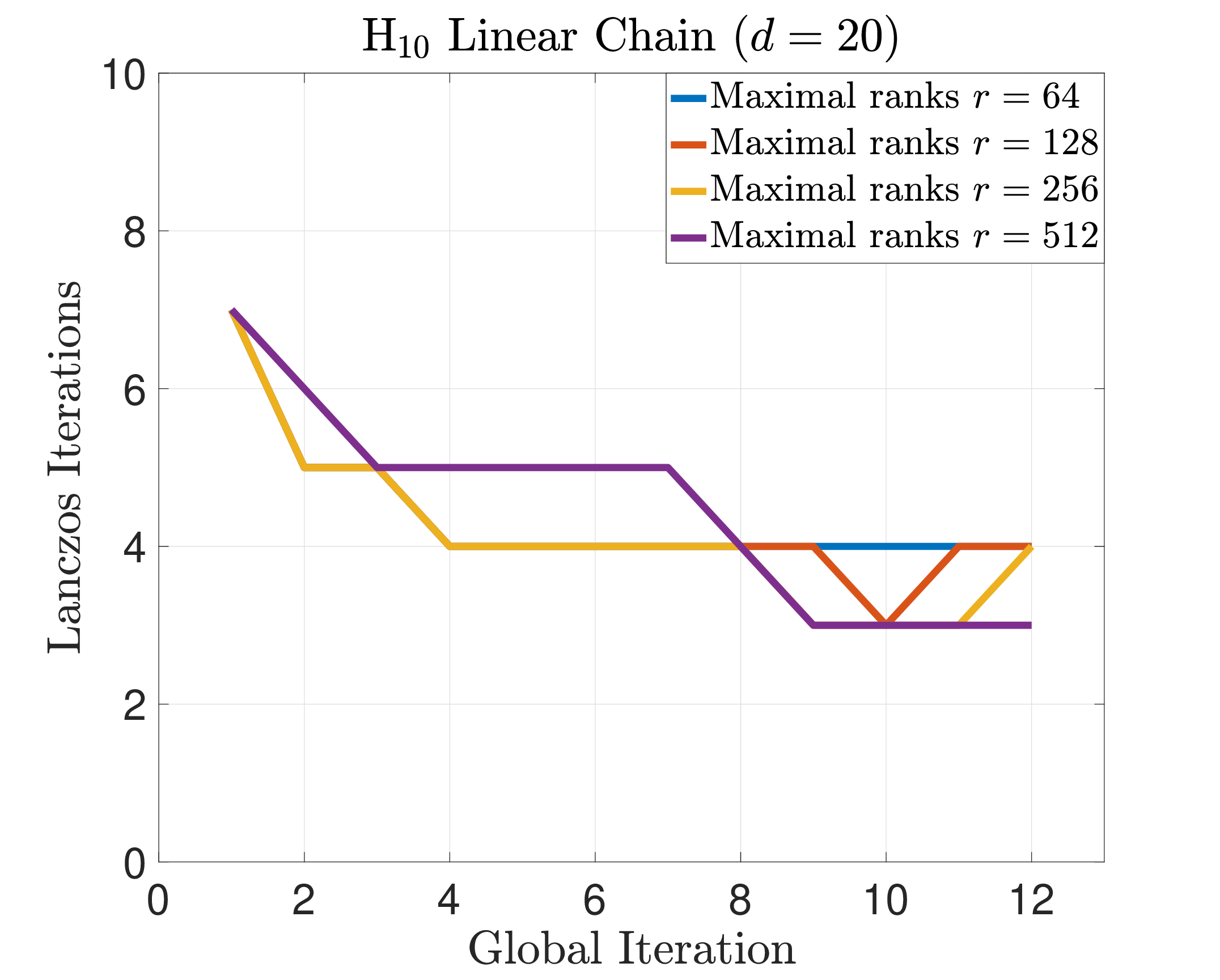} 
	\end{subfigure}\hfill
	\begin{subfigure}[t]{0.49\textwidth}
		\centering
		\includegraphics[width=\textwidth, height=4.8cm]{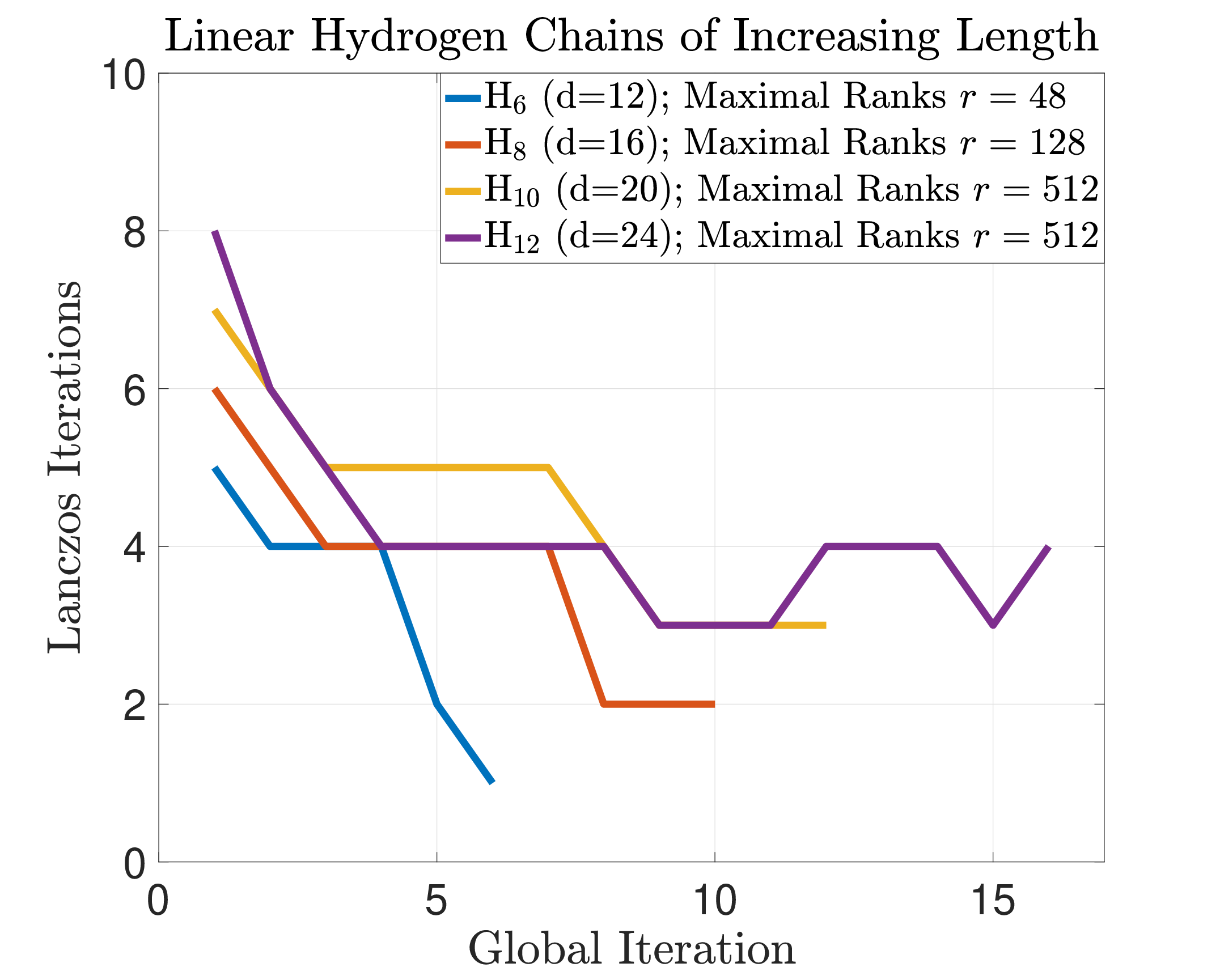} 
	\end{subfigure}\hfill
	\caption{The number of Lanczos iterations required for the second-level minimization problem in the A2DMRG algorithms. All tolerances were set to $10^{-6}$.}
	\label{fig:new_1}
\end{figure}

For completeness, we also plot in Figure \ref{fig:new_1}, the number of Lanczos iterations required to solve the second-level minimization problem in several of these test cases. As before, we observe there that the number of Lanczos iterations is stable as a function of both the maximal rank parameters and the order $d$ of the tensor.

\begin{figure}[h!]
	\centering
	\begin{subfigure}[t]{0.49\textwidth}
		\centering
		\includegraphics[width=\textwidth]{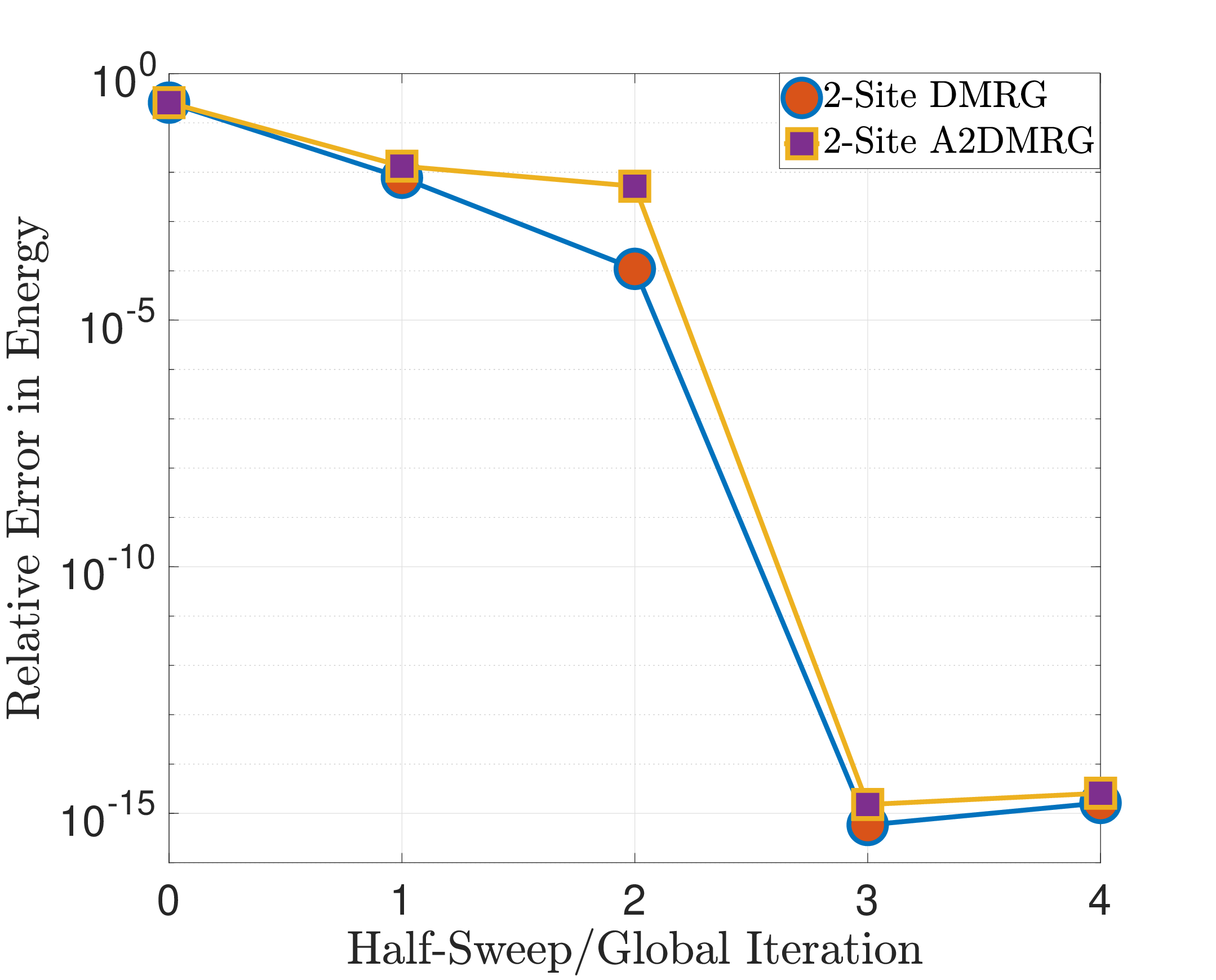} 
	\end{subfigure}\hfill
	\begin{subfigure}[t]{0.49\textwidth}
		\centering
		\includegraphics[width=\textwidth]{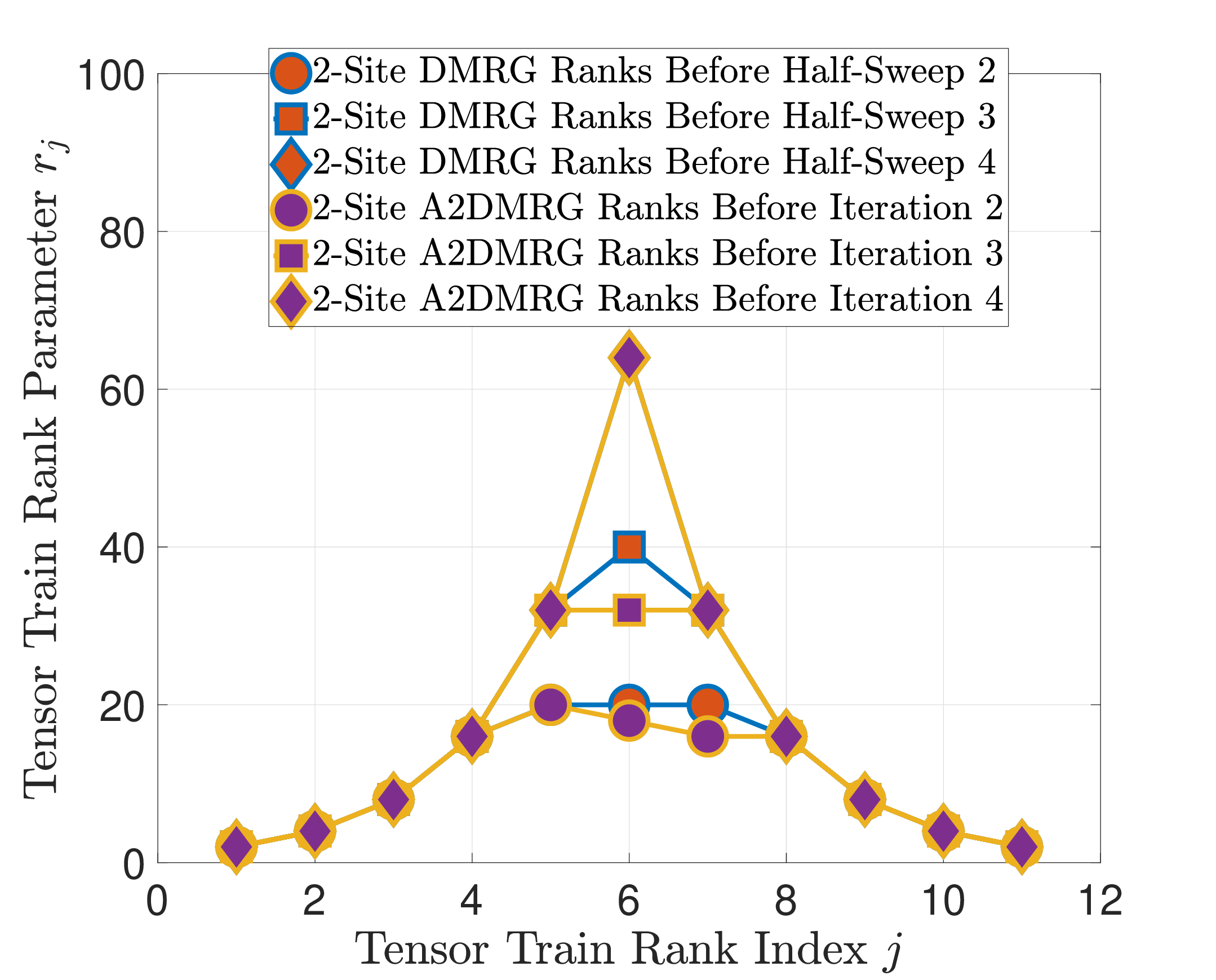} 
	\end{subfigure}\hfill
	\centering
	\begin{subfigure}[t]{0.49\textwidth}
		\centering
		\includegraphics[width=\textwidth]{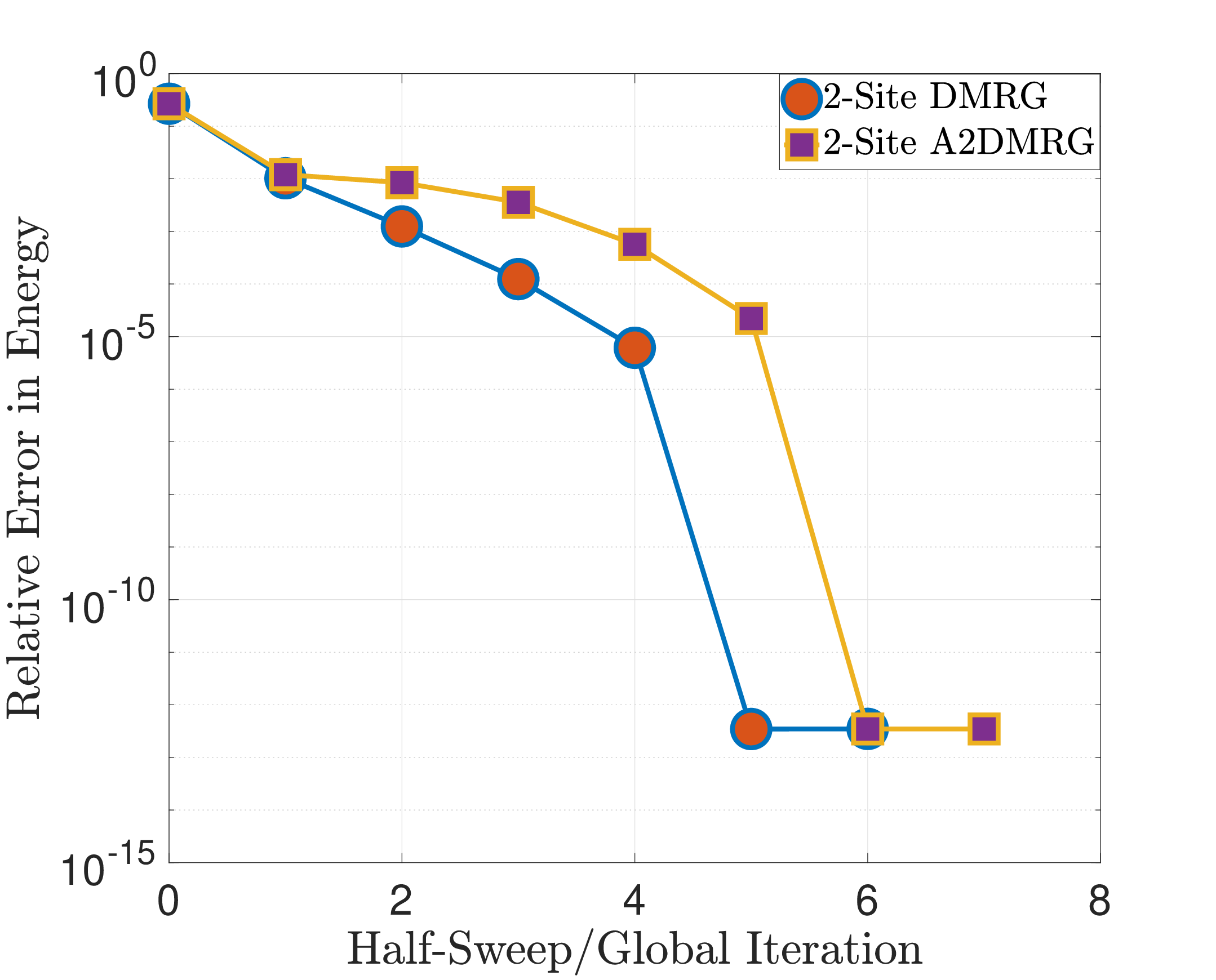} 
	\end{subfigure}\hfill
	\begin{subfigure}[t]{0.49\textwidth}
		\centering
		\includegraphics[width=\textwidth]{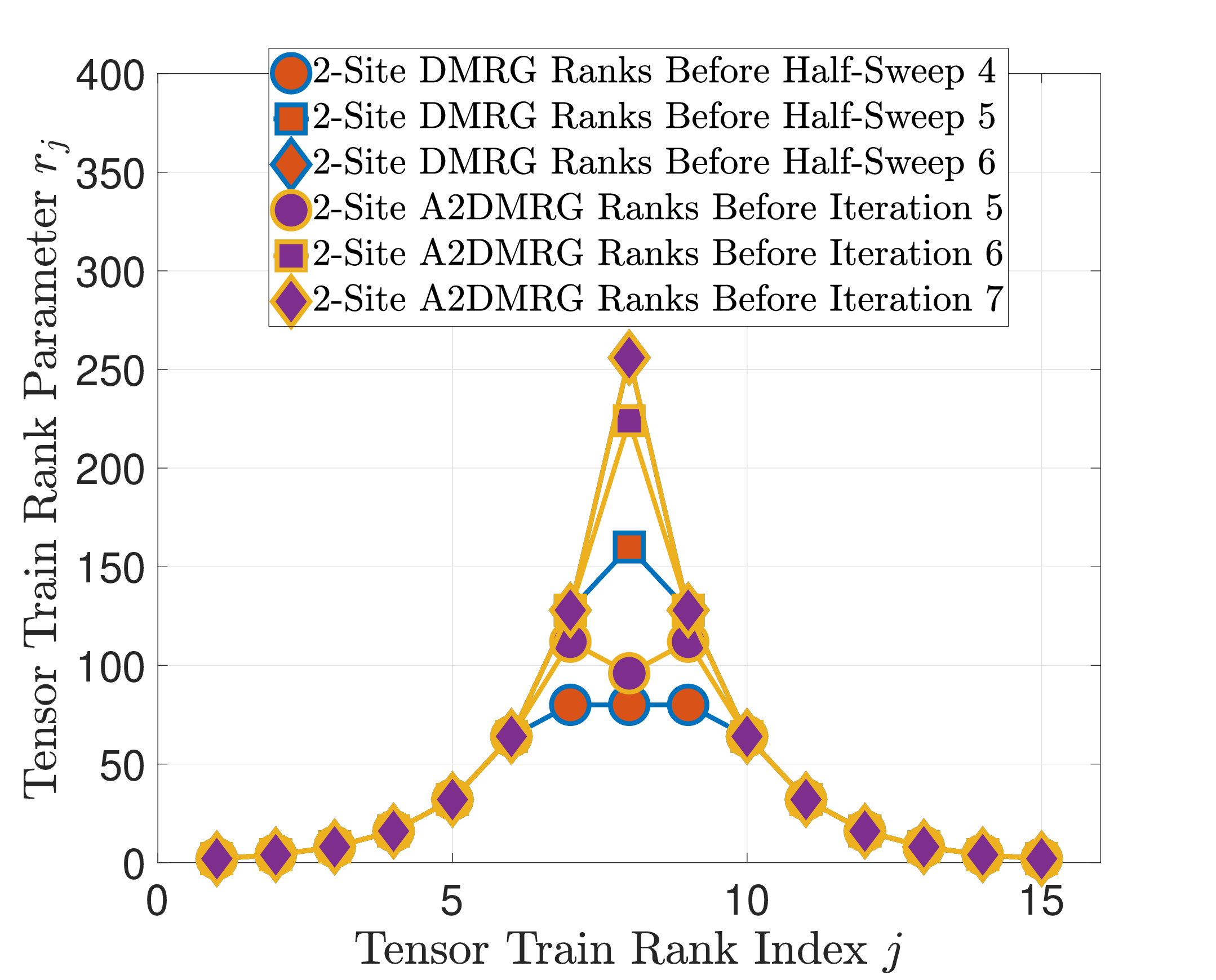} 			
	\end{subfigure}
	\caption{Numerical evidence for the convergence of the classical DMRG and A2DMRG algorithms in a single half-sweep when the tensor train ranks of the ansatz equal the separation ranks of the sought minimizer. The tolerances for the eigenvalue solvers were set to $10^{-10}$, and the algorithms were run until the relative energy difference between successive half-sweeps/global iterations was smaller than $10^{-10}$. The molecular systems here are H$_6$ (above) for which $d=12$ and H$_8$ (below) for which $d=16$.}
	\label{fig:7}
\end{figure}

A well-known property of the classical one-site DMRG algorithm is that if the ranks of the ansatz tensor-train decomposition are exactly the separation ranks of the minimizing tensor, then the DMRG algorithm converges in a single half-sweep \cite{rohwedder2013local}. A similar property also seems to holds for the two-site DMRG algorithm. An interesting question to ask is if the A2DMRG \Cref{alg:PALS} satisfies a similar property. In order to answer this question, we show in Figure \ref{fig:7} the error plots for the classical DMRG and A2DMRG algorithms applied to the H$_6$ and H$_8$ molecule with \emph{no rank truncation}. These plots indicate that as soon as the TT ranks of the ansatz are sufficiently close to the maximal rank $\bold{r}=64$ for H$_6$ and $\bold{r}=256$ for H$_8$, the A2DMRG algorithm converges to within machine precision in a single global iteration. We remark that this property does not seem to be satisfied by tensor minimization algorithms based on Riemannian optimization.




\section{Discussion and Perspective} 

The additive two-level DMRG algorithm raises interesting questions from the perspective of both analysis and applications. On the one hand, a rigorous convergence analysis of the algorithm, including a study of how the choice of coarse space impacts convergence rates, remains to be undertaken. Moreover, the parallelism of local solves suggests that adaptive Schwarz strategies wherein the coarse space is dynamically enriched (possibly based on error indicators) may yield further speedups and accuracy. Addressing these points can improve our understanding of Schwarz methods in the context of tensor networks and will be the subject of future work. 

\section*{Acknowledgments}
The first author (L. G.) acknowledges support from the European Research Council (ERC) under the European Union’s Horizon 2020 research and innovation program (Grant No. 810367), project EMC2. The second author (M. H.) warmly thanks Mi-Song Dupuy for discussions that initiated this project, Siwar Badreddine for explaining in painstaking detail the implementation of the DMRG algorithm, and Fabrice Serret for very helpful comments on the manuscript.

\bibliographystyle{plain}

\bibliography{refs.bib}

\appendix 

\section{Computational Scaling of the A2DMRG Algorithm}\label{sec:3.3}~

The aim of this section is to discuss in more detail the implementation and computational scaling of various steps in the two-level A2DMRG \Cref{alg:PALS} when the underlying energy functional $\mathcal{J}\colon U\subset \mathbb{R}^{n_1\times \ldots \times n_d} \rightarrow \mathbb{R}$ takes the form of a Rayleigh quotient as given by Equation \eqref{eq:energy_eig}, i.e.,
\begin{align*}
	\forall \bold{X}\in  U=\mathbb{R}^{n_1 \times n_2 \times \ldots n_d}\setminus \{0\}\colon \qquad        \mathcal{J}(\bold{X})= \frac{\langle \bold{X}, \bold{A}\bold{X}\rangle}{\Vert \bold{X}\Vert^2},
\end{align*}
where $\bold{A}\colon \mathbb{R}^{n_1 \times n_2 \times \ldots n_d} \rightarrow \mathbb{R}^{n_1 \times n_2 \times \ldots n_d}$ is a symmetric operator with a simple lowest eigenvalue. The case when $\mathcal{J}$ is a quadratic form as given by Equation \eqref{eq:energy_lin}, generated by a symmetric, positive definite tensor operator  $\bold{A}\colon \mathbb{R}^{n_1 \times n_2 \times \ldots n_d} \rightarrow \mathbb{R}^{n_1 \times n_2 \times \ldots n_d}$ is very similar and therefore omitted for brevity.

As is standard in the literature, we assume that the tensor operator $\bold{A}\colon \mathbb{R}^{n_1 \times \ldots \times n_d}$ $ \rightarrow \mathbb{R}^{n_1 \times \ldots \times n_d}$ possesses a tensor train decomposition. In other words, viewing $\bold{A}$ as an order $d$-tensor in the space $\mathbb{R}^{(n_1 \times n_1) \times \ldots \times (n_d \times n_d)}$, we assume that there exists some vector of ranks $\bold{R}= (R_0, \ldots R_{d})\in \mathbb{N}^{d+1}$ with ${R}_0={R}_d=1$ and $\{\bold{A}_j\}_{j=1}^d, ~ \bold{A}_j \in \mathbb{R}^{R_{j-1}\times n_j \times n_j \times R_j}$ such that
\begin{align}\label{eq:MPO}
	&\bold{A}\big((x_1,x_1'), \ldots, (x_d, x_d')\big)\\[0.25em] \nonumber
	:= &\tau(\bold{A}_1, \ldots, \bold{A}_d)\big((x_1,x_1'), \ldots, (x_d, x_d')\big)\\[0.25em] \nonumber
	:=&  \bold{A}_1(:, x_1,x_1', :)\bold{A}_2(:, x_2,x_2', :)\ldots \bold{A}_d(:, x_d,x_d', :)\\[0.25em]
	=&\sum_{k_1=1}^{R_1} \sum_{k_2=1}^{R_2} \ldots \sum_{k_{d-1}=1}^{R_{d-1}} \bold{A}_1(1, x_1, x_1', k_1)\bold{A}_2(k_1, x_2, x_2', k_2)\ldots \bold{A}_d(k_{d-1}, x_d, x_d', 1). \nonumber
\end{align}
Such a decomposition can, for instance, be obtained by reordering the indices in the matrix representation of $\bold{A}$ and applying the TT-SVD algorithm \cite{oseledets2011tensor}. We refer to \cite[Appendix]{badreddine2024leveraging} for such an approach tailored to the electronic Hamiltonian operator arising from quantum chemistry applications as described in Section \ref{sec:3.2}.

\subsection{Computational Scaling of Classical DMRG Algorithms}\label{sec:3.3.1}~

Let us first recall the implementation of the classical one-site DMRG ~\Cref{alg:ALS}. We assume that we have at hand a $j$-orthogonal tensor train decomposition $\bold{U}= (\bold{U}_1, \ldots, \bold{U}_d)$ for $j \in \{2, \ldots, d-1\}$ (the cases $j=1, d$ being simpler), and we are interested in performing the $j^{\rm th}$ one-site micro-iteration in the course of a left-to-right half-sweep. To do so we are required to perform matrix-vector products of the form
\begin{align}\label{eq:mat_vec}
	\bold{A}_j^{\rm 1site, mat} \bold{v}_j, \qquad j \in  \{1, \ldots, d\}, ~\bold{v}_j \in \mathbb{R}^{r_{j-1}n_jr_j}.
\end{align}
Here, $\bold{A}_j^{\rm 1site, mat} \in \mathbb{R}^{r_{j-1}n_jr_j \times r_{j-1}n_jr_j}$ is the matricization of the operator $\mathbb{P}_{\bold{U}, j,1}^* \bold{A}\mathbb{P}_{\bold{U}, j, 1}$ introduced in Lemma \ref{lem:ALS_micro}, and we remind the reader that $\mathbb{P}_{\bold{U}, j,1}, ~\bold{U}\in \mathcal{U}_{\bold{r}}$ is the retraction operator defined through Definition \ref{def:retraction}.

An efficient method to compute the matrix-vector product \eqref{eq:mat_vec} is to split the construction of the matrix~$\bold{A}_j^{\rm 1site}$ into a left and a right component, frequently referred to as the so-called left and right environments. Indeed, let us first introduce for each $j \in \{1, \ldots, d\}$, the tensor $\bold{G}_{j} \in \mathbb{R}^{r_{j-1} \times r_j \times R_{j-1} \times R_{j} \times r_{j-1}\times r_{j}}$ given by
\begin{align*}
	\bold{G}_j(k_{j-1}, k_{j}, K_{j-1}, K_{j}, &k_{j-1}', k_{j}')=\\
	&\sum_{\ell_j=1}^{n_j} \sum_{\ell_j'=1}^{n_j} \bold{U}_j(k_{j-1}, \ell_j, k_{j})\bold{A}_j(K_{j-1}, \ell_j, \ell_j', K_{j}) \bold{U}_j(k_{j-1}', \ell_j', k_{j}'),
\end{align*}
where the indices $k_{j-1}, k_{j-1}' \in \{1, \ldots, r_{j_1}\}$, $k_{j}, k_{k}' \in \{1, \ldots, r_{j}\}$, $K_{j-1} \in \{1, \ldots, R_{j_1}\}$, and $K_{j} \in \{1, \ldots, R_{j}\}$.

Since $r_0=R_0=r_d=R_d=1$, we can identify $\bold{G}_1$ and $\bold{G}_d$ as elements of $\mathbb{R}^{r_1 \times R_{1}\times r_{1}}$ and $\mathbb{R}^{r_{d-1} \times R_{d-1}\times r_{d-1}}$ respectively. With this construction in mind, we now define for each $j \in \{1, \ldots, d-1\}$, the $j^{\rm th}$ left environment $\bold{H}^{\rm L}_j \in \mathbb{R}^{r_j \times R_j \times r_j}$ as
\begin{align}\label{eq:env_L} 
	\bold{H}^{\rm L}_j(k_j, K_j, k_j')=\sum_{k_1=1}^{r_1}\ldots \sum_{k_{j-1}=1}^{r_{j-1}} \sum_{K_1=1}^{R_1}\ldots \sum_{K_{j-1}=1}^{R_{j-1}} & \sum_{k_1'=1}^{r_1}\ldots \sum_{k_{j-1}'=1}^{r_{j-1}}\bold{G}_1(k_1, K_1, k_1') \ldots \\ \nonumber
	& \ldots \bold{G}_{j}(k_{j-1},k_j, K_{j-1},K_j, k_{j-1}', k_j').
\end{align}
Here, the indices $k_{j}, k_{j}' \in \{1, \ldots, r_{j}\}$ and $K_{j} \in \{1, \ldots, R_{j}\}$. Similarly, we define for each $j \in \{2, \ldots, d\}$, the $j^{\rm th}$ right environment $\bold{H}^{\rm R}_j \in \mathbb{R}^{r_{j-1} \times R_{j-1} \times r_{j-1}}$ as
\begin{align}\label{eq:env_R}
	\bold{H}^{\rm R}_j(k_{j-1}, K_{j-1}, &k_{j-1}')= \sum_{k_{d-1}=1}^{r_{d-1}}\ldots \sum_{k_{j}=1}^{r_{j}} \sum_{K_{d-1}=1}^{R_{d-1}}\ldots \sum_{K_j=1}^{R_{j}}  \sum_{k_{d-1}'=1}^{r_{d-1}}\ldots \\ \noindent
	&\ldots \sum_{k_{j}'=1}^{r_{j}}\bold{G}_j(k_{j-1}, k_j, K_{j-1}, K_j, k_{j-1}', k_j') \ldots \bold{G}_{d}(k_{d-1}, K_{d-1}, k_{d-1}'),
\end{align}
where the indices $k_{j-1}, k_{j-1}' \in \{1, \ldots, r_{j-1}\}$ and $K_{j-1} \in \{1, \ldots, R_{j-1}\}$.

It can now readily be seen that for micro-iteration $j \in \{2, \ldots d-1\}$, the matrix-vector product \eqref{eq:mat_vec} can be computed by contracting the left environment $\bold{H}^{\rm L}_{j-1}$ with $\bold{A}_j, \bold{v}_j$ and the right environment $\bold{H}^{\rm R}_{j+1}$. Moreover, for micro-iteration $j=1$ (respectively, $j=d$) we simply need to contract $\bold{A}_j$ and $\bold{v}_j$ with the right environment $\bold{H}^{\rm R}_{2}$ (respectively, the left environment $\bold{H}^{\rm L}_{d-1}$). The computational cost of this step, i.e., the cost of constructing the environments $\bold{H}_{j-1}^{\rm L}, \bold{H}_{j+1}^{\rm R}$ and performing the necessary tensor contractions scales as
\begin{align}\label{eq:scaling_ALS_1}
	{\rm Cost}_{j, \rm Mat-vec }&= {\rm Cost}_{j, \rm Left\; Env }+{\rm Cost}_{j, \rm Right\; Env }+ {\rm Cost}_{j, \rm Contract }\\ \nonumber
	&= \mathcal{O}(j n r^3 R) +\mathcal{O}((d-j) n r^3 R) + \mathcal{O}(n r^3 R).
\end{align}
Here, we denote $n=\max \{n_j\}_{j=1}^d $, $r= \max\bold{r}$ and $R=\max\bold{R}$, and we have assumed (as is the case for state-of-the-art quantum chemistry simulations) that $r > nR$. We also emphasize that these costs are calculated under the assumption that the left environment $\bold{H}_{j-1}^{\rm L}$ is constructed by successively contracting the tensors $\{\bold{G}_\ell\}_{\ell=1}^{j-1}$ from left-to-right while the right environment $\bold{H}_{j+1}^{\rm R}$ is constructed by successively contracting the tensors $\{\bold{G}_\ell\}_{\ell=d}^{j+1}$ from right-to-left.


It follows that the dominant cost in computing the matrix vector product \eqref{eq:mat_vec} is the construction of the left and right environments $\bold{H}^{\rm L}_{j-1}$ and $\bold{H}^{\rm R}_{j+1}$ respectively, which raises the question of whether these objects can be precomputed at the beginning of the half-sweep and updated on the fly after each micro-iteration. Recalling the one-site DMRG Algorithm \ref{alg:ALS} and considering Equations \eqref{eq:env_L}-\eqref{eq:env_R} that define these environments, we see that the answer is yes in the following sense: In the course of a left-to-right half-sweep, as we transition from micro-iteration $j$ to $j+1$, the left environment $H_{j}^{\rm L}$ can be constructed from $\bold{H}_{j-1}^{\rm L}$ simply by contracting with $\bold{A}_j$ and the newly obtained solution $\widetilde{\bold{U}}_{j}$ to the $j^{\rm th}$ one-site DMRG micro-iteration. Consequently, in the course of a left-to-right half-sweep, at any one-site DMRG micro-iteration $j \in \{1, \ldots, d\}$, the construction from scratch of the left-environment $\bold{H}^{\rm L}_{j-1}$ can be replaced with a simple update step of reduced cost. This leads to the following computational scaling of the matrix vector product \eqref{eq:mat_vec}:
\begin{align}\label{eq:scaling_ALS_2}
	{\rm Cost}^{\rm update}_{j, \rm Mat-vec}&= {\rm Cost}^{\rm update}_{j, \rm Left\; Env }+{\rm Cost}_{j, \rm Right\; Env }+ {\rm Cost}_{j, \rm Contract }\\ \nonumber
	&= \mathcal{O}(n r^3 R) +\mathcal{O}((d-j) n r^3 R) + \mathcal{O}(n r^3 R).
\end{align}

Remaining within the framework of a left-to-right half-sweep as before, we see that the situation with the right environment is different since transitioning from micro-iteration $j$ to $j+1$ requires us to compute $\bold{H}^{\rm R}_{j+2}$, which, in view of Equation \eqref{eq:env_R} cannot be obtained by updating $\bold{H}^{\rm R}_{j+1}$. On the other hand, it is also clear from the description of the DMRG algorithm \ref{alg:ALS} that the right environments $\{\bold{H}^{\rm R}_j\}_{j >1}$ do not depend on the solutions $\{\widetilde{\bold{U}}_\ell\}_{\ell < j}$ obtained from the one-site DMRG micro-iterations in the current half-sweep. With this observation, we see that we have essentially two possibilities:

\begin{itemize}
	\item If the maximal rank parameters $r, R$ and tensor order $d$ are sufficiently small and the storage capacity is sufficiently large, then we can precompute all $d$ right-environments $\{\bold{H}^{\rm R}_j\}_{j>1}$ at the beginning of the left-to-right half-sweep. Similar to the updates of the left environments discussed above, the right environment $\bold{H}_{j}^{\rm R}$ can be constructed by updating $\bold{H}_{j-1}^{\rm R}$ at a cost of $\mathcal{O}(n r^3 R)$. Consequently, the total computational cost of this pre-computation step is $\mathcal{O}(dn r^3 R)$. The matrix-vector product $\bold{A}_1^{\rm 1site, mat}\bold{v}_1$ in the first micro-iteration can now be performed simply by contracting $\bold{A}_1$ and $\bold{v}_1$ with the right environment $\bold{H}^{\rm R}_{2}$. Once the first micro-iteration is performed, we define the first left-environment $\bold{H}_1^{\rm L}$ and then adaptively update $\bold{H}_{j}^{\rm L}, ~ j>1$ in the course of the one-site DMRG half-sweep. Since, we can additionally store the left-environments $\{\bold{H}_{j}^{\rm L}\}_{j=1}^{d}$ computed in this manner in memory, and since the role of the left and right environments is inverted when we transition from a left-to-right to a right-to-left half-sweep (see Algorithm \ref{alg:ALS}), the pre-computation of the right environments $\{\bold{H}^{\rm R}_j\}_{j>1}$ has to be done only once at the start of the one-site DMRG algorithm. Consequently, we obtain the following improvement to the computational cost \eqref{eq:scaling_ALS_2} of a matrix-vector product in each one-site DMRG micro-iteration: 
	\begin{align*}
		{\rm Cost}^{\rm update}_{j, \rm Mat-vec}&= {\rm Cost}^{\rm update}_{j, \rm Left\; Env }+{\rm Cost}^{\rm contract}_{j, \rm Right\; Env }+ {\rm Cost}_{j, \rm Contract }\\
		&= \mathcal{O}(n r^3 R).
	\end{align*}
	In this framework therefore, if we denote by $K$ the maximal number of matrix-vector products needed in all one-site DMRG micro-iteration and by $M$ the number of one-site DMRG half-sweeps, we obtain the following computational scaling of the one-site DMRG algorithm:
	\begin{align}\label{eq:scaling_ALS_3}
		{\rm Cost}^{\rm stored\; env}_{\rm ALS, compute}&={\rm Cost}_{\rm precompute} + {\rm Cost}_{\rm sweep}\\ \nonumber
		&=\mathcal{O}(dn r^3 R) + \mathcal{O}(dn r^3 R K M)\\ \nonumber
		&=\mathcal{O}(dn r^3 R K M).
	\end{align}
	The computational cost \eqref{eq:scaling_ALS_3} is supplemented with a storage cost of
	\begin{align}\label{eq:scaling_ALS_4}
		{\rm Cost}^{\rm stored\; env}_{\rm ALS, storage}&={\rm Cost}_{\bold{X}, \rm storage}+ {\rm Cost}_{\bold{A}, \rm storage} + {\rm Cost}_{\rm envs/intermediates, storage}\\ \nonumber
		&=\mathcal{O}(dn r^2)+\mathcal{O}(dn^2R^2)+\mathcal{O}(dr^2R)+\mathcal{O}(nr^2R).
	\end{align}
	Let us remark here that in problems arising from quantum chemistry, we have $n \in \{2, 4\}$ while $R = \mathcal{O}(d^2) < r$ so that the third term above has by far the dominant cost. 
	
	\item Turning now to the case when the parameters $r, R, $ and $d$ are very large and the storage capacity is limited, we see that the $j^{\rm th}$ micro-iteration in a left-to-right half-sweep requires an on-the-fly computation of the right-environment $\bold{H}^{\rm L}_{j+1}$ with cost ${\rm Cost}_{j, \rm Right\; Env }$. In this case, the computational cost \eqref{eq:scaling_ALS_2} cannot be further improved, and if we denote again by $K$ the number of matrix-vector products needed in each one-site DMRG micro-iteration and by $M$ the number of one-site DMRG half-sweeps, we obtain the following computational scaling:
	\begin{align}\label{eq:scaling_ALS_5}
		{\rm Cost}^{\rm construct\; env}_{\rm 1site, compute}&=  {\rm Cost}^{\rm update}_{\rm Env}+{\rm Cost}^{\rm construct}_{\rm Env}+ {\rm Cost}_{\rm Contract}\\ \nonumber
		&= \mathcal{O}(d n r^3 RM) +\mathcal{O}(d^2 n r^3 RM) + \mathcal{O}(dn r^3 RKM)\\ \nonumber
		&=\mathcal{O}\big(d n r^3 RM(K+d)\big),
	\end{align}
	where we have assumed that for each one-site DMRG micro-iteration, the (left or right) environment needs to be constructed only once so that the Krylov iterations required for eigenvalue solver necessitate only tensor contractions with the left and right environments. Compared to Estimate \eqref{eq:scaling_ALS_3} for the first framework, the computational cost in the present case scales with an additional factor $d$, i.e., the order of the underlying tensor space. On the positive side, the increased computational cost \eqref{eq:scaling_ALS_5} is supplemented with a lower storage cost of
	\begin{align}\label{eq:scaling_ALS_6}
		{\rm Cost}^{\rm construct\; env}_{\rm 1site, storage}&={\rm Cost}_{\bold{X}, \rm storage}+ {\rm Cost}_{\bold{A}, \rm storage} + {\rm Cost}_{\rm env/intermediates, storage}\\ \nonumber
		&=\mathcal{O}(dn r^2)+\mathcal{O}(dn^2R^2)+\mathcal{O}(r^2R)+\mathcal{O}(nr^2R).
	\end{align}
\end{itemize}

Several remarks are now in order. First, note that in the above calculations, we have not explicitly taken into account the computational cost of performing the QR decomposition after each one-site DMRG micro-iteration. Since this cost scales as $\mathcal{O}(nr^3)$ per micro-iteration, it is dominated by the cost of the other steps in the one-site DMRG algorithm. Second, we emphasize that the same considerations (with some obvious modifications) apply in the case of the two-site DMRG algorithm. In particular, if all $d$ left/right environments are not being stored due to memory constraints, then we obtain the following computational scaling of the two-site DMRG algorithm:
\begin{align}\label{eq:scaling_ALS_7}
	{\rm Cost}^{\rm construct\; env}_{\rm 2site, compute}&=\mathcal{O}\big(d n r^3 RM(nK+d)\big)\\
	{\rm Cost}^{\rm construct\; env}_{\rm 2site, storage}&=\mathcal{O}(dn r^2)+\mathcal{O}(dn^2R^2)+\mathcal{O}(r^2R)+\mathcal{O}(n^2r^2R). \nonumber
\end{align}
Finally, let us mention that the estimates presented here assume that the the components of the underlying tensor train decompositions are dense. As remarked in Section \ref{sec:3.2}, for problems arising from quantum chemistry and electronic structure calculations, this is not necessarily the case. Essentially, various conserved quantities in the physical system (such as the number of particles and the total spin of the system) lead to a \emph{block-spare} structure in the components of the tensor train decompositions of both $\bold{A}$ and the sought-after minimizer $\bold{X} \in \mathcal{T}_{\bold{r}}$. We refer to \cite{bachmayr2022particle, badreddine2024leveraging} for further details about these sparsity structures. 

We end this subsection with a short discussion on the computational costs of the one-site and two-site DMRG algorithms in a multi-processor setting. Several strategies to parallelize various steps in these algorithms (such as the tensor contractions) have been proposed in the physics literature (see, e.g., \cite{brabec2021massively, zhai2021low, stoudenmire2013real, levy2020distributed, menczer2024parallel}). However, with the exception of the work of White and Stoudenmire \cite{stoudenmire2013real}, these parallelization strategies are all based on so-called pipeline parallelism in the sense that they focus on parallelism within the various steps comprising a micro-iteration. To our knowledge the work of White and Stoudenmire \cite{stoudenmire2013real} is the only attempt at parallelism within the left/right half-sweeps by fundamentally altering the structure of the two-site DMRG algorithm.

\subsection{Computational Scaling of the A2DMRG \Cref{alg:PALS}}\label{sec:3.3.2}~

\vspace{3mm}
\textbf{\small The One-site A2DMRG Algorithm}~

As emphasized in Section \ref{sec:3.1}, each global iteration of the A2DMRG \Cref{alg:PALS} essentially consists of four steps, the first of which is a sequence of LQ decompositions designed to obtain a $j$-orthogonal copy of the initial tensor train decomposition $\bold{U}^{(n), d}$ for each $j \in \{d, d-1, \ldots, 1\}$. The total computational cost of this step is therefore
\begin{align*}
	{\rm Cost}^{\rm 1site}_{\rm Step\;1}= \mathcal{O}(dnr^3). 
\end{align*}

The next step in the A2DMRG \Cref{alg:PALS} consists of performing $d$ one-site DMRG micro-iterations. As discussed in detail in Section \ref{sec:3.3.1} above, each one-site DMRG micro-iteration in the present setting requires the construction of the left and right environments. Note that this is somewhat different from the classical one-site DMRG algorithm in which, during a left-to-right half-sweep for instance, only the right environment needs to be constructed from scratch at each micro-iteration with the left environment being obtained by a cheaper update step. Regardless, assuming that a maximum of $K\in \mathbb{N}$ matrix-vector products are required for each one-site DMRG micro-iterations, the total cost of this step of the one-site A2DMRG algorithm is (recall Estimate \eqref{eq:scaling_ALS_5})
\begin{align*}
	{\rm Cost}^{\rm 1site}_{\rm Step\;2}= \mathcal{O}\big(d n r^3 R(K+d)\big).
\end{align*}

Turning now to the third step, i.e., the second-level minimization in the A2DMRG \Cref{alg:PALS}, we see that we are required to compute the lowest eigenvalue of the symmetric eigenvalue problem \eqref{eq:aux_eig_final_sym}. To do so, we must first compute the singular value decomposition of the mass matrix $\widehat{\bold{S}}\in \mathbb{R}^{d+1\times d+1}$ defined as
\begin{align*}
	\forall i, j, \in \{1, \ldots, d+1\}\colon \qquad   [\widehat{\bold{S}}]_{i,j} = \big\langle \tau(\widetilde{\bold{U}}^{(n+1), {i-1}}), \tau(\widetilde{\bold{U}}^{(n+1), {j-1}}) \big\rangle,
\end{align*}
where $\widetilde{\bold{U}}^{(n+1), \ell}, \ell \in \{1, \ldots, d\}$ is the tensor train decomposition obtained by performing a one-site DMRG micro-iteration to update the $\ell^{\rm th}$ component of the previous global iterate $\widetilde{\bold{U}}^{(n+1), 0}$. It follows that each entry of the mass matrix $\widehat{\bold{S}}$ requires computing an inner product between two tensors possessing TT decompositions and therefore carries a computational cost of $\mathcal{O}(dnr^3)$. Consequently, the total cost of constructing the symmetric matrix $\widehat{\bold{S}}$ and computing its singular value decomposition scales as
\begin{align*}
	{\rm Cost}_{\bold{S}, \rm SVD}= \mathcal{O}\Big(\frac{d(d+1)}{2}dnr^3\Big) + \mathcal{O}(d^3).
\end{align*}
Note that although this cost, a priori, scales cubically in the tensor order $d$, it does \emph{not} involve contractions with the underlying tensor operator $\bold{A}$. Consequently, if the separation ranks of $\bold{A}$ scale (at least) quadratically with $d$ (which is the case for the second-quantized Hamiltonian in quantum chemistry \cite[Section 3.2]{holtz2012alternating}), then the total cost of this step is comparable with the cost of constructing a left or right environment in the course of the DMRG micro-iterations described above.

Next, let us consider the matrix $\widehat{\bold{A}}\in \mathbb{R}^{d+1\times d+1}$ appearing in the eigenvalue problem \eqref{eq:aux_eig_final_sym}, and defined as
\begin{align*}
	\forall i, j, \in \{1, \ldots, d+1\}\colon \qquad [\widehat{\bold{A}}]_{i,j} = \big\langle \tau(\widetilde{\bold{U}}^{(n+1), {i-1}}), \bold{A}\tau(\widetilde{\bold{U}}^{(n+1), {j-1}}) \big\rangle.
\end{align*}

There are now essentially two possibilities:
\begin{itemize}
	\item We can appeal to a direct computation of $\widehat{\bold{A}}$. Following the same arguments presented in Section \ref{sec:3.3.1}, computing each entry of $\widehat{\bold{A}}$ carries a computational cost of $\mathcal{O}(dnr^3R)$ leading to a total matrix construction cost of 
	\begin{align*}
		{\rm Cost}_{\bold{A}}= \mathcal{O}\Big(\frac{d(d+1)}{2}dnRr^3\Big).
	\end{align*}
	We can then obtain lowest eigenvalue of the symmetric eigenvalue problem \eqref{eq:aux_eig_final_sym} using a direct solver with a total cost of
	\begin{align*}
		{\rm Cost}^{\rm 1site}_{\rm Step\;3, direct}= \mathcal{O}\Big(\frac{d(d+1)}{2}dnr^3\Big) + \mathcal{O}\Big(\frac{d(d+1)}{2}dnRr^3\Big) + \mathcal{O}(d^3).
	\end{align*}
	The main computational bottleneck in this approach is, of course, the construction cost of the matrix $\widehat{\bold{A}}$ which scales cubically in the tensor order $d$. If, therefore, $d$ is sufficiently large, then this cost will dominate the cost the DMRG micro-iterations in the second step of the A2DMRG \Cref{alg:PALS}. It is nevertheless important to note that this matrix construction can easily be performed in parallel provided that sufficient independent processors are available. In particular, if we have at our disposal $d(d+1)/2$ independent processors, then each processor needs to compute only a single order $d$ tensor contraction so that the cost per processor scales linearly in $d$.
	
	\item An alternative strategy is to eschew the explicit construction of $\widehat{\bold{A}}$ and make use of a Krylov solver to compute the lowest eigenvalue of the eigenvalue problem \eqref{eq:aux_eig_final_sym}. To do so, it suffices to consider the calculation of matrix vector products of the form
	\begin{align*}
		\mathbb{R}^{p}\ni \bold{z}=(\widehat{\bold{\Sigma}}_{\bold{S}}^+)^{-1/2}(\widehat{\bold{V}}_{\bold{S}}^+)^*\widehat{\bold{A}} \widehat{\bold{V}}_{\bold{S}}^+(\widehat{\bold{\Sigma}}_{\bold{S}}^+)^{-1/2} \bold{y} \qquad \text{for arbitrary vectors } \bold{y} \in \mathbb{R}^{p}.
	\end{align*}
	Here, $p \in \mathbb{N}$ denotes the numerical rank of the mass matrix $\widehat{\bold{S}}$, and the diagonal matrix $\widehat{\bold{\Sigma}}_{\bold{S}}^+ \in \mathbb{R}^{p \times p}$ and left-orthogonal matrix $\widehat{\bold{V}}_{\bold{S}}^+ \in \mathbb{R}^{d+1 \times p}$ are obtained from the singular value decomposition of $\widehat{\bold{S}}$ (see Section \ref{sec:3.1}).
	
	Notice now that for any input vector $\bold{y}$, the $i^{\rm th}$ entry of the output vector $\bold{z}$ is given by
	\begin{align}\nonumber
		\bold{z}_i=&\sum_{j=1}^{d+1} \sum_{k=1}^{d+1} \sum_{\ell=1}^p [\widehat{\bold{\Sigma}}_{\bold{S}}^+]_{i,i}^{-1/2}[(\widehat{\bold{V}}_{\bold{S}}^+)^*]_{i, j}[\widehat{\bold{A}}]_{j, k} [\widehat{\bold{V}}_{\bold{S}}^+]_{k, \ell}[\widehat{\bold{\Sigma}}_{\bold{S}}^+]_{\ell, \ell}^{-1/2} \bold{y}_{\ell}\\ \nonumber
		=&\sum_{j=1}^{d+1} \sum_{k=1}^{d+1} \sum_{\ell=1}^p [\widehat{\bold{\Sigma}}_{\bold{S}}^+]_{i,i}^{-1/2}[(\widehat{\bold{V}}_{\bold{S}}^+)^*]_{i, j}\left\langle \tau(\widetilde{\bold{U}}^{(n+1), {j-1}}), \bold{A}\tau(\widetilde{\bold{U}}^{(n+1), {k-1}})\right\rangle [\widehat{\bold{V}}_{\bold{S}}^+]_{k, \ell}[\widehat{\bold{\Sigma}}_{\bold{S}}^+]_{\ell, \ell}^{-1/2} \bold{y}_{\ell}\\ \nonumber
		=&\left \langle \sum_{j=1}^{d+1}  [\widehat{\bold{\Sigma}}_{\bold{S}}^+]_{i,i}^{-1/2}[(\widehat{\bold{V}}_{\bold{S}}^+)^*]_{i, j} \tau(\widetilde{\bold{U}}^{(n+1), {j-1}}), \bold{A} \sum_{k=1}^{d+1} \sum_{\ell=1}^p  [\widehat{\bold{V}}_{\bold{S}}^+]_{k, \ell}[\widehat{\bold{\Sigma}}_{\bold{S}}^+]_{\ell, \ell}^{-1/2} \bold{y}_{\ell} \tau(\widetilde{\bold{U}}^{(n+1), {k-1}})\right\rangle\\ \label{eq:aux_mat-vec_prod}
		=&\left \langle \sum_{j=1}^{d+1}  {\bold{t}}^{(i)}_j\tau(\widetilde{\bold{U}}^{(n+1), {j-1}}), \bold{A} \sum_{k=1}^{d+1} \bold{w}_k \tau(\widetilde{\bold{U}}^{(n+1), {k-1}})\right\rangle,
	\end{align}
	where we have introduced the vectors $\bold{t}^{(i)}, \bold{w} \in \mathbb{R}^{d+1}$ given by
	\begin{align*}
		\forall j \in \{1, \ldots, d+1\}\colon \hspace{1mm}    \bold{t}^{(i)}_j =[\widehat{\bold{\Sigma}}_{\bold{S}}^+]_{i,i}^{-1/2}[(\widehat{\bold{V}}_{\bold{S}}^+)^*]_{i, j}, \hspace{1mm} \bold{w}_j = \sum_{\ell=1}^p  [\widehat{\bold{V}}_{\bold{S}}^+]_{k, \ell}[\widehat{\bold{\Sigma}}_{\bold{S}}^+]_{\ell, \ell}^{-1/2} \bold{y}_{\ell}. 
	\end{align*}
	
	In other words, the $i^{\rm th}$ entry of the output vector $\bold{z}$ can be expressed in terms of an inner product involving $\bold{A}$ and two different linear combinations of the tensor train decompositions $\{\tau(\widetilde{\bold{U}}^{(n+1), j})\}_{j=0}^d$. It is well known that a linear combination of $d$ tensor train decompositions with TT ranks bounded by $r$ can itself be expressed as a tensor train decomposition with TT ranks bounded by $dr$. We claim that in fact, for any vectors $\bold{t}^{(i)}, \bold{w} \in \mathbb{R}^d$, each of the linear combinations
	\begin{align*}
		\sum_{j=1}^{d+1}  {\bold{t}}_j\tau(\widetilde{\bold{U}}^{(n+1), {j-1}}), \quad \sum_{k=1}^{d+1} \bold{w}_k \tau(\widetilde{\bold{U}}^{(n+1), {k-1}}),
	\end{align*}
	possesses a tensor train decomposition with TT ranks bounded by~$2r$. To see this, recall that, by construction, each $\widetilde{\bold{U}}^{(n+1), j}, ~j\in \{1, \ldots, d\}$ is of the form
	\begin{align*}
		\widetilde{\bold{U}}^{(n+1), j}= \big(\bold{U}_{1}^{(n)}, \ldots \bold{U}_{j-1}^{(n)}, \delta \bold{U}_{j}^{(n+1)}, \bold{V}_{j+1}^{(n)}, \ldots \bold{V}_{d}^{(n)}\big),
	\end{align*}
	where the component $\delta \bold{U}_{j}^{(n+1)} \in \mathbb{R}^{r_{j-1}\times n_j \times r_j}$ has been obtained through the $j^{\rm th}$ one-site DMRG micro-iteration applied to the initial $j$-orthogonal tensor train iterate $\tau({\bold{U}}^{(n), j})$. Since $\tau({\bold{U}}^{(n+1), 0})=\tau({\bold{U}}^{(n), j})$ for all $j \in \{1, \ldots, d\}$, we can, e.g., also take the right-orthogonal tensor train decomposition of $\tau({\bold{U}}^{(n+1), 0})$, i.e., 
	\begin{align*}
		\tau({\bold{U}}^{(n+1), 0})= \tau\big({\delta W_1}^{(n)}, \bold{V}_2^{(n)} \ldots \bold{V}_{d}^{(n)}\big).
	\end{align*}
	It follows from a direct calculation that for all $x_\ell \in \{1, \ldots, n_\ell\}, ~ \ell=1, \ldots, d$ we can write
	\begin{align}\label{eq:reparameter_0}
		\Big(\sum_{j=1}^{d+1} \bold{t}^{(i)}_j\tau(\widetilde{\bold{U}}^{(n+1), {j-1}})\Big)(x_1,\ldots,x_d)&= \bold{W}_1(x_1)\bold{W}_2(x_2)\ldots \bold{W}_d(x_d),
	\end{align}
	where
	\begin{align}\label{eq:reparameter}
		\begin{split}
			\bold{W}_1(x_1)&= \begin{bmatrix}
				\bold{t}_1^{(i)}{\delta W_1}^{(n)}(x_1)+\bold{t}_2^{(i)}\delta \bold{U}^{(n+1)}_1(x_1) &\bold{U}_1^{(n)}(x_1)
			\end{bmatrix} \in \mathbb{R}^{r_0 \times 2r_1},\\[0.5em]
			\bold{W}_\ell(x_\ell)&=\begin{bmatrix}
				\bold{V}^{(n+1)}_\ell(x_\ell) & 0\\
				\bold{t}_{\ell+1}^{(i)} \delta \bold{U}^{(n+1)}_\ell(x_\ell) & \bold{U}^{(n+1)}_\ell(x_\ell)\\
			\end{bmatrix} \in \mathbb{R}^{2r_{\ell-1} \times 2r_\ell}\hspace{1mm} \forall \ell \in \{2, \ldots, d-1\},\\
			\bold{W}_d(x_d)&=\begin{bmatrix}
				\bold{V}^{(n+1)}_d(x_d)\\ \bold{t}_{d+1}^{(i)} \delta \bold{U}_d^{(n+1)}(x_d)
			\end{bmatrix} \in \mathbb{R}^{2r_{d-1}\times r_d}.
		\end{split}
	\end{align}
	
	Thus, the tensor $\sum_{j=1}^{d+1}\bold{t}^{(i)}_j\tau(\widetilde{\bold{U}}^{(n+1), {j-1}})$ has a tensor train decomposition $\big(\bold{W}_1, \ldots \bold{W}_d\big)$ with components defined through Equation \eqref{eq:reparameter}.  A similar calculation of course also works when the vector $\bold{t}^{(i)}$ is replaced with~$\bold{w}$.
	
	Returning now to Equation \eqref{eq:aux_mat-vec_prod}, we see that the $i^{\rm th}$ entry $\bold{z}_i$ of our sought-after matrix-vector product can be obtained by performing a tensor contraction of the type described in Section \ref{sec:3.3.1} above. Assuming therefore, that $K'\in \mathbb{N}$ Krylov iterations are required to solve the eigenvalue problem \eqref{eq:aux_eig_final_sym}, we arrive at the following total computational cost of this third step of the one-site A2DMRG algorithm \ref{alg:PALS}:
	\begin{align*}
		{\rm Cost}_{\rm Step\;3, Krylov}^{\rm 1site}=& \mathcal{O}\big(d^2 nr^3RK'\big) +  {\rm Cost}_{\bold{S}, \rm SVD}\\
		=& \mathcal{O}\big(d^2 nr^3RK'\big)+ \mathcal{O}(d^3nr^3 +d^3).
	\end{align*}
	In other words, the use of a Krylov solver allows the reduction of the $d^3$ scaling cost in the construction of $\widehat{\bold{A}}$ to a $d^2$ scaling cost. This cost can be further reduced to linear scaling in $d$ per processor provided that $d$ independent processors are available.
\end{itemize}

It now remains to consider the computational cost of the fourth and final step in the one-site A2DMRG algorithm, namely the TT rounding step. To do so, we simply appeal once again to Equations \eqref{eq:reparameter_0}-\eqref{eq:reparameter} to deduce that any linear combination $\sum_{j=1}^d  {\bold{d}}_j\tau(\widetilde{\bold{Y}}_j^{(n+1)})$ of tensor possesses a tensor train decompositions with TT ranks bounded by $2r$. The TT rounding algorithm of Oseledets \cite{oseledets2011tensor} therefore requires a total computational cost
\begin{align*}
	{\rm Cost}_{\rm Step\;4}^{\rm 1site}=\mathcal{O}(dnr^3),
\end{align*}
and we thus conclude that each global iteration of the  A2DMRG Algorithm \ref{alg:PALS}-- assuming the use of a Krylov solver in the third step-- has computational cost
\begin{align}\nonumber
	{\rm Cost}_{\rm 1site, compute}^{\rm A2DRMG}&={\rm Cost}^{\rm 1site}_{\rm Step\;1}+ {\rm Cost}^{\rm 1site}_{\rm Step\;2} + {\rm Cost}^{\rm 1site}_{\rm Step\;3, Krylov} + {\rm Cost}^{\rm 1site}_{\rm Step\;4}\\  \label{eq:PALS_cost_1}
	&= \mathcal{O}\big(d n r^3 R(K+d)\big)+ \mathcal{O}\big(d^2 nr^3RK'\big)+\mathcal{O}(d^3nr^3).
\end{align}

Before proceeding to a study of the two-site A2DMRG \Cref{alg:PALS}, let us conclude this discussion by emphasizing that although the computational costs of the one-site DMRG and one-site A2DMRG algorithms are ostensibly similar, A2DMRG is well-suited for parallelization. Indeed, assuming that we have at hand $d$  processors (or $d(d+1)/2$ processors if the matrices in the second-level minimization are constructed explicitly), we can distribute the workload of the computationally intensive portion of the A2DMRG Algorithm \ref{alg:PALS}, namely, steps two and three, to these processors and arrive at the following cost per processor for a single iteration:
\begin{align}
	{\rm Cost}_{\rm 1site, compute}^{\rm A2DRMG, pp}=  \mathcal{O}\big(n r^3 R(K+d)\big)+ \mathcal{O}\big(d nr^3RK'\big)+\mathcal{O}(d^2nr^3).
\end{align}
In the specific case of quantum chemistry applications, i.e., when the energy functional $\mathcal{J}$ to be minimized is of the form \eqref{eq:energy_eig}
\begin{align*}
	\forall \bold{X}\in  U=\mathbb{R}^{n_1 \times n_2 \times \ldots n_d}\setminus \{0\}\colon \qquad        \mathcal{J}(\bold{X})= \frac{\langle \bold{X}, \bold{H}\bold{U}\rangle}{\Vert \bold{X}\Vert^2},
\end{align*}
with $\bold{H}\colon \mathbb{R}^{n_1 \times n_2 \times \ldots n_d} \rightarrow \mathbb{R}^{n_1 \times n_2 \times \ldots n_d}$ denoting a second-quantized molecular Hamiltonian \cite[Section 3.2]{holtz2012alternating}, it is well-known that $R = \mathcal{O}(d^2)$. Consequently, the cost per processor for a single iteration scales in this case as
\begin{align}
	{\rm Cost}_{\rm 1site, compute}^{\rm A2DRMG, pp}=  \mathcal{O}\big(n r^3 R(K+d)\big)+ \mathcal{O}\big(d nr^3RK'\big).
\end{align}

	
	
	

\noindent \textbf{\small The Two-site A2DMRG Algorithm}~

Let us now consider the computational cost of the two-site A2DMRG algorithm. Mimicking the same arguments as those used to study the one-site DMRG case, we see that the computational cost of the first two steps of Algorithm \eqref{alg:PALS} are given by
\begin{align*}
	{\rm Cost}^{\rm 2site}_{\rm Step\;1}  +    {\rm Cost}^{\rm 2site}_{\rm Step\;2}= \mathcal{O}\big(d n r^3\big)+\mathcal{O}\big(d n r^3 R(nK+d)\big).
\end{align*}

The third step in the two-site A2DMRG algorithm requires the computation of the lowest eigenvalue of the symmetric eigenvalue problem \eqref{eq:aux_eig_final_sym}. As in the one-site DMRG algorithm case, this can either be done using an explicit construction of the underlying matrices $\widehat{\bold{S}}\in \mathbb{R}^{d\times d}$ and $\widehat{\bold{A}}$ or by relying on a Krylov solver. To evaluate the cost of the latter approach, it suffices to check that a factorization of the form \eqref{eq:reparameter_0}-\eqref{eq:reparameter} can also be obtained for linear combinations of tensor train decompositions obtained from the second step of the two-site A2DMRG Algorithm~\ref{alg:PALS}. We claim that this is indeed the case. Indeed, consider a linear combination of tensor of the form
\begin{align*}
	\sum_{j=1}^{d}  {\bold{y}}_j\tau(\widetilde{\bold{U}}^{(n+1), {j-1}}),
\end{align*}
where each ${\bold{y}}_j \in \mathbb{R}$, and each tensor train decomposition $\widetilde{\bold{U}}^{(n+1), j}, ~j\in \{1, \ldots, d-1\}$ is obtained as the solution to the $j^{\rm th}$ two-site DMRG micro-step in accordance with step two of the A2DMRG \Cref{alg:PALS}. By construction, each TT decomposition $\widetilde{\bold{U}}^{(n+1), j}, ~j \in \{1, \ldots, d-1\}$ takes the form
\begin{align*}
	\widetilde{\bold{U}}^{(n+1), j}= \big(\bold{U}_{1}^{(n)}, \ldots \bold{U}_{j-1}^{(n)}, \delta \bold{U}_{j}^{(n+1)}, \delta \bold{U}_{j+1}^{(n+1)}, \bold{V}_{j+2}^{(n)}, \ldots \bold{V}_{d}^{(n)}\big)
\end{align*}
where the components $\delta \bold{U}_{j}^{(n+1)} \in \mathbb{R}^{r_{j-1}\times n_j \times \widetilde{r}_j} $ and $ \delta \bold{U}_{j+1}^{(n+1)} \in \mathbb{R}^{\widetilde{r}_{j}\times n_j \times r_{j+1}}$ have been obtained through the $j^{\rm th}$ two-site DMRG micro-iteration applied to the initial $j$-orthogonal tensor train iterate $\bold{U}^{(n), j}$ followed by a truncated SVD step. Since $\tau({\bold{U}}^{(n+1), 0})=\tau({\bold{U}}^{(n), j})$ for all $j \in \{1, \ldots, d-1\}$, we can once again take the right-orthogonal tensor train decomposition of $\tau({\bold{U}}^{(n+1), 0})$, i.e., 
\begin{align*}
	\tau({\bold{U}}^{(n+1), 0})= \tau\big({\delta W_1}^{(n)}, \bold{V}_2^{(n)} \ldots \bold{V}_{d}^{(n)}\big).
\end{align*}
It now follows from a direct calculation that for all $x_\ell \in \{1, \ldots, n_\ell\}, ~ \ell=1, \ldots, d$ we can write
\begin{align}\label{eq:reparameter_0_MALS}
	\Big(\sum_{j=1}^d \bold{y}_j\tau(\widetilde{\bold{U}}^{(n+1), {j-1}})\Big)(x_1,\ldots,x_d)&= \bold{W}_1(x_1, x_2)\bold{W}_2(x_2, x_3)\ldots \bold{W}_{d-1}(x_{d-1}, x_d),
\end{align}
where
\begin{align}\label{eq:reparameter_MALS}
	\begin{split}
		\bold{W}_1(x_1, x_2)&= \begin{bmatrix}
			\bold{y}_2\delta \bold{U}^{(n+1)}_1(x_1)\delta \bold{U}^{(n+1)}_2(x_2) + \bold{y}_1\delta \bold{W}^{(n)}_1(x_1)\bold{V}^{(n)}_2(x_2) &\bold{U}_1^{(n)}(x_1)
		\end{bmatrix} \in \mathbb{R}^{r_0 \times (r_2+r_1)},\\[1em]
		\bold{W}_\ell(x_\ell, x_{\ell+1})&=\begin{bmatrix}
			\bold{V}^{(n+1)}_{\ell+1}(x_{\ell+1}) & 0\\[0.4em]
			\bold{y}_{\ell+1} \delta \bold{U}^{(n+1)}_\ell(x_\ell)\delta \bold{U}^{(n+1)}_{\ell+1}(x_{\ell+1}) & \bold{U}^{(n+1)}_\ell(x_\ell)\\
		\end{bmatrix} \in \mathbb{R}^{(r_{\ell}+r_{\ell-1})\times (r_{\ell+1}+r_{\ell})},\\ \noindent
		&\hspace{7.5cm}\forall \ell \in \{2, \ldots, d-2\},\\[1em]
		\bold{W}_{d-1}(x_{d-1}, x_d)&=\begin{bmatrix}
			\bold{V}^{(n+1)}_d(x_d)\\[0.4em] \bold{y}_d \delta \bold{U}_{d-1}^{(n+1)}(x_{d-1})\delta \bold{U}_d^{(n+1)}(x_d)
		\end{bmatrix} \in \mathbb{R}^{(r_{d-1}+ r_{d-2})\times r_d}.
	\end{split}
\end{align}

Although Equation \eqref{eq:reparameter_0_MALS} does not correspond to a tensor train decomposition of the tensor $\sum_{j=1}^d \bold{y}_j\tau(\widetilde{\bold{U}}^{(n+1), {j-1}})$, the structure of this factorization closely resembles a tensor train decomposition, and we can therefore use the usual TT-type tensor contractions to efficiently compute inner products involving tensors of the form $\sum_{j=1}^d \bold{y}_j\tau(\widetilde{\bold{U}}^{(n+1), {j-1}})$. In particular, the computational cost of calculating an inner-product of the form (c.f., Equation \eqref{eq:aux_mat-vec_prod})
\begin{align*}
	\left \langle \sum_{j=1}^d  {\bold{t}}^{(i)}_j\widetilde{\bold{U}}^{(n+1), {j-1}}, \bold{A} \sum_{k=1}^d \bold{w}_k \widetilde{\bold{U}}^{(n+1), {k-1}}\right\rangle,
\end{align*}
for arbitrary vectors $\bold{t}^{(i)}, \bold{w} \in \mathbb{R}^d$ scales as $\mathcal{O}(d^2n^3r^3R + dn^2\widetilde{r}r^2)$, where we have introduced $\widetilde{r}= \max\{\widetilde{r}_j\}_{j=1}^{d-1}$. Assuming therefore, that $K'\in \mathbb{N}$ Krylov iterations are required to solve the eigenvalue problem \eqref{eq:aux_eig_final_sym}, we arrive at the following total computational cost of this third step of the two-site A2DMRG Algorithm \ref{alg:PALS}:
\begin{align*}
	{\rm Cost}_{\rm Step\;3, Krylov}^{\rm 2site}=& \mathcal{O}\big(d^2 n^3r^3RK'+ dn^2\widetilde{r}r^2\big) +  {\rm Cost}_{\bold{S}, \rm SVD}\\
	=& \mathcal{O}\big(d^2 n^3r^3RK'+dn^2\widetilde{r}r^2K'\big)+ \mathcal{O}(d^3n\widetilde{r}^3 +d^3).
\end{align*}

It remains to consider the computational cost of the fourth and final step in the two-site A2DMRG algorithm, namely the TT rounding step. In contrast to the one-site DMRG case, an additional complication now arises, namely, that the factorization \eqref{eq:reparameter_0_MALS}-\eqref{eq:reparameter_MALS} of a linear combination of tensor train decompositions does \underline{not} itself represent a tensor train decomposition of ranks bounded by $2\bold{r}$. Consequently, we can no longer appeal directly to the TT rounding algorithm of Oseledets \cite{oseledets2011tensor}. Instead, we propose to rephrase the TT rounding step as a best-approximation problem on the manifold $\mathcal{T}_{\widetilde{r}}\subset \mathbb{R}^{n_1\times \ldots \times n_d}$ with ranks $(\widetilde{r}_1, \ldots, \widetilde{r}_{d-1})$. In other words, given the solution $\sum_{j=0}^d  {\bold{c}}_j^*\tau(\widetilde{\bold{U}}^{(n+1), {j}})$ to the eigenvalue problem \eqref{eq:aux_eig_final_sym}, we consider the minimization problem
\begin{align}\label{eq:MALS_rounding}
	\underset{\bold{X}\in \mathcal{T}_{\widetilde{\bold{r}}}}{\text{argmin}} \frac{1}{2}\Vert \sum_{j=0}^d  {\bold{c}}_j^*\tau(\widetilde{\bold{U}}^{(n+1), j})-\bold{X}\Vert^2= \underset{\bold{X}\in \mathcal{T}_{\widetilde{\bold{r}}}}{\text{argmin}}  \frac{1}{2} \langle \bold{X}, \bold{X}\rangle - \Big\langle \sum_{j=0}^d  {\bold{c}}_j^*\tau(\widetilde{\bold{U}}^{(n+1), j}), \bold{X}\Big\rangle.
\end{align}
Equation \eqref{eq:MALS_rounding} can, for instance, be solved using the one-site A2DMRG \Cref{alg:PALS} together with the factorization \eqref{eq:reparameter_0_MALS}-\eqref{eq:reparameter_MALS}. Recalling the computational cost of each iteration of the one-site A2DMRG algorithm \Cref{alg:PALS} given by Equation \eqref{eq:PALS_cost_1}, and denoting by $M'\in \mathbb{N}$ the total number of global iterations required for convergence up to a tolerance, this step scales as
\begin{align}\label{eq:new_amend}
	{\rm Cost}_{\rm Step\;4}^{\rm 2site}=\mathcal{O}\big(d n \widetilde{r}^3 (K_{\rm 1Site}+d)M'\big)+\mathcal{O}\big(d^2 n \widetilde{r}^3 K_{\rm 1Site}'M'\big)+\mathcal{O}(M' d^3),
\end{align}
where $K_{\rm 1Site}, K_{\rm 1Site}'$ denote the number of Krylov solver iterations required for the one-site DMRG micro-iterations and the solution to the second-level minimization problem \eqref{eq:aux_eig_final_sym} respectively. We emphasize here that the application of the one-site A2DMRG algorithm to the best-approximation problem \eqref{eq:MALS_rounding} requires the solution of a linear system of equations, both for the micro-iterations and the second-level minimization. Consequently, the computational cost ${\rm Cost}_{\rm Step\;4}^{\rm 2site}$ given by Equation \eqref{eq:new_amend} differs slightly from the computational cost of one-site A2DMRG calculated earlier (see Equation \eqref{eq:PALS_cost_1}).

We thus conclude that each global iteration of the A2DMRG Algorithm \ref{alg:PALS} has computational cost
\begin{align}\nonumber
	{\rm Cost}_{\rm 2site, compute}^{\rm A2DRMG}&={\rm Cost}^{\rm 2site}_{\rm Step\;1}+ {\rm Cost}^{\rm 2site}_{\rm Step\;2} + {\rm Cost}^{\rm 2site}_{\rm Step\;3, Krylov} + {\rm Cost}^{\rm 2site}_{\rm Step\;4}\\ \nonumber
	&=\mathcal{O}\big(d n r^3 R(Kn+d)\big) + \mathcal{O}\big(d^2 n^3r^3RK'+dn^2\widetilde{r}r^2K'\big)+\mathcal{O}(d^3n\widetilde{r}^3)\\
	&+\mathcal{O}\big(d n \widetilde{r}^3 (K_{\rm 1Site}+d)M'\big)+\mathcal{O}\big(d^2 n \widetilde{r}^3 K_{\rm 1Site}'M'\big). \label{eq:PALS_cost_2}
\end{align}

We conclude this section by emphasizing that-- similar to the one-site DMRG case-- the computational costs of the two-site DMRG and two-site A2DMRG algorithms have similar scaling. However, the two-site A2DMRG \Cref{alg:PALS} is well-suited for parallelization. Indeed, assuming that we have at hand $d$ processors (or $d(d+1)/2$ processors if the matrices in the second-level minimization are constructed explicitly), the cost per processor for a single iteration of the two-site A2DMRG \Cref{alg:PALS} is given by:
\begin{align*}
	{\rm Cost}_{\rm 2site, compute}^{\rm A2DRMG, pp}&=  \mathcal{O}\big(n r^3 R(Kn+d)\big) + \mathcal{O}\big(d n^3r^3RK'+n^2\widetilde{r}r^2K'\big)+\mathcal{O}(d^2n\widetilde{r}^3)\\
	&+\mathcal{O}\big(n \widetilde{r}^3 (K_{\rm 1Site}+d)M'\big)+\mathcal{O}\big(d n \widetilde{r}^3 K_{\rm 1Site}'M'\big).
\end{align*}
Let us remark that throughout our testing, $K_{\rm1site}, K_{\rm1site}' \approx \mathcal{O}(1)$ so that the last two terms are dominated in cost by the first two terms.

\end{document}